\documentclass{daj}

\usepackage{amsfonts}
\usepackage{amsmath,amsthm}
\usepackage{amssymb}
\usepackage{indentfirst}
\usepackage[matrix,arrow,curve]{xy}
\usepackage[export]{adjustbox}
\usepackage{subcaption}
\usepackage{wrapfig}
\usepackage{thmtools}
\usepackage{thm-restate}
\usepackage[english]{babel}

\dajAUTHORdetails{%
  title = {Convex geometry and the Erd{\H o}s--Ginzburg--Ziv problem},
  author = {Dmitrii Zakharov},
  plaintextauthor = {Dmitrii Zakharov},
  plaintexttitle = {Convex geometry and the Erdos-Ginzburg-Ziv problem},
  runningtitle = {Convex geometry and the EGZ problem},
  runningauthor = {Dmitrii Zakharov},
  keywords = {Erdos-Ginzburg-Ziv problem, zero-sum problems, finite fields, convex geometry},
}

\dajEDITORdetails{
   year={2026},
   number={3},
   received={29 January 2021},
   revised={23 March 2023},
   published={21 July 2026},
   doi={10.19086/da.165216},
 }

\renewcommand{\le}{\leqslant}
\renewcommand{\ge}{\geqslant}

\newcommand{\R}{\mathbb R}
\newcommand{\s}{\mathfrak s}
\newcommand{\w}{\mathfrak w}

\newcommand{\Q}{\mathbb Q}
\newcommand{\C}{{\mathbb C}}
\newcommand{\Z}{\mathbb Z}
\newcommand{\A}{\mathbb A}
\newcommand{\N}{\mathbb N}
\newcommand{\FF}{\mathbb F}

\newcommand{\conv}{{\rm conv}\,}
\newcommand{\codim}{{\rm codim}\,}
\newcommand{\F}{\mathcal F}
\newcommand{\D}{\mathcal D}
\newcommand{\J}{\mathcal J}

\renewcommand{\P}{\mathcal P}

\newcommand{\wc}{\mbox{\rm w-conv}}
\renewcommand{\c}{\mbox{\rm conv}}

\newtheorem{theorem}{Theorem}[section]
\newtheorem{lemma}[theorem]{Lemma}
\newtheorem{claim}[theorem]{Claim}
\newtheorem{cor}[theorem]{Corollary}
\newtheorem{prop}[theorem]{Proposition}
\newtheorem{obs}[theorem]{Observation}
\newtheorem{conj}[theorem]{Conjecture}

\theoremstyle{definition}
\newtheorem{defi}[theorem]{Definition}

\theoremstyle{remark}
\newtheorem*{remark}{Remark}

\theoremstyle{definition}
\newtheorem{example}[theorem]{Example}

\theoremstyle{definition}

\begin{document}

\begin{frontmatter}[classification=text]
\title{Convex geometry and the Erd{\H o}s--Ginzburg--Ziv problem}
\author[dz]{Dmitrii Zakharov}
\date{}

\begin{abstract}
Denote by ${\mathfrak s}({\mathbb F}_p^d)$ the Erd{\H o}s--Ginzburg--Ziv constant of $\FF_p^d$, that is, the minimum $s$ such that every sequence of $s$ vectors in ${\mathbb F}_p^d$ contains $p$ vectors whose sum is zero. Let ${\mathfrak w}({\mathbb F}_p^d)$ be the maximum size of a sequence of vectors $v_1, \ldots, v_s \in {\mathbb F}_p^d$ such that, for all integers $\alpha_1, \ldots, \alpha_s \ge 0$ with sum $p$, we have $\alpha_1 v_1 + \ldots + \alpha_s v_s \neq 0$ unless $\alpha_i = p$ for some $i$. 

In 1995, Alon--Dubiner proved that $\s(\FF_p^d)$ grows linearly in $p$ when $d$ is fixed. In this work, we determine the constant of linearity:
for fixed $d$ and growing $p$, we show that ${\mathfrak s}({\mathbb F}_p^d) = (1+o(1)) {\mathfrak w}({\mathbb F}_p^d) p$. Furthermore, for every prime $p$ and every $d$, we show that ${\mathfrak w}({\mathbb F}_p^d) \le {2d-1 \choose d}+1$. In particular,  ${\mathfrak s}({\mathbb F}_p^d) \le 4^d p$ for all sufficiently large $p$ and fixed $d$.
\end{abstract}
\end{frontmatter}

\section{Introduction}\label{sec1}

\subsection{History and new upper bound} 

In 1961, Erd{\H o}s, Ginzburg and Ziv \cite{EGZ} showed that, among every collection of $2n-1$ integers, one can always select exactly $n$ whose sum is divisible by $n$. Harborth \cite{Har} considered a higher-dimensional generalization of this problem: for given natural numbers $n$, $d$, what is the minimum number $s$ such that, among every collection of $s$ points in the integer lattice $\Z^d$, there are $n$ points whose centroid is also a lattice point?
Equivalently, after reducing points of the lattice $\Z^d$ modulo $n$, the quantity $s$ is the maximum size of a multiset of points in $\Z_n^d$ such that the sum of every $n$ of them is not congruent to 0 modulo $n$. In light of the latter interpretation, the number $s$ is denoted by $\s(\Z_n^d)$ and called the {\it Erd{\H o}s--Ginzburg--Ziv constant} of the group $\Z_n^d$. Note that points are allowed to coincide in this definition.
The problem of determining $\s(\Z_n^d)$ for various $n$ and $d$ has received considerable attention, but the precise value of $\s(\Z_n^d)$ is still unknown for the majority of parameters $(n, d)$. One can also define the Erd{\H o}s--Ginzburg--Ziv constant of an arbitrary finite abelian group $G$; see \cite{GG} for details and various generalizations.

Confirming a conjecture of Kemnitz \cite{Kem}, Reiher \cite{Rei} showed that $\s(\Z_n^2) = 4n-3$ for every $n \ge 2$. In \cite{AD}, Alon and Dubiner showed that for every $n$ and $d$ we have
\begin{equation}\label{adin}
    \s(\Z_n^d) \le (C d \log d)^d n
\end{equation}
for some absolute constant $C > 0$. In particular, if we fix $d$ and let $n \rightarrow \infty$, then $\s(\Z_n^d)$ grows linearly with $n$. On the other hand, it is not hard to see that $\s(\Z_n^d) \ge 2^d (n-1)+1$. Indeed, take the vertices of the boolean cube $\{0, 1\}^d$, with each vertex taken with multiplicity $n-1$. This multiset has no $n$ elements that sum up to 0 in $\Z_n^d$. 
The best known lower bound on $\s(\Z_n^d)$ is due to Edel \cite{Ed}:
\begin{equation}\label{lower}
    \s(\Z_n^d) \ge 96^{\lfloor d/6\rfloor} (n-1)+1 \approx 2.139^d n,
\end{equation}
which holds for all odd $n$. The corresponding set of points generalizes the boolean cube construction. Namely, it is a Cartesian product of $\lfloor d/6\rfloor$ copies of a certain explicitly constructed set $A \subset \Z^6$ of cardinality $96$, with each point taken with multiplicity $n-1$. Note that the condition that $n$ be odd is necessary: for $n = 2^k$, we have \cite{Har} that $\s(\Z_n^d) = 2^d(n-1)+1$ holds for all $d$.

The case when $n = p$ is a prime number is of particular interest. On the one hand, the vector space structure on $\Z_p^d \cong \FF_p^d$ significantly simplifies the analysis, but the problem is still highly non-trivial. On the other hand, as was observed for $d=1$ in \cite{EGZ}, one can deduce upper bounds on $\s(\Z_n^d)$ from upper bounds on $\s(\FF_p^d)$ by using a simple induction on the prime decomposition of $n$. 
This paper focuses on the case in which the dimension $d$ is fixed and $n= p$ is a very large prime number. We remark that the complementary case, i.e. the case in which the prime $p$ is fixed and the dimension $d$ is large, is also of great interest and, in a way, even more intriguing. 
The current best bounds are $\s(\FF_3^d) \le 2.756^d$ for $p=3$, proved by Ellenberg--Gijswijt in a breakthrough paper \cite{ElG}, and $\s(\FF_p^d) \le C_p (2 \sqrt{p})^d$ for $p \ge 5$, due to Sauermann \cite{Sau}. Note that the best known lower bound in this regime is also (\ref{lower}), which creates a significant gap between the bases of the exponents.
We refer to \cite{Sau} and references therein for the history and state of the art in this question. After the release of this paper, Sauermann and the author \cite{SauZ} showed that, for fixed $p$ and large $d$, we have $\s(\FF_p^d) \le D_{p, \varepsilon} (C_\varepsilon p^\varepsilon)^d$ for every fixed $\varepsilon>0$.

The main result of the present paper is an improvement of the Alon--Dubiner bound (\ref{adin}) for fixed $d$ and sufficiently large primes $p$.

\begin{theorem}\label{thcr}
Let $d \ge 1$ and $p > p_0(d)$ be a sufficiently large prime number. Then we have 
\begin{equation}\label{crude}
    \s(\FF_p^d) \le 4^d p.
\end{equation}
More generally, if all prime divisors of $n > 1$ are larger than $p_0(d)$, then we have $\s(\Z_n^d) \le 4^d n$. 
\end{theorem}

Unfortunately, the condition that $p > p_0$  is necessary for our arguments and cannot be removed. Note that by a classical argument from \cite{EGZ}, the bound for composite $n$ in Theorem \ref{thcr} (essentially) follows from the corresponding bound for primes.

\paragraph{Multiplicity $p-1$ sets.}
As we discussed above, taking each vertex of the boolean cube (or a certain more general set) with multiplicity $p-1$ leads to a lower bound construction for $\s(\FF_p^d)$. It is natural to ask what is the best possible lower bound construction of this form?
For $d \ge 1$ and a prime $p$, define $\w(\FF_p^d)$ to be the maximum number $s$ for which there are vectors $v_1, \ldots, v_s \in \FF_p^d$ with the following property: for all non-negative integers $\alpha_1, \ldots, \alpha_s$ with sum $p$, we have $\alpha_1 v_1 + \ldots + \alpha_s v_s \equiv 0 \pmod p$ if and only if $\alpha_i = p$ for some $i$. For brevity, let us call every set of vectors $X = \{v_1, \ldots, v_s\}$ satisfying this property {\em $p$-hollow} (a justification for this name will become clearer later).

Note that if $X$ is $p$-hollow, then taking each element of $X$ with multiplicity $p-1$ results in a multiset not containing $p$ vectors with zero sum. This implies that for all $p$ and $d$ we have
\begin{equation}\label{ggc}
    \s(\FF_p^d) \ge \w(\FF_p^d)(p-1)+1.
\end{equation}
It is easy to see that the Cartesian product of two $p$-hollow sets is again $p$-hollow, so every lower bound on $\w(\FF_p^{d_0})$ for some fixed $d_0$ extends to a lower bound for all $d \ge d_0$ by the product construction. In fact, all known lower bounds on $\s(\FF_p^d)$ follow from this observation combined with (\ref{ggc}), and in particular (\ref{lower}) follows from $\w(\FF_p^6) \ge 96$ for all $p > 2$.
In \cite{GG}, Gao--Geroldinger conjectured that equality holds in (\ref{ggc}). We confirm their conjecture asymptotically as $p \rightarrow \infty$. 

\begin{theorem}\label{main}
For every fixed $d \ge 1$ and $p \rightarrow \infty$, we have $\s(\FF_p^d) =  \w(\FF_p^d) p + o(p)$.
\end{theorem}

Using the slice rank method, Naslund \cite{N} showed that $\w(\FF_p^d) \le 4^d-1$. So it follows that Theorem \ref{thcr} is in fact a consequence of Theorem \ref{main}. We have the following slight improvement of the slice rank bound:
\begin{prop}\label{thw}
For every $d \ge 1$ and every prime $p$, we have $\w(\FF_p^d) \le {2d-1 \choose d} + 1$.
\end{prop}
Note that $\w(\FF_p^1) = 2 = {1 \choose 1}+1$ and $\w(\FF_p^2) = 4 = {3 \choose 2} + 1$ so the bound in Proposition \ref{thw} is achieved for $d = 1, 2$. 
On the other hand, for $d=3$ it can be shown that $\w(\FF_p^3) = 9$ for large $p$ while Proposition \ref{thw} only gives an upper bound of 11.

{\em Notation.} We use the asymptotic notation $A \gg B$ to denote that $A \ge c B$ for some constant $c > 0$, possibly depending on other parameters. The set of natural numbers $\N$ is the set $\{0, 1, \ldots\}$. A multiset $X \subset A$ of some set $A$ is an unordered sequence of elements of $A$, possibly with repetitions. Two elements of $X$ are said to be distinct if they are on different positions in the sequence (even though they may coincide as elements of $A$).

\subsection{Connection to convex geometry} 

The main new ingredient in the proof of Theorem \ref{main} is a certain connection of the Erd{\H o}s--Ginzburg--Ziv problem to convex geometry. Recall that the original formulation of the question by Harborth was in terms of centroids of $n$ points in $\Z^d$. Thus, it is natural to expect that tools from convex geometry could be useful for tackling the problem. On the other hand, there does not seem to be a direct way to employ this idea. To the author's knowledge, convex geometry has not been used in the study of the Erd{\H o}s--Ginzburg--Ziv problem or related zero-sum problems before.

Throughout this paper, a polytope $P \subset \Q^d$ is the convex hull of a finite set of points in $\Q^d$. A lattice $\Lambda \subset \Q^d$ is an affine image of the set $\Z^r \subset \Q^r$ for some $r \le d$. We define a notion of integer points of polytopes in the following way.

\begin{defi}[Integer point]\label{intpt}
Let $P \subset \Q^d$ be a polytope and let $q \in P$. Let $\Gamma \subset P$ be the minimal face of $P$ containing the point $q$ and let $\Lambda$ be the minimal lattice containing all vertices of $\Gamma$. We say that $q$ is an {\it integer point} of $P$ if $q \in \Lambda$.
\end{defi}

Let us say a couple of words on why this notion is natural. 
If we have a polytope $P \subset \Q^d$, one might say that a point $q \in P$ is an integer point if simply $q \in \Z^d$. However, this notion depends on the choice of the integer lattice $\Z^d$ and therefore is not an `intrinsic' property of $P$ and $q$. To fix this, we could modify the definition as follows: let $\Lambda$ be the lattice spanned by the vertices of $P$ and say that a point $q \in P$ is integer if $q \in \Lambda$. This definition clearly does not depend on the lattice $\Z^n$ and is closer to what we want.
On the other hand, this notion has a problem. If $\Gamma \subset P$ is a face of $P$ and $q \in \Gamma$ is a point, then the properties of $q$ being integer with respect to $P$ and with respect to $\Gamma$ are not the same. It is easy to construct examples where $q$ is an integer point of $P$ but not of $\Gamma$. Thus, this notion of integer points is not invariant under passing to a face of $P$. To fix this, we introduce an additional step: we choose the minimal face $\Gamma$ containing $q$, define the lattice $\Lambda$ spanned by the vertices of $\Gamma$, and say that $q$ is an integer point of $P$ if $q \in \Lambda$. This definition is invariant both under changes of basis of $\Q^d$ and under passing to a face.

We say that a polytope $P \subset \Q^d$ is a {\it hollow} polytope if $P$ has no integer points besides its vertices. For $d\ge 1$, let $L(d)$ be the maximum number of vertices in a hollow polytope $P \subset \Q^d$. It turns out that vertices of a hollow polytope precisely correspond to $p$-hollow sets modulo almost all primes $p$.

\begin{prop}\label{wl}
Let $P \subset \Q^d$ be a hollow polytope and suppose that the set of vertices $X$ of $P$ is a subset in $\Z^d$. Then for all but a finite list of primes $p$, the reduction of $X$ modulo $p$ is a $p$-hollow set.
In particular, for $d \ge 1$ and all sufficiently large primes $p \ge p_0(d)$ we have $\w(\FF_p^d) \ge L(d)$. 
\end{prop}

Note that the list of forbidden primes can be written explicitly in terms of $P$, see Section \ref{theasy1} for details and the proof. 

As a matter of fact, all known lower bound constructions for $p$-hollow sets come from constructions of hollow polytopes, even though the notion of hollow polytopes has not been given explicitly in the zero sum set literature before.
In particular, Elsholtz \cite{El1} showed that $L(3) \ge 9$, Edel \cite{Ed} and Elsholtz \cite{El2} showed that $L(4) \ge 20$, and in \cite{Ed2} Edel showed that $L(5) \ge 42$, $L(6) \ge 96$, $L(7) \ge 196$. Note that the lower bound $L(6) \ge 96$ and a product construction give the best known asymptotic lower bounds on $\w(\Z_n^d)$ and $\s(\Z_n^d)$ cited previously.

In light of this, it seems reasonable to expect that the converse of Proposition \ref{wl} should also be true.

\begin{conj}\label{conj}
For $d \ge 1$ and all sufficiently large primes $p$ we have $\w(\FF_p^d) = L(d)$.
\end{conj}

It is an easy exercise to check this for $d=1, 2$ and with some extra work one can show that $\w(\FF_p^3) = L(3) = 9$. The proof of $L(3)=9$ appears in the appendix of version 4 of the arXiv version of this paper. The fact that $\w(\FF_p^3) = L(3) $ can be established using the techniques of this paper. We decided to omit both of these results to keep the paper shorter. On the other hand, the next special case $d=4$ seems to be out of reach and even computing $L(4)$ seems computationally unfeasible.

All known lower bound constructions suggest that the case when the multiset $X \subset \FF_p^d$ is contained in a box $[-K, K]^d$ of bounded size $K$ should play a special role in the Erd{\H o}s--Ginzburg--Ziv problem. Namely, the exact value of the optimal constant $C$ in the bound $\s(\FF_p^d) \le (C+o_p(1)) p $ for large $p$ should come from this special case. This is essentially the content of Conjecture \ref{conj}. As a first step towards the proof of Theorem \ref{main}, and as a way to explain our key ideas, we show the following.

\begin{theorem}\label{thbounded}
    Fix $d, K \ge 1$, $\varepsilon>0$, and let $p > p_0(d, K, \varepsilon)$ be a prime. Suppose that $X \subset [-K, K]^d$ is a multiset of at least $(L(d)+\varepsilon) p$ elements. Then $X$ contains $p$ elements with zero sum modulo $p$.
\end{theorem}

Note that the constant $L(d)$ in the above result is tight, as can be seen by the lower bound constructions for the Erd{\H o}s--Ginzburg--Ziv problem discussed above.
This result can be thought of as a variant of Theorem \ref{main} in the important special case when $X$ is contained in a bounded-size box. The proof of Theorem \ref{thbounded} heavily relies on ideas from convex geometry, and we give a detailed overview of these ideas in Section \ref{over}. We note that Theorem \ref{thbounded} does not follow directly from Theorem \ref{main}, since we are unable to show that $\w(\FF_p^d) = L(d)$ holds for large $p$. Nevertheless, the proof of the former can be easily extracted from the proof of the latter. For this reason, we do not present a complete proof of Theorem \ref{thbounded} in this paper; instead, we give only an outline demonstrating the main ideas.

\subsection{A structural result} 
The proof of Theorem \ref{main} is based on the ideas from the proof of Theorem \ref{thbounded} (which we discuss later in Section \ref{over}) but requires new ingredients. Indeed, as the following construction shows, a large set $X \subset \FF_p^d$ without $p$ elements with zero sum need not be contained, or even almost contained, in a bounded-size box $[-K, K]^d$ in any coordinate system.

Let us take $1\le d' < d$ and consider a multiset $X' \subset \FF_p^{d'}$ without $p$ elements with zero sum, for example one constructed from a hollow polytope $P \subset \Q^{d'}$ or a $p$-hollow set in $\FF_p^{d'}$. Let $h: \FF_p^{d'} \rightarrow \FF_p^{d-d'}$ be a uniformly random function and let 
$$
X = \{ (x, h(x)), ~ x \in X'\} \subset \FF_p^{d'} \times \FF_p^{d-d'} \cong \FF_p^d.
$$
It is then easy to see that $X$ also does not contain $p$ elements with zero sum. On the other hand, the intersection of $X$ with an affine image of the box $[-K, K]^d$ is tiny and does not give much information about $X$.

To encapsulate constructions like this, we need a more detailed structural description of $X$.
Roughly speaking, for a given multiset $X \subset \FF_p^d$, we want to find a basis of $\FF_p^d$ and some $0\le d' \le d$ such that $X$ is `bounded' on the first $d'$ coordinates and `random' on the last $d-d'$ coordinates. Then, if $X$ does not contain $p$ elements with zero sum, we should expect this fact to be visible already in the projection $X' \subset [-K, K]^{d'}$ of $X$ onto the first $d'$ coordinates.
Indeed, this can be seen by the following heuristic argument. For the sake of contradiction, suppose that the projection $X'$ has some elements $x_1', \ldots, x_p' \in X'$ with sum zero. We want to lift the elements $x_i'$ to elements $x_i \in X$ such that the sum $y := x_1+\ldots +x_p$ is zero. The sum $y$ is already zero on the first $d'$ coordinates, so we only need to choose the $x_i$'s in such a way that the last coordinates are zero as well.
The set $X'$ is contained in a box $[-K, K]^{d'}$, where $K$ and $d$ are fixed and $p$ is large. In particular, we expect every element $x' \in X'$ to have at least $\frac{|X|}{(2K+1)^{d'}} \gg p$ preimages in $X$. By our assumption, these preimages are distributed quite randomly in the fiber $\{x'\} \times \FF_p^{d-d'}$. Thus, the set of all possible sums $y=x_1+\ldots+x_p$ should be rather uniformly distributed on the last $d-d'$ coordinates. In particular, we can find the $x_i$'s such that all last coordinates are zero. This gives $p$ elements $x_1, \ldots, x_p \in X$ with zero sum, contradicting the initial assumption. 

Thus, if $X$ has no $p$ elements with zero sum, then neither does its projection $X'$. But $X'$ is bounded, so we can apply Theorem \ref{thbounded} to $X'$ and conclude that $|X| = |X'| \le (L(d)+\varepsilon) p$. This is roughly how we will eventually prove Theorem \ref{main}. The actual argument, however, is more involved, and we only get the constant $\w(\FF_p^d)$ in the upper bound instead of $L(d)$.
As it turns out, splitting coordinates into two parts where $X$ is `bounded' and `random' is not sufficient to show that the projection $X'$ has no $p$ elements with zero sum. The obstruction comes from the precise notion of `random' that we need to use. Roughly speaking, it says that $X$ is not concentrated on any strip of width $K' \gg K$ around a hyperplane $H \subset \FF_p^d$, except for the hyperplanes $H$ coming from the first $d'$ coordinates (on which $X$ is in fact concentrated). This condition always implies that we can easily find subsets of $X$ of size $p$ whose sum is zero on the last $d-d'$ coordinates.
However, if we start with a collection of elements $x_1', \ldots, x_p' \in X'$ which sum to zero on the first $d'$ coordinates, then we are only allowed to choose elements from the subset $\tilde X \subset X$ consisting of elements $x$ of the form $(x'_i, \tilde x) \in X$ for some $i$ and $\tilde x$. We do not control exactly which set $\tilde X$ we get. It is quite possible that the `randomness' condition is violated for $\tilde X$, in which case the lifting procedure cannot be performed. 

To fix this problem, one might try to apply the same structural decomposition to the set $\tilde X$. Namely, one could consider a new coordinate system in which $\tilde X$ is bounded on the first $d''$ coordinates and `random' on the rest, for some new $d'' > d'$, and then try to run the lifting procedure on $\tilde X$ with respect to this new coordinate system.
The problem, however, is that we do not know whether the sum $x_1'+\ldots+x_p'$ is zero on the first $d''$ coordinates, since our first step guaranteed this only for the first $d'$ coordinates. It is tempting to apply the first step of the argument again with $d'$ replaced by $d''$ and with $X$ replaced by $\tilde X$. But the new set $\tilde X$ might be too small for that step to work directly. Indeed, the only information we have is that $\tilde X$ contains at least $p$ elements, which is far too small.
Thus, the first step of the proof, where we find elements $x_1', \ldots, x_p' \in X'$ which sum to zero on the first $d'$ coordinates, must already account for the possibility that the corresponding sets $\tilde X$ are not random enough for the lifting procedure. We want to ensure that whenever we find the collection of vectors $x_1', \ldots, x_p'$, their sum is not only zero on the first $d'$ coordinates but also zero on the first $d''$ coordinates after we apply the structural decomposition to the corresponding set $\tilde X$.

One of our main technical results, Theorem \ref{fdl}, which we call the Flag Decomposition Lemma, is designed exactly for this purpose. Section \ref{sfdl} is devoted to formulating and proving this result. Namely, we define certain poset structures, called `convex flags', which are composed of many subspaces in $\FF_p^d$ and polytopes in $\Q^{d'}$. We use these structures to decompose an arbitrary set of points $X \subset \FF_p^d$ into pieces with several useful properties.
Using these properties and the convex flag structure, we can adapt the ideas from the proof of Theorem \ref{thbounded} to obtain a collection of points whose sum is zero on the `bounded' part of the decomposition. Then the properties of the decomposition will give us sufficient randomness conditions on the remaining set of coordinates to run the lifting argument. In the end, this leads to the desired upper bound on the size of $X$. In particular, this strategy requires us to perform our convex geometry argument in an abstract poset setting. The precise statement is Theorem \ref{Helly}, which we call Helly's Theorem for Convex Flags, and we devote Section \ref{s2} to stating and proving it.

In Section \ref{set_expansion} we perform the lifting step of the argument. In many ways, this part of the proof is closely related to the Alon--Dubiner \cite{AD} proof of the linear upper bound $\s(\FF_p^d) \le C_d p$. Specifically, their argument corresponds precisely to the case $d' = 0$ in the above discussion, that is, to the case in which the set $X$ is entirely random-looking.
In this case, the delicate convex geometric obstructions are no longer present, and one can show much stronger bounds on $|X|$. Namely, as long as $X$ is random-looking enough and $|X| \ge (1+\varepsilon)p$ holds for some fixed $\varepsilon >0$, one can already find $p$ points with zero sum inside $X$. To show this, Alon--Dubiner \cite{AD} used tools from additive combinatorics and spectral graph theory. Roughly speaking, they showed that not only zero but in fact every element of $\FF_p^d$ can be expressed as a sum of $p$ elements of $X$.
They then deduce a linear upper bound on the size of $X$ by observing that, if the conditions for this argument are not satisfied, then for some hyperplane $H$ we have $|H\cap X| \gg |X|$. Thus, one can replace $X$ with $X \cap H$ and use induction on $d$ to finish the proof. In our situation, we use these additive combinatorics tools to lift the elements $x_1', \ldots, x_p' \in X'$ to elements of $X$ so that their sum is zero on the last $d-d'$ coordinates. Our argument requires some modifications, since we are more restricted in the choice of elements $x_i \in X$ that can be used for set expansion. However, the main idea and structure of this part of our argument are similar to those of Alon and Dubiner.

In Section \ref{balcon} we prove an auxiliary convex geometry statement which allows us to find $p$ points with zero sum using some geometric properties of $X$. We explain this in further detail in Section \ref{over} where we give an outline of the proof of Theorem \ref{thbounded}, which addresses the special case when $X$ is contained in a box of bounded size.

In Section \ref{s5} we put everything together and prove Theorem \ref{main}.

\subsection{Outline of the proof of Theorem \ref{thbounded}}\label{over}

Fix $d, K \ge 1, \varepsilon>0$, and let $p$ be a large prime. Consider a multiset $X \subset [-K, K]^d$ of size at least $(L(d)+\varepsilon) p$. We want to find $p$ elements of $X$ which sum to zero modulo $p$. Let us rephrase this problem in a more convenient form. Let $w: [-K, K]^d \rightarrow \N$ denote the indicator function of $X$ accounting for multiplicities. Then we want to find a point $q \in [-K, K]^d$ with integer coordinates
and non-negative integer coefficients $a_x$, for $x \in [-K, K]^d$, such that:
\begin{align}\label{thb1}
    \sum_{x \in [-K, K]^d} a_x &= p,\\
    q = \frac{1}{p}\sum_{x \in [-K, K]^d} &a_x x,\label{thb2}\\
    a_x \le w(x),\text{ for every }&x\in[-K, K]^d.\label{thb3}
\end{align}

For a function $w: \R^d \rightarrow \R_{\ge 0}$ with finite support, a point $q \in \R^d$ and $\theta \in [0, 1]$ we say that $q$ is $\theta$-{\it central} for $w$ if for every half-space $H^+ \subset \R^d$ which contains $q$ we have
$$
\sum_{x \in H^+} w(x) \ge \theta \sum_{x \in \R^d} w(x),
$$
i.e. the half-space $H^+$ contains at least a $\theta$-fraction of the weight of $w$.  

First, we consider a fractional relaxation of (\ref{thb1})-(\ref{thb3}); that is, we allow coefficients $a_x$ that are not necessarily integers. We observe that one can find the desired coefficients $a_x$ provided that the point $q$ on the left hand side of (\ref{thb2}) is $\theta$-central for $w$ with some parameter $\theta$. Indeed, if $q$ is $\theta$-central for $w$ for some $\theta>0$, then $q$ belongs to the convex hull of the support of $w$. Thus, there exists a convex combination with coefficient vector $(b_x)$, where $x \in \operatorname{supp}w$, such that $q = \sum_x b_x x$. It turns out that we can control the magnitude of the coefficients $b_x$ in terms of $\theta$:

\begin{prop}\label{propb}
    Let $w:\R^d \rightarrow \R_{\ge 0}$ be a function such that $\sum_x w(x) = W$ and let $q$ be a $\theta$-central point for $w$ for some $\theta > 0$. Then there exist real coefficients $b_x \ge 0$ such that: $\sum b_x = 1$, $q = \sum_x b_x x$ and for every $x$ we have $b_x \le \theta^{-1} \frac{w(x)}{W}$.
\end{prop}

See Section \ref{balcon} for a proof. In our application, we have $b_x = \frac{a_x}{p}$ and we need $a_x$ to satisfy $a_x \le w(x)$. Thus, the relaxed version of (\ref{thb1})-(\ref{thb3}) would follow from Proposition \ref{propb} if $q$ is a $\theta$-central point for $w$ with $\theta = \frac{p}{|X|}$. 

To solve the original question we have two problems:
\begin{itemize}
    \item[(i)] make sure that the coefficients $a_x = p b_x$ are integers,
    \item[(ii)] construct a point $q \in \Z^d$ which is $\theta$-central for $w$ for a suitable $\theta$,
\end{itemize}

To guarantee (i), we need to require an additional condition on $q$. Indeed, it is very possible that for some choices of $w: \Z^d \rightarrow \N$ and $q \in \Z^d$, no convex combination $b_x$ as in Proposition \ref{propb} has all $p b_x$ integral. For example, if the function $w$ is supported on a sublattice $\Lambda \subset \Z^d$ and $q \in \Z^d \setminus \Lambda$, then it is easy to check that if, for some prime $p$, we have $p b_x \in \Z$ for all $x$, then $p$ must divide the index $|\Z^d / \Lambda|$. 
To overcome this obstacle we put the following restriction:
\begin{equation}\label{thbq}
    q\text{ belongs to the minimal lattice spanned by the support of }w.
\end{equation}
This condition, however, is not sufficient either. Suppose that $d=2$, let $w$ be supported on the points $(0, 0), (2, 0), (0, 1), (1, 1) \in \Z^2$, and take $q = (1, 0)$. Then the support of $w$ spans the whole lattice $\Z^2$ and $q \in \Z^2$. However, for every $p>2$, there does not exist a convex combination of the form $q = \sum_{x \in \operatorname{supp} w} \frac{a_x}{p}x$ with integer $a_x$. Indeed, since $q$ lies on the boundary of the support of $w$, only points $(0, 0)$ and $(2, 0)$ can be used in the convex combination. We then run into the previous problem: $q$ does not belong to the lattice spanned by these points. Thus, we need to refine (\ref{thbq}) as follows. If $P$ is the convex hull of the support of $w$ and $\Gamma$ is the minimal face containing $q$, then we need
\begin{equation}\label{thbq2}
    q\text{ belongs to the minimal lattice spanned by the support of }w|_\Gamma.
\end{equation}

This condition turns out to be sufficient (up to some minor conditions on $w$).
Thus, if for some constant $\theta >0$ we can show that for a given function $w$ there exists a $\theta$-central point $q \in \Z^d$ satisfying (\ref{thbq2}), then for large $p$ there exist coefficients $a_x$ satisfying (\ref{thb1})-(\ref{thb3}) provided that $|X| > (1+\varepsilon)\frac{p}{\theta}$. The $\varepsilon$-error term comes from rounding coefficients when passing from a fractional to an integral solution. This would then imply the upper bound of $(1+\varepsilon) \frac{p}{\theta}$ in the special case of the Erd{\H o}s--Ginzburg--Ziv problem.

So far we have reduced our problem to constructing a $\theta$-central point $q$ with the additional integral property (\ref{thbq2}). Again, we start the discussion with relaxed versions of these conditions and then make the necessary adjustments to get what we need. The classical Centerpoint Theorem in convex geometry gives us a way to construct central points for arbitrary functions on $\R^d$:
\begin{theorem}[Centerpoint Theorem]\label{centerb}
    Let $w:\R^d \rightarrow \R_{\ge 0}$ be a function with finite support. Then there exists a $\frac{1}{d+1}$-central point $q \in \R^d$ for $w$.
\end{theorem}

This result is a consequence of another classical convex geometry result, Helly's Theorem:
\begin{theorem}[Helly's Theorem]\label{hell}
    Let $\F$ be a collection of compact convex sets in $\R^d$ such that every $d+1$ of them share a common point. Then all sets in $\F$ share a common point.
\end{theorem}

Note that the number $d+1$ in the statement cannot be lowered since otherwise we can take $\F$ to be the collection of faces of a $d$-dimensional simplex $\Delta = \conv\{0, e_1, \ldots, e_d\} \subset \R^d$.

\begin{proof}[Proof of Theorem \ref{centerb}]
    Let $\F$ be the collection of all closed half-spaces $H^+$ such that $w(H^+) > \frac{d}{d+1} w(\R^d)$ (to make the sets compact we can intersect $H^+$ with a ball of sufficiently large radius). Now, by the pigeonhole principle, every collection of $d+1$ halfspaces in $\F$ shares a common point in the support of $w$. Thus, by Helly's theorem there exists a point $q$ belonging to all half-planes in $\F$. This point is $\frac{1}{d+1}$-central for $w$.
\end{proof}

In order to construct central points $q$ with additional properties, we need to prove more refined versions of Helly's Theorem. The deduction of the corresponding Centerpoint Theorem always proceeds in the same way as in the above argument.

Now to get a central point satisfying (\ref{thbq}) (which in itself is not enough for us but closer to the condition (\ref{thbq2}) which we really want) we can use the following integer Helly's Theorem due to Doignon \cite{D}:
\begin{theorem}[Integer Helly's Theorem]\label{inthell}
    Let $\F$ be a collection of compact convex sets in $\R^d$ such that every $2^d$ of them share a common point $q\in \Z^d$. Then all sets in $\F$ share a common point $q \in \Z^d$.
\end{theorem}

Note that the constant $2^d$ in the above is tight: let $\F$ consist of convex sets $F=\conv (\{ 0, 1\}^d \setminus \{x\})$ over all vertices $x$ of the boolean cube $\{0, 1\}^d$. Then every $2^{d}-1$ of the sets in $\F$ share an integer point, but the intersection of all of them is disjoint from $\Z^d$. Using this result, we can find a $2^{-d}$-central point $q$ lying in the lattice spanned by the support of $w$. If this point happens to lie in the interior of the convex hull of the support of $w$, then the second part of the argument goes through and we get $|X| < (1+\varepsilon) 2^d p$. Recall, however, that for $d\ge 3$ the Erd{\H o}s--Ginzburg--Ziv constant is strictly larger than $2^d$, so this strategy has no chance of working without modifying Theorem \ref{inthell} in some way. 

To illustrate that it is not always possible to find a point $q$ in Theorem \ref{inthell} which lies in the interior of $\conv(\operatorname{supp}w)$, let us consider the following example coming from the lower bound $\s(\FF_p^3) \ge 9(p-1)+1$. Consider the following collection of 9 points in $\Z^3$:
\begin{align*}
S = \{ (0, 0, 0), (1, 0, 0), (0, 1, 0), (1, 1, 0),\\
    (2, 0, 1), (2, 1, 1), (0, 2, 1), (1, 2, 1),\\
    (2, 2, 2)\},    
\end{align*}
let $w$ be the characteristic function of $S$. The minimal lattice containing $S$ is $\Z^3$ and the $\frac18$-central point $q$ guaranteed by Theorem \ref{inthell} is $q = (1, 1, 1)$. Indeed, every half-space containing $q$ contains at least $2$ points of $S$. The point $q$ does not belong to the interior of $P = \conv S$: it lies on a face $\Gamma \subset P$ given by 
$$
\Gamma = \conv\{ (0, 0, 0), (2, 2, 2), (2, 0, 1), (0, 2, 1) \}.
$$
Note that the minimal lattice $\Lambda$ containing the 4 points above is smaller than the intersection of $\Z^3$ with the affine hull of $\Gamma$ and that the central point $q$ does not belong to it. Thus, our argument so far breaks down on this example. Indeed, this example shows that the Erd{\H o}s--Ginzburg--Ziv constant in 3 dimensions is at least 9. Moreover, note that the polytope $P$ is in fact a hollow polytope. 
This is not a coincidence: it turns out that hollow polytopes play a similar role for our variant of Helly's Theorem as the role of a simplex $\Delta$ for Theorem \ref{hell} and the role of the boolean cube $\{0, 1\}^d$ for Theorem \ref{inthell}. 

By adapting the proof of Theorem \ref{inthell} from \cite{D} we can prove a variant of Helly's Theorem which achieves (\ref{thbq2}):

\begin{restatable}{theorem}{maintheorem}\label{thm:cpt}
Let $P \subset \Q^d$ be a polytope and let $w: P \rightarrow \R_{\ge 0}$ be a function with finite support. Then there exists a face $\Gamma \subset P$ and a point $q$ in the interior of $\Gamma$ such that $q$ is $\frac{1}{L(d)}$-central for $w$ and $q$ belongs to the lattice spanned by the support of $w|_\Gamma$.
\end{restatable}

To obtain (\ref{thbq2}) we apply this theorem to $P = \conv(\operatorname{supp} w)$. The constant $L(d)$ in Theorem \ref{thm:cpt} is tight: similarly to the previous examples, let $P$ be a hollow polytope with $L(d)$ vertices and define $\F$ to be the family of sets $\conv(S \setminus \{x\})$ where $S$ is the set of vertices of $P$ and $x$ ranges over $S$.

We prove Theorem \ref{thm:cpt} in Section \ref{polyhelly}. Now we finally have all the tools needed to prove Theorem \ref{thbounded}. We start with an arbitrary set $X \subset [-K, K]^d$ of size at least $(1+\varepsilon) L(d) p$ and let $w$ be its characteristic function. After pruning $X$ a bit to remove all elements with very small multiplicity, we let $P = \conv(X)$ and apply Theorem \ref{thm:cpt} to $P$ and $w$. We obtain a face $\Gamma$ and a point $q$ in the interior of $\Gamma$ which is $\theta$-central for $w$, with $\theta = \frac{1}{L(d)}$, and lies in the lattice spanned by $X \cap \Gamma$.
Using Proposition \ref{propb} and the lattice condition, we conclude that, for large enough $p$, there are nonzero coefficients $b_x = \frac{a_x}{p}$ with sum 1 such that $q = \sum_{x \in X} b_x x$. Moreover, for each $x \in X$, we have $a_x\in \N$ and 
$$
a_x = p b_x \le p (1+\varepsilon) \theta^{-1} \frac{w(x)}{|X|} = \frac{(1+\varepsilon)L(d)p}{|X|} w(x) \le w(x),
$$
so we obtain the desired $p$ elements of $X$ which sum to zero modulo $p$.

\section{Proofs of Proposition \ref{thw} and Proposition \ref{wl}}

\subsection{Proof of Proposition \ref{thw}}\label{theasy2}

We argue indirectly. Assume that there are vectors $v_1, \ldots, v_n \in \FF_p^d$, with $n \ge {2d-1 \choose d}+2$ such that for every non-negative integers $\alpha_1, \ldots, \alpha_n$ whose sum is $p$, we have $\sum \alpha_i v_i = 0$ if and only if $\alpha_i = p$ for some $i$. Note that this condition implies that vectors $v_i$ are pairwise distinct. Let $S = \{v_1, \ldots, v_n\}$.

\begin{claim}\label{h}
There is a nonzero function $h: \{1, \ldots, n\} \rightarrow \FF_p$ such that $h(n) = 0$ and for every polynomial $f \in \FF_p[x_1, \ldots, x_d]$ of degree at most $d-1$ we have
\begin{equation*}
    \sum_{i = 1}^n h(i) f(v_i) = 0.
\end{equation*}
\end{claim}

\begin{proof}
Recall that the dimension of the linear space of polynomials with $\FF_p$-coefficients of degree at most $d-1$ is equal to $2d-1 \choose d$. Thus, the desired function $h$ is a solution of a system consisting of ${2d-1 \choose d} + 1$ linear equations in $n \ge {2d-1 \choose d} + 2$ variables.
\end{proof}

For $i = 1, \ldots, p$ and $j = 1, \ldots, d$, let $y_{i, j}$ be a set of variables. Let $y_i$ be the $d$-dimensional vector $(y_{i, 1}, \ldots, y_{i, d})^T$.
Consider the following polynomial in $p \times d$ variables:
\begin{equation}\label{poly}
    F(y_1, \ldots, y_p) = \prod_{j = 1}^d \left( 1 - \left(\sum_{i = 1}^p y_{i, j}\right)^{p-1}  \right).
\end{equation}
Note that if we substitute in $P$ some vectors $y_i \in \FF_p^d$ then $F(y_1, \ldots, y_p) = 1$ if $y_1 + \ldots + y_p = 0$ and it equals 0 otherwise. Thus, if we consider a sequence $v_{i_1}, \ldots, v_{i_p}$ of $p$ elements of $S$ then $F(v_{i_1}, \ldots, v_{i_p}) =1$ if $i_1 = \ldots = i_p$ and $F(v_{i_1}, \ldots, v_{i_p}) = 0$ otherwise.

Now we define a function $\Phi: \{1, \ldots, n\} \rightarrow \FF_p$ by:
\begin{equation}\label{f}
    \Phi(t) = \sum_{i_1, \ldots, i_{p-1} \in [n]} h(i_1)\ldots h(i_{p-1}) F(v_{i_1}, \ldots, v_{i_{p-1}}, v_t).
\end{equation}
Let us compute $\Phi(t)$ in two different ways and arrive at a contradiction. On the one hand, the summand in (\ref{f}) is zero unless $v_{i_1} = \ldots = v_{i_{p-1}} = v_t$. Hence 
\begin{equation}\label{eqn}
    \Phi(t) \equiv h(t)^{p-1} \pmod p.
\end{equation}
On the other hand, $F(y_1, \ldots, y_p)$ is a polynomial in variables $y_{i, j}$ of degree $d(p-1)$ and so it can be expressed as a linear combination of monomials of the form $m_1(y_1) m_2(y_2) \ldots m_{p}(y_p)$ where $m_i \in \Z[x_1, \ldots, x_d]$ and $\sum_{i=1}^p \deg m_i \le (p-1)d$. Restricting the sum (\ref{f}) on a fixed monomial we obtain:
\begin{equation}\label{sm}
    \sum_{i_1, \ldots, i_{p-1} \in [n]} h(i_1)\ldots h(i_{p-1}) m_1(v_{i_1}) m_2(v_{i_1}) \ldots m_{p-1}(v_{i_1}) m_{p}(v_t) = m_{p}(v_t) \prod_{j = 1}^{p-1} \left ( \sum_{i = 1}^n h(i)m_j(v_{i}) \right).
\end{equation}
Thus, by Claim \ref{h}, if $\deg m_j \le d-1$ for some $j \le p-1$ then the corresponding multiple in (\ref{sm}) must be zero. Otherwise, $\deg m_j \ge d$ for all $j \le p-1$. But this implies that $\deg m_p = 0$, that is, $m_p$ is a constant function. Thus, in either case the expression (\ref{sm}) does not depend on $t$. However, by the construction of $h$ and (\ref{eqn}) we have $\Phi(n) \equiv 0 \pmod p$ and $\Phi(t)$ is not zero for some $t \in \{1, \ldots, n\}$ because $h$ is a nonzero function by Claim \ref{h}. This contradiction completes the proof.

\subsection{Proof of Proposition \ref{wl}}\label{theasy1}

We begin with a different characterization of integer points of polytopes. For a polytope $P \subset \Q^d$, we denote by $\Lambda(P)$ the minimal by inclusion lattice containing the vertices of $P$. We say that a prime $p$ is {\it $P$-good} if, for every face $\Gamma \subset P$, the quotient group $\Lambda(P) / \Lambda(\Gamma)$ has no elements of order $p$.

\begin{claim}\label{crit}
Let $P \subset \Q^d$ be a polytope and let $q \in P$ be a point. Let $q_1, \ldots, q_s$ be the vertices of $P$. The following assertions are equivalent:
\begin{itemize}
    \item[(i)] The point $q$ is an integer point of $P$.
    \item[(ii)] There exists a constant $n_0(P)$ such that for all numbers $n > n_0(P)$ there are nonnegative integer coefficients $\alpha_1, \ldots, \alpha_s$ such that:
\begin{equation}\label{e}
    \sum_{i=1}^s \alpha_i q_i = n q, ~~ \sum_{i=1}^s \alpha_i = n.
\end{equation}
    \item[(iii)] The point $q$ belongs to the minimal lattice containing points $q_1, \ldots, q_s$ and the condition on $n$ in (ii) is satisfied for some $n = p$, where $p$ is a $P$-good prime.
\end{itemize}
\end{claim}

\begin{proof}
If $q$ is a vertex of $P$, then there is nothing to prove. Thus, for the rest of the proof, we may assume that $q$ is not a vertex of $P$.

\paragraph{(i)$\Rightarrow$(ii).} By replacing $P$ with the minimal face containing the point $q$ we reduce to the case when $q$ is an interior point of $P$.
This implies that there exists a convex combination 
\begin{equation*}
    (q, 1) = \sum_{i = 1}^s \beta_i (q_i, 1),
\end{equation*}
where all coefficients $\beta_i$ are positive rational numbers. Let $m_0$ be the least common multiple of the denominators of $\beta_i$. Since there are only a bounded number of points in $\Lambda(P) \cap P$, $m_0$ is bounded by a constant $m_0(P)$. Then we can write $\beta_i = b_i / m_0$ for some positive integers $b_i$.

Since $q$ belongs to $\Lambda(P)$, there is an integer affine combination
\begin{equation}
    \sum_{i = 1}^s c_i (q_i, 1) = (q, 1),
\end{equation}
where $c_i \in \Z$ are integer coefficients. Let $K = \max |c_i|$; again, we have $K \le K_0(P)$ for some constant $K_0(P)$. We claim that one can now take $n_0(P) = 2 K_0(P) m_0(P)^2$. Indeed, consider an arbitrary number $n > 2K m_0^2$. Write $n =  m_0 k + r$ for some $0 \le r < m_0$ and define the coefficients $\alpha_i$ by setting $\alpha_i = k b_i + r c_i$. Then we have
\begin{equation}
    \sum_{i=1}^s \alpha_i(q_i, 1) = k \sum_{i=1}^s b_i(q_i, 1) + r \sum_{i = 1}^s c_i(q_i, 1) = (km_0 +r)(q, 1) = n(q, 1),
\end{equation}
and for every $i$ we have $\alpha_i = kb_i + rc_i \ge k - rK \ge \lfloor n / m_0\rfloor - K m_0 > 0$ by the choice of $n$. Thus, the coefficients $\alpha_i$ satisfy (ii). 

\paragraph{(ii)$\Rightarrow$(iii).} All primes $p\ge 2$ except for a finite collection are $P$-good, so we can take $p$ to be a $P$-good prime larger than $n_0(P)$ and apply (ii). To see that $q$ lies in the minimal lattice, apply (ii) for two consecutive values of $n$ and take the difference of the corresponding expressions (\ref{e}).

\paragraph{(iii)$\Rightarrow$(i).} Let $\Gamma$ be the minimal face of $P$ containing $q$. By a shift of coordinates we may assume that the origin $0$ lies in $\Lambda(\Gamma)$ so that it becomes a linear lattice, not just an affine one. 
By our assumption, $q \in \Lambda(P)$ and there is a $P$-good prime $p$ and nonnegative integer coefficients $\alpha_1, \ldots, \alpha_s$ such that
$$
\sum_{i=1}^s \alpha_i (q_i, 1) = p(q, 1).
$$
Since $q \in \Gamma$, we have $\alpha_i = 0$ for all $i$ such that $q_i \not \in \Gamma$. Let $\Lambda$ be the intersection of $\Lambda(P)$ with the affine hull of $\Gamma$. Then the quotient $G= \Lambda/\Lambda(\Gamma)$ is a finite group, and the definition of a $P$-good prime implies that $p$ is coprime to $|G|$. Let $b > 1$ be an integer such that $p b = 1 \pmod{|G|}$. Then
\begin{equation}\label{latticeeq}
(q, 1) = pb (q, 1) - (pb-1) (q, 1) = \sum_{i=1}^s \alpha_i b (q_i, 1) - \frac{pb-1}{|G|} (|G| q, |G|).
\end{equation}
By the definition of $\Lambda(\Gamma)$, the points $q_i$ with $\alpha_i >0$ belong to $\Lambda(\Gamma)$. By Lagrange's theorem the point $|G|q$ belongs to $\Lambda(\Gamma)$ so, by (\ref{latticeeq}), the point $q$ lies in $\Lambda(\Gamma)$ as well. This completes the last implication and the claim is proved. 
\end{proof}

Now we are ready to prove Proposition \ref{wl}. Let $P \subset \Q^d$ be a hollow polytope with $L(d)$ vertices. After rescaling $P$, we may assume that $P \subset \Z^d$ and that $\Z^d$ is the minimal lattice containing the vertices of $P$. Denote the vertices of $P$ by $q_1, \ldots, q_s$. Let $p$ be a $P$-good prime. Then we can view the vertices of $P$ as a subset of $\FF_p^d$. If $P$ modulo $p$ has a zero-sum $\sum \alpha_i q_i \equiv 0 \pmod p$ for some non-negative integers $\alpha_i$ whose sum is $p$, then the point $q = \frac{1}{p} \sum \alpha_i q_i$ belongs to $P \cap \Z^d$. By Claim \ref{crit}, $q$ is an integer point of $P$. Since $P$ is hollow, we must have $q = q_i$ for some $i$, which implies that $\alpha_i = p$. 
We conclude that $\w(\FF_p^d) \ge L(d)$ for all primes $p$ which are $P$-good, and in particular this is true for sufficiently large primes $p$.

\section{Convex flags and Helly's theorem}\label{s2}

\subsection{Basic notions} 

In this section we define a certain generalization of polytopes which we call {\em convex flags}. Convex flags will be a convenient way to describe the combinatorial structures appearing during the proof of Theorem \ref{main}. We start by explaining how a polytope $P$ can be viewed as a convex flag, and after that we give a general definition.

Recall that a {\it polytope} $P$ in $\Q^d$ is a convex hull of a finite, non-empty set of points of $\Q^d$. Note that the dimension of $P$ may be less than $d$. For a polytope $P$ in $\Q^d$ let $\P(P)$  be the set of all faces of $P$ (including $P$ itself but excluding the ``empty'' face) with the partial order induced by inclusion. 

Note that for every set of faces $S \subset \P(P)$ there is a unique minimal face $\Gamma \in \P(P)$ which contains all faces from $S$. 
Based on this observation, we call an arbitrary (finite) poset $\P$ {\it convex} if every subset $S \subset \P$ has a {\it supremum} $\sup S$. That is, the set of all $x \in \P$ such that $y \preceq x$ for every $y \in S$ has a minimum element\footnote{This terminology is not standard. In literature, posets which have this property are usually called {\it upper semilattices} but we prefer to use a simpler and more intuitive term instead. }.

Let $P_1, P_2$ be arbitrary polytopes in some $\Q$-spaces $\A_1$ and $\A_2$. An affine map $\psi: \A_1 \rightarrow \A_2$ is called a {\it map of polytopes} from $P_1$ to $P_2$ if $\psi(P_1) \subset P_2$. Clearly, a composition of maps of polytopes is again a map of polytopes. Note that $\psi$ is neither assumed to be injective nor surjective. 

Note that if $P_1$ is a face of $P_2$ then the corresponding inclusion map $\psi_{P_2, P_1}$ is a map of polytopes from $P_1$ to $P_2$. 
Thus, we can equip the set $\P(P)$ of faces of a polytope $P$ with the following structure: for every pair $x \preceq y \in \P(P)$ we consider the corresponding inclusion map $\psi_{y, x}$. We have thus encoded the structure of the original polytope $P$ in terms of its faces and inclusion maps between them. If we now allow maps $\psi_{y, x}$ that are not necessarily injective and replace $\P(P)$ by an arbitrary convex poset $\P$ then we arrive at the notion of a convex flag.

\begin{defi}[Convex flag] 
A {\em convex flag} is a convex poset $(\P, \prec)$ with the following additional structure.
For every $x \in \P$ there is a polytope $P_x \subset \A_x$ embedded in a $\Q$-space $\A_x$ and for every comparable pair $y \preceq x$ there is a map $\psi_{x, y}: \A_y \rightarrow \A_x$ from $P_y$ to $P_x$ with the property that for every chain $z \preceq y \preceq x$ we have $\psi_{x, z} = \psi_{x, y}\psi_{y, z}$. In particular, $\psi_{x, x}$ is the identity map of $\A_x$.
\end{defi}

When we say that $\P$ is a convex flag, we mean that $\P$ is a convex poset and that we have fixed corresponding polytopes $P_x \subset \A_x$ and maps $\psi_{x, y}$. 
As mentioned above, every polytope $P$ naturally corresponds to a convex flag, which we denote by $\P(P)$. Let us provide some other examples of convex flags.

\begin{example}[Binary tree, Figure \ref{fig1}]\label{ex1}
Let $\P$ be the set of strings $a_1a_2\ldots a_i$ consisting of $0$'s and $1$'s and of length $i \le d$ (including the empty string). For strings $s_1, s_2$ we have $s_1 \succeq s_2$ if $s_1$ is an initial segment of $s_2$. Note that $|\P| = 2^{d+1}-1$. 

For every $s \in \P$ let $\A_s = \Q$ and $P_s = [0, 1]$. Let $s \in \P$ and $s' = s a$ be a successor of $s$, where $a \in \{0, 1\}$. We define the map $\psi_{s, sa}: [0, 1] \rightarrow [0, 1]$ to be the projection on the point $a \in [0, 1]$. 
\end{example}

\begin{example}[Sunflower, Figure \ref{fig2}]\label{ex2}
Let $\P = \{ a, b_1, \ldots, b_n,  c_1, \ldots, c_n\}$. Here $a$ is the maximum element of $\P$ while elements $b_i$ and $c_i$ are ordered as follows: we have $c_i \prec b_i$ and $c_i \prec b_{i+1}$ (with indexes taken modulo $n$). Let $P_a \subset \R^2$ be an arbitrary $n$-gon and let $E_1, \ldots, E_n$ be the edges of $P_a$ labeled in a cyclic order. Let $v_{i-1}, v_i$ be the vertices of the edge $E_i$.

Let $P_{b_i} \subset \R^2$ be an arbitrary polygon which has a pair of parallel edges $F_{i0}, F_{i1} \subset P_{b_i}$. For every $i = 1, \ldots, n$, let $P_{c_i} = [0, 1]$. We now define maps between the polygons $P_a, P_{b_i}, P_{c_i}$. The map $\psi_{a, b_i}: P_{b_i} \rightarrow P_a$ is a projection of $P_{b_i}$ along its edges $F_{i0}$ and $F_{i1}$ onto the edge $E_i$. In particular, we have $\psi_{a, b_i}(F_{i0}) = v_{i-1}$ and $\psi_{a, b_i}(F_{i1}) = v_i$.
Now let $\psi_{b_i, c_i}: P_{c_i} \rightarrow P_{b_i}$ be an arbitrary affine map such that $\psi_{b_i, c_i}(P_{c_i}) \subset F_{i1}$. Similarly, let $\psi_{b_i, c_{i-1}}: P_{c_{i-1}} \rightarrow P_{b_i}$ be an arbitrary affine map such that $\psi_{b_i, c_{i-1}}(P_{c_{i-1}}) \subset F_{i0}$. 

The map $\psi_{a, c_i}: P_{c_i}\rightarrow P_a$ is now defined uniquely: we let $\psi_{a, c_i}(x) = v_i$ for every $x \in P_{c_i}$. This definition implies that $\psi_{x, z} = \psi_{x, y}\psi_{y, z}$ for all $z \preceq y \preceq x$ in $\P$. Indeed, the only triples $x, y, z$ for which this equality does not follow automatically are $(x, y, z) = (a, b_i, c_i)$ and $(a, b_i, c_{i-1})$. Therefore, we have defined a convex flag structure on $\P$.

The name ``sunflower'' comes from the following interpretation of $\P$: $P_a$ is the ``core'' of the sunflower $\P$, and the $P_{b_i}$'s are the ``petals'' which are glued together along edges $P_{c_i}$ and attached to $P_a$ at edges $E_i$. 

We may also allow $F_{i0}$ or $F_{i1}$ to degenerate into a single vertex; the resulting structure on $\P$ will also form a convex flag.
\end{example}

\begin{figure}
    \includegraphics[width=0.9\linewidth]{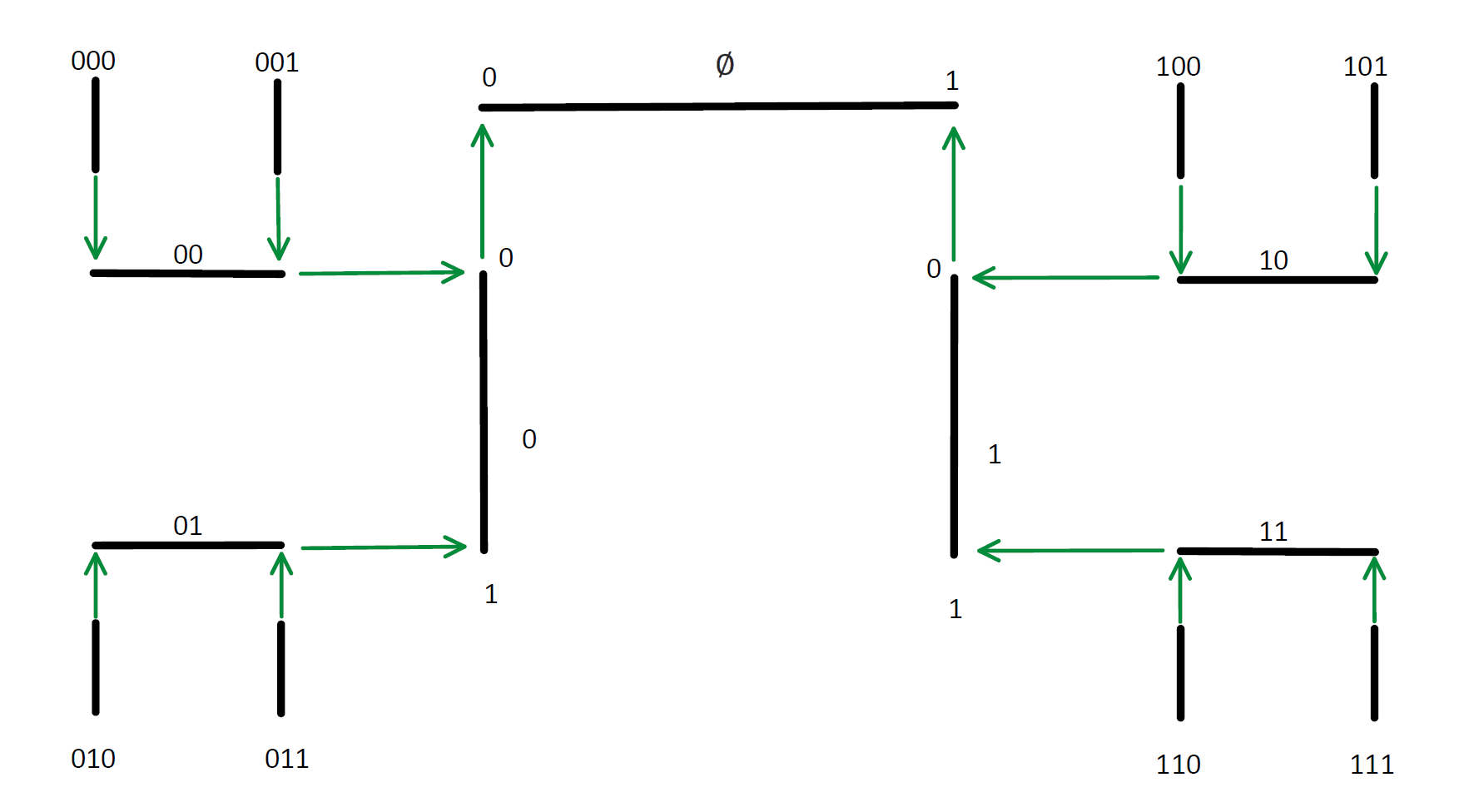}
    \caption{Binary tree for $d = 3$}
    \label{fig1}
\end{figure}
\begin{figure}
    \includegraphics[width=0.9\linewidth]{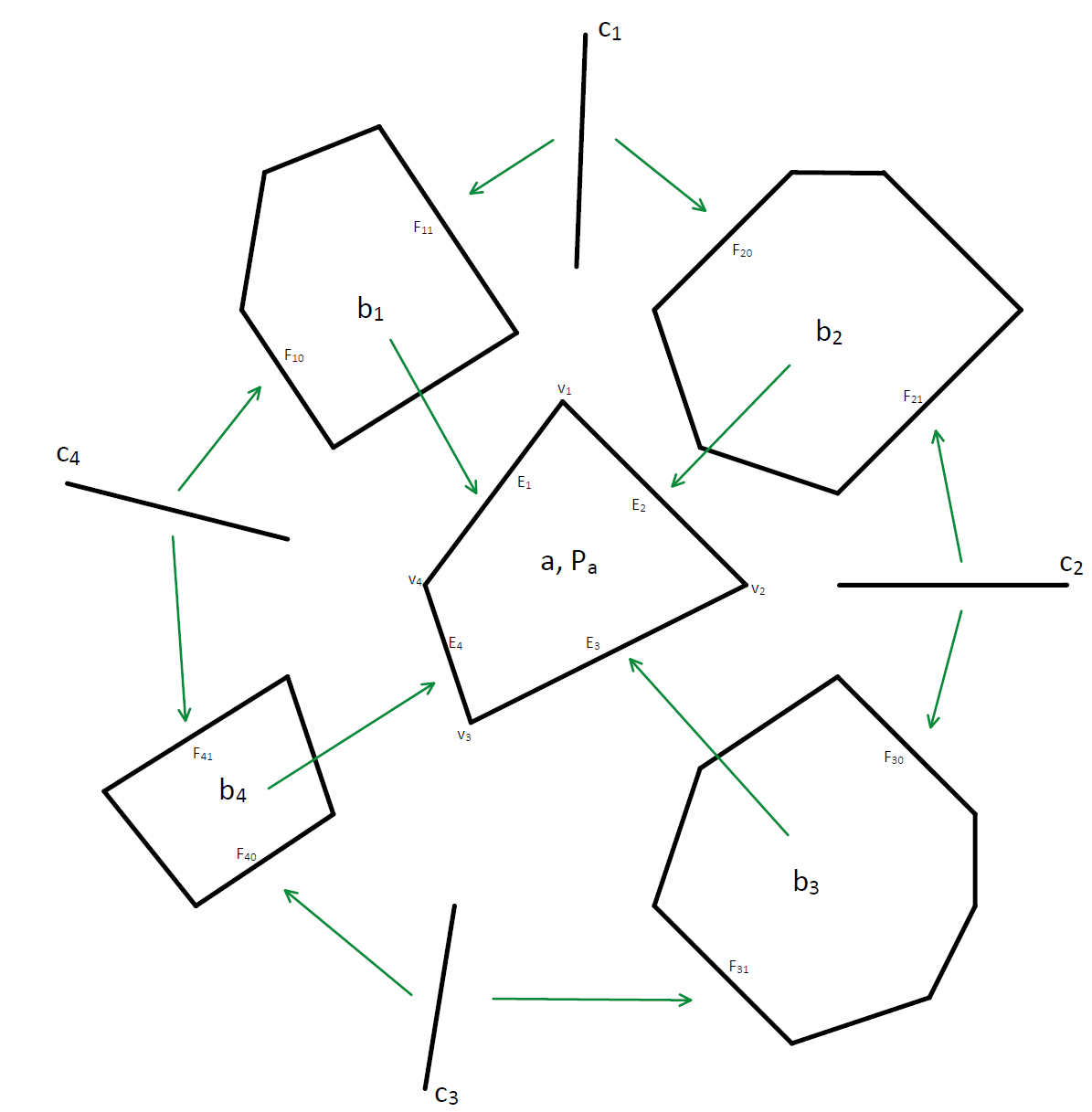}
    \caption{Sunflower for $n = 4$}
    \label{fig2}
\end{figure}

Now we translate the usual definitions of points and linear functions to this new setting. 

\begin{defi}[Linear functions]
A linear function $\xi$ on a convex flag $\P$ is a linear function $\xi_x: \A_x \rightarrow \R$ for some $x \in \P$. The {\em domain} $\D_\xi$ of $\xi$ is the set $\P_x := \{ y \in \P~|~y \preceq x\}$. For every point $q \in \A_y$, where $y \in \D_\xi$, we define $\xi_y(q) := \xi_x \psi_{x, y}(q)$.
\end{defi}

For $x \in \P$, we denote $\P^x :=\{ y \in \P~|~ x \preceq y\}$. Since $\P$ is a convex poset, for every $x_1, \ldots, x_n \in \P$ the set $\P^{x_1} \cap \ldots \cap \P^{x_n}$ also has the form $\P^x$ for some $x \in \P$. Namely, we take $x = \sup\{x_1, \ldots, x_n\}$. In particular, this intersection is non-empty.

\begin{defi}[Points]\label{pts}
A point $\bf q$ of a convex flag $\P$ is a point ${\bf q}_x \in P_x$ for some $x \in \P$ together with its images ${\bf q}_y = \psi_{y, x}{\bf q}_x$ for all $y \in \P^x$. We denote by $\D^{\bf q} := \P^x$ the {\em domain} of ${\bf q}$. The expression $\inf \D^{\bf q} := x$ denotes the minimum element $x$ of $\D^{\bf q}$.
\end{defi}

If for a linear function $\xi$ and a point $\bf q$ the sets $\D_\xi$ and $\D^{\bf q}$ intersect then we can define the value $\xi({\bf q})$ to be equal to $\xi_x({\bf q}_x)$ for some $x \in \D_\xi \cap \D^{\bf q}$. It follows from our definitions that this number does not actually depend on $x$.

For a set of points ${\bf q}_1, \ldots, {\bf q}_n$ of a convex flag $\P$ and non-negative coefficients $\alpha_1, \ldots, \alpha_n$ with sum 1, we define the convex combination $\alpha_1 {\bf q}_1 + \ldots + \alpha_n {\bf q}_n$ to be the unique point ${\bf q}$ of $\P$ such that $\D^{\bf q} = \bigcap_{i: \alpha_i > 0} \D^{{\bf q}_i}$ and for every $y \in \D^{\bf q}$ we have 
\begin{equation}\label{convc}
{\bf q}_y = \sum_{i:\, \alpha_i > 0} \alpha_i {\bf q}_{i, y}.
\end{equation}
For a set of points $S$ of a convex flag $\P$ we define the convex hull $\conv S$ to be the set of all points $\bf q$ which can be expressed as a convex combination of points from $S$.

Now we introduce the notion of lattices in convex flags.

\begin{defi}[Lattice]
\emph{A lattice $\Lambda$ in a convex flag} $\P$ is a collection of lattices $\Lambda_x \subset \A_x$ for all $x \in \P$ such that for every $x \preceq y$ we have $\psi_{y, x}\Lambda_x \subset \Lambda_y$.
\end{defi}

In what follows, we will usually work with a fixed convex flag $\P$ and a lattice $\Lambda$ on $\P$. For shorthand, we will refer to the pair consisting of a convex flag $\P$ and a lattice $\Lambda$ in $\P$ as a ``convex flag $(\P, \Lambda)$''.

A point $\bf q$ {\it belongs to the lattice} $\Lambda$ if ${\bf q}_x \in \Lambda_x$ for every $x \in \D^{\bf q}$. Equivalently, $\bf q$ belongs to $\Lambda$ if ${\bf q}_x \in \Lambda_x$, where $x = \inf \D^{{\bf q}}$. We denote the fact that ${\bf q}$ belongs to $\Lambda$ by the expression ${\bf q} \in \Lambda$ and we will call ${\bf q}$ {\it an integer point} of the convex flag $(\P, \Lambda)$. 

\subsection{Helly's theorem}

Fix a convex flag $(\P, \Lambda)$ with a lattice $\Lambda$. Let $\Omega$ be a set of points of the convex flag $\P$ which is closed under convex combinations (i.e. $\Omega = \conv\,\Omega$). Points ${\bf q} \in \Omega$ will be called {\it proper} points of the convex flag $(\P, \Lambda)$. Until the end of this section, we suppose that a set $\Omega$ of proper points for $(\P, \Lambda)$ is fixed, but we often omit it from the notation.

\begin{defi}[Helly constant]\label{arithm}
For a convex flag $(\P, \Lambda)$ with a set of proper points $\Omega$, define the Helly constant $L(\P, \Lambda, \Omega)$ as the maximum size $L$ of a collection of proper integer points ${\bf q}_1, \ldots, {\bf q}_L$ with the following property. Consider a convex combination
$$
{\bf q} = \sum_{i=1}^L \alpha_i {\bf q}_i,
$$
and suppose that ${\bf q} \in \Lambda$. Then $\alpha_i = 1$ for some $i$.
\end{defi}

We should point out that the last condition is not equivalent to saying that ${\bf q} = {\bf q}_i$ for some $i$. For brevity, we will usually omit $\Omega$ from the notation and write $L(\P, \Lambda)$ instead of $L(\P, \Lambda, \Omega)$.

\begin{example}
    Let $\P = \{x\}$ be a one-element poset, let $P_x \subset \Q^d$ be a polytope, let $\Lambda = \Z^d$, and let $\Omega$ be the set of all points of $P_x$. Then we have $L(\P, \Lambda) \le 2^d$. Indeed, if we have points $q_1, \ldots, q_{2^d+1} \in P_x \cap \Lambda$, then, by the pigeonhole principle there are indices $i \neq j$ such that $q_i = q_j \pmod 2$. Hence $q = \frac12 q_i + \frac12 q_j$ belongs to $\Lambda$, violating the condition in Definition \ref{arithm}. If $P_x$ contains the boolean cube $\{0, 1\}^d$, then the Helly constant of $(\P, \Lambda)$ equals $2^d$.
\end{example}

\begin{example}\label{exmpl}
    Let $P \subset \Q^d$ be a polytope and consider the corresponding convex flag $\P = \P(P)$. Let $\Omega$ be the set of points $\bf q$ of $\P$ such that $\inf \D^{\bf q}$ is the minimum face of $P$ which contains $\bf q$. Thus, the set of proper points $\Omega$ is in one-to-one correspondence with the set of points of $P$. For a face $\Gamma$ of $P$ let $\Lambda_\Gamma$ be the minimal lattice containing the vertices of $\Gamma$.
    
    Note that the integer points of the convex flag $(\P, \Lambda)$ are in bijection with integer points of the polytope $P$, according to Definition \ref{intpt}. Then if $P$ is a hollow polytope then $L(\P, \Lambda)$ is at most the number of vertices of $P$ which, in turn, is at most $L(d)$. 
    In fact, we show later that the inequality $L(\P, \Lambda) \le L(d)$ holds for every polytope $P \subset \Q^d$.
\end{example}

For the usual notion of convexity in $\Q^d$, we have the Hahn--Banach theorem: for every finite set $S$ and every $q \not \in \conv S$, there exists a linear function $\xi$ `separating' $q$ from $S$. However, this is no longer the case in the setting of convex flags. With this in mind, we define a second notion of convex hull:

\begin{defi}[Weak convex hull]\label{defwc}
For a set of points $S$ of $(\P, \Lambda)$, we define the {\it weak convex hull} $\wc(S)$ of $S$ to be the set of points ${\bf q}$ such that, for every linear function $\xi$ for which $\xi({\bf q})$ is defined, there is a point ${\bf s} \in S$ such that 
$$
\xi({\bf s}) \ge \xi({\bf q}).
$$
\end{defi}

Let ${\bf q}, {\bf q}'$ be a pair of points of a convex flag $(\P, \Lambda)$. We say that ${\bf q}$ is a projection of the point ${\bf q}'$ if $\D^{\bf q} \subset \D^{{\bf q}'}$ and ${\bf q}_x = {\bf q}'_x$ for every $x \in \D^{\bf q}$.

\begin{prop}\label{pf1}
For an arbitrary set of points $S$ of $(\P, \Lambda)$ and an arbitrary point ${\bf q}$,
we have ${\bf q} \in \wc(S)$ if and only if there exists ${\bf q}' \in \c(S)$ such that ${\bf q}$ is a projection of ${\bf q}'$.
\end{prop}

\begin{proof}
Take ${\bf q} \in \wc(S)$ and let $x = \inf \D^{\bf q}$. Let $X \subset P_x$ be the set of points ${\bf q}'_x \in \A_x$ over all ${\bf q}' \in \c(S)$ defined over $x$. Then $X$ is a convex subset of $P_x$. Note that if ${\bf q}_x \not \in X$ then by the usual Hahn--Banach theorem there is a linear function $\xi_x$ such that $\xi_x({\bf q}_x) > \xi_x(s)$ for every $s \in X$. Let $\xi$ be the unique linear function on $\P$ extending $\xi_x$ and note that we obtain a contradiction with ${\bf q} \in \wc(S)$. We conclude that ${\bf q}_x \in X$. Thus, there is some ${\bf q}' \in \c(S)$ such that ${\bf q}'_x = {\bf q}_x$. In other words, ${\bf q}$ is a projection of ${\bf q}'$.

The inverse implication can be checked directly. Indeed, replacing a point ${\bf q}' \in \c(S)$ by a projection ${\bf q}$ only decreases the number of conditions one has to satisfy.
\end{proof}

A set of points $S$ is {\it in weakly convex position} if no point of $S$ belongs to the weak convex hull of the other points. 


\begin{example}
    Let $P = [0, 1]$ and consider the convex flag $\P = \P(P)$. Let ${\bf 0}, {\bf 1}$ be the points of $\P$ such that $\D^{\bf 0} = \{[0, 1], \{0\}\}$, $\D^{\bf 1} = \{[0, 1], \{1\}\}$. Let ${\bf 0'}, {\bf 1'}$ be points such that ${\bf 0'}_P = {\bf 0}_P = 0$ and ${\bf 1'}_P = {\bf 1}_P = 1$ but $
    \D^{\bf 0'} = \D^{\bf 1'} = \{[0, 1]\}$. 
    So points ${\bf 0'}$ and ${\bf 1'}$ are projections of ${\bf 0}$ and ${\bf 1}$ respectively.

    Then the set $S = \{{\bf 0}, {\bf 1}, {\bf 0'}, {\bf 1'}\}$ is in convex position but not in weakly convex position. Indeed, the point ${\bf 0'}$ belongs to the weak convex hull of ${\bf 0}$ but cannot be expressed as a convex combination of ${\bf 0}$, ${\bf 1}$ and ${\bf 1'}$. We also have ${\bf 0'} = \frac{1}{2} {\bf 0'} + \frac{1}{2}{\bf 0}$. 
\end{example}

The following result explains why we call $L(\P, \Lambda)$ a Helly constant.

\begin{theorem}[Helly's Theorem for Convex Flags]\label{Helly}
Let $(\P, \Lambda)$ be a convex flag with a set of proper points $\Omega$ and denote by $L = L(\P, \Lambda, \Omega)$ its Helly constant.
Let $\mathcal F$ be a collection of sets $F \subset \Omega$ with the property that for every $F_1, \ldots, F_L \in \mathcal F$ there exists a proper integer point ${\bf q}$ such that ${\bf q} \in \bigcap_{i=1}^L \wc(F_i)$. Then there exists a proper integer point ${\bf q} \in \bigcap_{F \in \mathcal F} \wc(F)$.
\end{theorem}

Let us emphasize the fact that we cannot take ${\bf q}$ to be in the intersection of convex hulls $\c(F)$ but only weak convex hulls $\wc(F)$. On the other hand, we can still guarantee that ${\bf q}$ is a proper point. Recall that $\Omega$ is only assumed to be closed under convex combinations and it will not be the case in the applications that $\Omega = \wc(\Omega)$. 

\begin{proof}

As in the standard proof of Helly's Theorem, we proceed by induction on the size of the family $\F$. The base case $|\F| \le L$ follows from the assumption of the theorem. Let $\F = \{F_1, \ldots, F_n\}$ be a family of size $n > L$ satisfying the assumption of Theorem \ref{Helly}. By induction, for each $i = 1, \ldots, n$ there is a proper integer point ${\bf q}_i$ such that
$$
{\bf q}_i \in \bigcap_{j = 1,~ j \neq i}^n \wc(F_j).
$$

Now let $S = \{ {\bf q}_1, \ldots, {\bf q}_n\}$. We will show more generally that for every set $S$ of proper integer points of size $n$ there is a proper integer point ${\bf q}$ such that
\begin{equation}\label{scre}
    {\bf q} \in \bigcap_{i=1}^n \wc(S\setminus \{{\bf q}_i\}).
\end{equation}
Note that (\ref{scre}) implies that ${\bf q}$ belongs to the intersection of the weak convex hulls of all sets from $\F$. Thus, the proof of Theorem \ref{Helly} is reduced to showing (\ref{scre}). \footnote{The following argument is based on \cite[Proof of Proposition 4.2]{D}}

Suppose that (\ref{scre}) does not hold for some set $S = \{{\bf q}_1, \ldots, {\bf q}_n\}$. The fact that none of the points ${\bf q} = {\bf q}_i$ satisfy (\ref{scre}) implies that $S$ is in weakly convex position. Since there are only finitely many proper integer points in $\P$ we may also assume $S$ to be a minimal counterexample to (\ref{scre}), that is, the set $\wc(S)$ is minimal by inclusion among all possible counterexamples $S$.

Let $\mathcal Q$ be the set of all proper integer points ${\bf q}$ such that ${\bf q} = \sum_{i=1}^n \alpha_i {\bf q}_i$ for some coefficients $0\le \alpha_i< 1$ satisfying $\sum \alpha_i = 1$. 
Since $|S| = n > L = L(\P, \Lambda, \Omega)$, there is a convex combination ${\bf q} = \sum_{i=1}^n \alpha_i {\bf q}_i$ such that ${\bf q}$ is integral and all coefficients $\alpha_i$ are non-zero and less than $1$. This means that the set $\mathcal Q$ is non-empty.

\begin{claim}\label{projc}
If ${\bf q} \in \mathcal Q$ is a projection of ${\bf q}_j$ for some $j$ then ${\bf q}$ satisfies (\ref{scre}).
\end{claim}

\begin{proof}
Since ${\bf q}$ is a projection of ${\bf q}_j$, Proposition \ref{pf1} already implies that ${\bf q}$ belongs to all of the weak convex hulls in (\ref{scre}) except potentially $\wc(S\setminus \{{\bf q}_j\})$. 

Write ${\bf q} = \sum_{i=1}^n \alpha_i {\bf q}_i$ for some convex combination with $\alpha_i < 1$ and consider a point ${\bf q}'$ defined as follows:
$$
{\bf q}' = \sum_{i \neq j} \frac{\alpha_i}{1-\alpha_j} {\bf q}_i,
$$
Since $\alpha_j <1$, this is a well-defined convex combination, and we have ${\bf q} = \alpha_j {\bf q}_j + (1-\alpha_j) {\bf q}'$. This identity, together with the fact that ${\bf q}$ is a projection of ${\bf q}_j$, implies that ${\bf q}$ is a projection of ${\bf q}'$. But ${\bf q}' \in \c(S \setminus \{{\bf q}_j\})$, so (\ref{scre}) follows.
\end{proof}

By Claim \ref{projc} we may assume that $\mathcal Q$ does not contain points which are projections of some of ${\bf q}_i$-s. 
Now let ${\bf r} \in \mathcal Q$ be a point which belongs to the maximum number of sets $\wc(S \setminus \{{\bf q}_i\})$ and let $I \subset [n]$ denote the set of all such $i$. If $I = [n]$ then ${\bf q} = {\bf r}$ satisfies (\ref{scre}) so for the sake of contradiction we may assume that $I \neq [n]$.

\begin{claim}\label{cl53}
If for some $j$ we have ${\bf r} \not \in \wc(S \setminus \{{\bf q}_j\})$ then the set $S' = S \setminus \{{\bf q}_j\} \cup \{{\bf r}\}$ is in weakly convex position and $\wc(S')$ is a proper subset of $\wc(S)$.
\end{claim}

\begin{proof} Suppose that $S'$ is not weakly convex. Then for some $i \neq j$ we have 
$$
{\bf q}_i \in \wc(S \cup \{{\bf r}\} \setminus \{{\bf q}_j, {\bf q}_i\} ) \subset \wc(S \cup \{{\bf r}\} \setminus \{{\bf q}_i\}  ).
$$
Thus, by Proposition \ref{pf1}, there exists a point ${\bf q}'_i \in \conv(S \cup \{{\bf r}\} \setminus \{{\bf q}_i\}  )$ such that ${\bf q}_i$ is a projection of ${\bf q}'_i$. Therefore, there is a convex combination
$$
{\bf q}'_i = \sum_{t \neq i} \alpha_t {\bf q}_t + \beta {\bf r}
$$
for some non-negative $\alpha_t, \beta$. 
Note that $\beta > 0$ because the set $S$ is weakly convex. Since ${\bf r} \in \wc(S)$ there is ${\bf r}' \in \conv S$ such that ${\bf r}$ is a projection of ${\bf r}'$. Now consider the point
\begin{equation}\label{qeqn}
    {\bf q}''_i = \sum_{t \neq i} \alpha_t {\bf q}_t + \beta {\bf r}'.
\end{equation}
Then ${\bf q}_i$ is a projection of ${\bf q}_i''$ and the point ${\bf q}_i''$ lies in $\conv(S)$. Since the set $S$ is weakly convex there exists a linear function $\xi$ such that the value $\xi({\bf q}_i)$ is defined and $\xi({\bf q}_t) < \xi({\bf q}_i)$ holds for all $t \neq i$ for which this is defined. 

Since $\xi$ is defined on ${\bf q}_i$, it is also defined on ${\bf q}_i''$ and thus on ${\bf r}'$ and all ${\bf q}_t$ such that $\alpha_t \neq 0$. 
Since ${\bf r}' \in \conv(S)$, we have $\xi({\bf r}') \le \xi({\bf q}_i)$ with equality if and only if ${\bf r}' = {\bf q}_i$ (note that $\xi$ is defined on all elements of the convex combination expressing ${\bf r}'$). In the latter case we get that ${\bf r}$ is a projection of ${\bf q}_i$, contradicting our assumption that no element of $\mathcal Q$ is a projection of an element of $S$.
On the other hand, if $\xi({\bf r'}) < \xi({\bf q}_i)$, then (\ref{qeqn}) gives
\[
 \xi({\bf q}_i) = \xi({\bf q}_i'') = \sum_{t\neq i,~\alpha_t > 0} \alpha_t \xi({\bf q}_t) + \beta \xi({\bf r}') < \xi({\bf q}_i),
\]
a contradiction. 

Now we show that $\wc(S \setminus \{{\bf q}_j\} \cup \{{\bf r}\})$ is strictly contained in $\wc(S)$. In fact, ${\bf q}_j \not \in \wc(S \setminus \{{\bf q}_j\} \cup \{{\bf r}\})$ holds. Indeed, this follows from the argument above applied to $j = i$.
\end{proof}

For such a $j$, set $S' = S \setminus \{{\bf q}_j\} \cup \{{\bf r}\}$. We conclude that $S'$ is a weakly convex set of size $n$ with $\wc(S')$ strictly contained in $\wc(S)$. 
Thus, by the minimality of $S$ there exists a proper integer point ${\bf s}$ which belongs to the intersection:
\begin{equation}\label{inter}
{\bf s} \in \wc(S \setminus \{{\bf q}_j\}) \cap \bigcap_{i \neq  j} \wc(S \cup \{{\bf r}\} \setminus \{{\bf q}_j, {\bf q}_i\}).
\end{equation}
On the other hand, if $i \in I$ then ${\bf r} \in \wc(S \setminus \{{\bf q}_i\})$ and so
$$
{\bf s} \in \wc(S \cup \{{\bf r}\} \setminus \{{\bf q}_j, {\bf q}_i\}) \subset  \wc(S \setminus \{{\bf q}_i\}).
$$
We conclude that the point ${\bf s}$ belongs to $\wc(S \setminus \{{\bf q}_i\})$ for all $i \in I \cup \{j\}$, contradicting the choice of ${\bf r}$. Therefore, our assumption $I \neq [n]$ is false and there exists a point satisfying (\ref{scre}). This completes the proof of Helly's theorem.
\end{proof}

Recall that the usual Helly's theorem in $\R^d$ implies a centerpoint theorem. The convex flag analogue of this result will play a crucial role in the proof of Theorem \ref{main}.

\begin{cor}[Centerpoint Theorem]\label{central}
Let $(\P, \Lambda)$ be a convex flag with a set of proper points $\Omega$. Let $\{{\bf q}_1, \ldots, {\bf q}_n\}$ be a set of pairwise distinct proper integer points of $\P$ and let $\omega_1, \ldots, \omega_n$ be non-negative weights with $\sum \omega_i = \omega$. 
Then there exists a proper integer point $\bf q$ of $\P$ such that for every linear function $\xi$ with $\D_\xi \cap \D^{\bf q} \neq \emptyset$ we have
\begin{equation}\label{ceq}
\sum_{i: ~\xi({\bf q}_i) \ge \xi ({\bf q})} \omega_i \ge \frac{\omega}{L(\P, \Lambda, \Omega)},
\end{equation}
where the sum is taken over all $i$ such that $\D_{\xi} \cap \D^{{\bf q}_i}\neq \emptyset$ and $\xi ({\bf q}_i) \ge \xi ({\bf q})$.
\end{cor}

\begin{proof}
For a linear function $\xi$ such that $\D_\xi \cap \D^{\bf q} \neq \emptyset$ and a real number $\alpha$, let $S_{\xi, \alpha} \subset \{{\bf q}_1, \ldots, {\bf q}_n\}$ be the set of points ${\bf q}_i$ such that $\xi ({\bf q}_i) \le \alpha$ or the value $\xi ({\bf q}_i)$ is not defined (i.e. $\D_\xi \cap \D^{{\bf q}_i} = \emptyset$). Let $\F$ be the family of all sets $S_{\xi, \alpha}$ such that
\begin{equation}\label{ineq}
\sum_{{\bf q}_i \in S_{\xi, \alpha}} \omega_i > \omega \frac{L(\P, \Lambda)-1}{L(\P, \Lambda)}.
\end{equation}
By the pigeonhole principle, every collection of $L(\P, \Lambda)$ sets from $\F$ shares a common element ${\bf q}_i$ for some $i$, which is in particular a proper integer point of $\P$. Thus, by Theorem \ref{Helly}, there exists a proper integer point ${\bf q}$ which lies in the intersection of the weak convex hulls of all sets from $\F$. Let us check that the conclusion of Corollary \ref{central} holds for this point.
Let $\xi$ be a linear function satisfying $\D_{\xi} \cap \D^{{\bf q}}\neq \emptyset$. For every $\varepsilon > 0$, let $\alpha = \xi({\bf q}) -\varepsilon$. Then, by Definition \ref{defwc}, ${\bf q}$ does not belong to $\wc(S_{\xi, \alpha})$ since it is separated from this set by the linear function $\xi$. Thus, the set $S_{\xi, \alpha}$ does not belong to the family $\mathcal F$. This means that (\ref{ineq}) does not hold and so \begin{equation}\label{eqeps}
    \sum_{i:~ {\bf q}_i \not \in S_{\xi, \alpha}} \omega_i \ge \frac{\omega}{L(\P, \Lambda)}.
\end{equation}
For $\varepsilon$ small enough, (\ref{eqeps}) coincides with (\ref{ceq}) and so we are done.
\end{proof}

\subsection{Application to polytopes: proof of Theorem \ref{thm:cpt}}\label{polyhelly}

From Theorem \ref{central} we can derive a centerpoint theorem for integer points of polytopes mentioned in Section \ref{over}. For convenience, we restate the result here.

\maintheorem*


\begin{proof}
Denote by $S$ the support of $w$.
Let $\P = \P(P)$ be the convex flag corresponding to the polytope $P$ and let $\Lambda$ be a lattice on $\P$ defined as follows: for a face $\Gamma \subset P$ we let $\Lambda_\Gamma \subset \A_\Gamma$ be the minimal lattice containing the set $S \cap \Gamma$. 
Let $\Omega$ be the set of points ${\bf q}$ of $\P$ such that $\inf \D^{\bf q}$ is the minimum face of $P$ containing ${\bf q}$ (cf. Example \ref{exmpl}). 
Then the set of proper points of $\P$ is in one-to-one correspondence with the points of $P$. In particular, if ${\bf q}$ is a proper integer point of $(\P, \Lambda)$ and $q$ is the corresponding point in $P$, then $q$ belongs to the minimal lattice $\Lambda_\Gamma$, where $\Gamma$ is the minimal face containing $q$.

Thus, by Corollary \ref{central}, the statement of the theorem follows from the upper bound $L(\P, \Lambda, \Omega) \le L(d)$ on the Helly constant of $\P$.
We check this inequality using Definition \ref{arithm}.
For $n > L(d)$ let ${\bf q}_1, \ldots, {\bf q}_n$ be proper integer points of $(\P, \Lambda)$. 
If ${\bf q}_i = {\bf q}_j$ for some $i\neq j$ then ${\bf q} = \frac{1}{2}{\bf q}_i + \frac{1}{2} {\bf q}_j$ is an integer point and we are done. So we may assume that all points ${\bf q}_i$ are pairwise distinct. 

For each $i$ let $q_i$ be the point of $P$ corresponding to ${\bf q}_i$. If the set $\{q_1, \ldots, q_n\}$ is not in convex position, then there exists a convex combination of the form
$$
q_i = \sum_{j \neq i} \alpha_j q_j,
$$
for some $0\le \alpha_j < 1$. Since ${\bf q}_i$ is an integer point, this gives the desired convex combination.

Now assume that $q_1, \ldots, q_n$ are in convex position. Since the polytope  $Q = \conv\{q_1, \ldots, q_n\}$ has $n > L(d)$ vertices, there is an integer point $q \in Q$ which is not a vertex of $Q$. Write $q = \sum \alpha_i q_i$, with $0\le \alpha_i<1$ and $\sum \alpha_i = 1$, and let ${\bf q}$ be the corresponding proper point of the convex flag $\P$. Clearly, we have ${\bf q} = \sum \alpha_i {\bf q}_i$. We want to show that ${\bf q}$ is an integer point of $\P$.

Let $\Gamma$ and $\Gamma'$ be the minimal faces of $P$ and $Q$, respectively, containing the point $q$. In particular, we have $\D^{\bf q} = \P^\Gamma$. Then $\Gamma' \subset \Gamma$ and $q$ belongs to the minimal lattice $\Lambda'$ containing the set $S' = \{q_i~|~q_i \in \Gamma'\}$. The set $S'$ is contained in the lattice $\Lambda_\Gamma$ and so we have ${\bf q}_\Gamma \in \Lambda' \subset \Lambda_\Gamma$.
Therefore, the point ${\bf q}$ belongs to the lattice $\Lambda$ and we conclude that $L(\P, \Lambda, \Omega) \le L(d)$ as desired.
\end{proof}

\section{Flag Decomposition}\label{sfdl}

\subsection{The statement}\label{fdlst}

In this section we formulate and prove the Flag Decomposition Lemma, a structural result about arbitrary subsets of $\FF_p^d$. This result will play a crucial role in the proof of Theorem \ref{main}. To state it, we need some additional notation and terminology.

Recall that a convex flag $(\P, \Lambda)$ is a collection of data consisting of spaces $\A_x$, convex polytopes $P_x \subset \A_x$, lattices $\Lambda_x \subset \A_x$ and connecting polytope maps $\psi_{y, x}: \A_x \rightarrow \A_y$ for all $x \preceq y$.

Let $V= \FF_p^d$ be a vector space over $\FF_p$ for some prime $p>2$. Let $V^*$ denote the space of linear functions on $V$, including functions with constant term.
For a function $f: V \rightarrow \R_{\ge 0}$ and for a subset $S \subset V$ we denote $f(S) := \sum_{v \in S} f(v)$. 
For a linear function $\xi \in V^*$ on $V$ a $K$-slab $H(\xi, K)$ is the set of points $v \in V$ such that $\xi(v) \in [-K, K]$. 

\begin{defi}[Thinness and thickness]\label{tt}
A function $f: V \rightarrow \R_{\ge 0}$ is called $(K, \varepsilon)$-thin along a linear function $\xi \in V^*$ if 
$$
f(H(\xi, K)) \ge (1 - \varepsilon) f(V),
$$
and $f$ is called $(K, \varepsilon)$-thick along $\xi$ otherwise.
\end{defi}

The next definition relates convex flags with vector spaces over $\FF_p$. For a lattice $\Lambda$ and a prime $p$ we denote by $\Lambda / p\Lambda$ the set of equivalence classes of points of $\Lambda$ with respect to the relation $x \sim_p y$ if and only if $\frac{x + (p-1)y}{p} \in \Lambda$. Note that $\Lambda / p\Lambda$ has size precisely $p^{\dim \Lambda}$ and can be identified with an affine space over $\FF_p$.

\begin{defi}[$\FF_p$-Representation]
Let $(\P, \Lambda)$ be a convex flag and $V$ be a vector space over $\FF_p$. A {\em representation} $\varphi$ of the flag $(\P, \Lambda)$ in $V$ is a collection of affine subspaces $V_x \subset V$, for $x \in \P$, and affine surjective maps $\varphi_x: V_x \rightarrow \Lambda_x / p\Lambda_x$ such that for every $x \preceq y$ we have $V_x \subset V_y$ and $\varphi_y|_{V_x} = \psi_{y, x} \varphi_x$ on $V_x$.
\end{defi}

For brevity, we write $\varphi: V \rightarrow (\P, \Lambda)$ to mean that $\varphi$ is a representation of $(\P, \Lambda)$ in $V$. The corresponding affine subspaces and maps will always be denoted by $V_x$ and $\varphi_x$, possibly with some superscripts when we work with multiple representations at once.

For instance, if $(\P,\Lambda)$ is given by Example \ref{exmpl} with a polytope $P$ whose vertices are in the integer lattice $\Z^d$, then we can take $V = \FF_p^d$ and for $x\in \P$ let $V_x$ be the span of vertices of $P_x$ taken modulo $p$. Then, provided that $p$ is $P$-good, we can take $\varphi_x: V_x \to \Lambda_x/p\Lambda_x$ to be the identity map. 

If $(\P,\Lambda)$ is the binary tree of depth $d$ from Example \ref{ex1}, then we can construct many representations $V\to (\P,\Lambda)$ inductively. If $d=0$, then every surjective map $V\to \Z/p\Z$ is an $\FF_p$-representation in $(\{\emptyset\}, \Z)$. Now suppose that $d\ge 1$. We can write $\P = \{\emptyset\} \sqcup \P_0\sqcup \P_1$, where $\P_i$ is the set of strings starting with $i$.
We can view each $\P_i$ as a convex poset with induced lattice $\Lambda_i$. Let $\varphi_\emptyset: V \to \Z/p\Z$ be surjective, and let $V_i \subset \varphi_\emptyset^{-1}(i)$ be an arbitrary affine subspace for $i=0,1$. Let $\varphi_i:V_i \to (\P_i, \Lambda_i)$ be inductively constructed $\FF_p$-representations of depth $d-1$ dyadic trees. Then the maps $\varphi_\emptyset,\varphi_0,\varphi_1$ combine into a representation $\varphi:V\to (\P,\Lambda)$.

An affine basis of a lattice $\Lambda \subset \Q^d$ consists of an origin point $o \in \Lambda$ and a set of linearly independent vectors $e_1, \ldots, e_l$ such that $\Lambda = \langle o + \sum_i \lambda_i e_i~|~ \lambda_i \in \Z \rangle$.
Given an affine basis $E$ of a lattice $\Lambda \subset \Q^d$, define a lifting map $\gamma = \gamma_E: \Lambda / p\Lambda \rightarrow \Lambda$ as follows: for every equivalence class $[v] \in \Lambda/p\Lambda$, $\gamma_E$ sends $[v]$ to the unique vector $v'$ whose coordinates in the basis $E$ belong to the set $\{-\frac{p-1}{2}, \ldots, \frac{p-1}{2}\}$. If $E$ is an affine basis of $\Lambda$ and $q \in \Lambda$, then we denote by $\|q\|_{\infty, E}$ the largest absolute value among the coordinates of $q$ in the basis $E$. We extend these definitions to the setting of convex flags.

\begin{defi}[Basis]\label{Kb}
Let $\Lambda$ be a lattice on a convex flag $\P$. A basis $E$ of the lattice $\Lambda$ is a collection of affine bases $E_x$ of $\Lambda_x$ for $x \in \P$. Let $K: \P \rightarrow \N$ be a decreasing function, that is, for every $x \prec y$ we have $K(x) \ge K(y)$. We say that $E$ is $K$-bounded if for every $x \in \P$ and $q \in P_x \cap \Lambda_x$ we have $\|q\|_{\infty, E_x} \le K(x)$.
\end{defi}

\begin{defi}[Flag decomposition]\label{flagd}
Let $f: V\rightarrow \N$ be a function from an affine space over $\FF_p$ to the set of non-negative integers.
A {\em flag decomposition $\Phi$ of $f$} is the following collection of data:
\begin{itemize}
    \item A convex flag $(\P, \Lambda)$ and a representation $\varphi: V \rightarrow (\P, \Lambda)$,
    \item A collection of functions $f_x: V_x \rightarrow \N$, $x \in \P$, such that, extending each $f_x$ by zero outside $V_x$, $f^\Phi := \sum_{x \in \P} f_x$ satisfies $f^\Phi \le f$ pointwise.
    \item A basis $E$ of $\Lambda$ with the following property. For $x \in \P$ let $\hat f_x: \Lambda_x \rightarrow \N$ be the function defined for $q \in \Lambda_x$ as
    \[
    \hat f_x(q) = \begin{cases}
        \sum_{y \preceq x} f_y(\varphi^{-1}_x [q]), \quad \text{if }\|q\|_{\infty, E_x} \le \frac{p-1}{2},\\
        0, \quad \text{otherwise},
    \end{cases}
    \]
    where $[q] \in \Lambda_x / p\Lambda_x$ denotes the class of $q$.
    In this notation, we require that $P_x$ is the convex hull of all points $q \in \Lambda_x$ such that $\hat f_x(q) \neq 0$.
\end{itemize}
\end{defi}

For a decreasing function $K: \P \rightarrow \N$ we call
a flag decomposition {\em $K$-bounded} if the corresponding basis $E$ is $K$-bounded.

Roughly speaking, a flag decomposition $\Phi$ of $f$ is a way to express an arbitrary function $f: V \rightarrow \N$ as a sum of functions $f_x$, $x \in \P$, and an `error' term $(f-f^\Phi)$, which we want to be small. The functions $f_x$ are equipped with additional structure: $f_x$ is supported on a subspace $V_x \subset V$. After applying a surjective map $V_x \rightarrow \Lambda_x / p \Lambda_x$ and lifting to the integer lattice $\Lambda_x$, the support of $f_x$ (together with all $f_y$, $y \preceq x$) defines a convex polytope $P_x$. We want these polytopes to have bounded size, hence the notion of a $K$-bounded decomposition.

For $x \in \P$, let $f_{\preceq x}: V_x \rightarrow \N$ denote the sum $f_{\preceq x} = \sum_{y \preceq x} f_y$. In particular, we have $f^\Phi = f_{\preceq \sup \P}$.
For an integer point ${\bf q}$ of $\P$ we define $\hat f({\bf q})$ to be equal to $\hat f_x({\bf q}_x)$ where $x = \inf \D^{\bf q}$. For a subset $S \subset \Lambda_x$ we denote by $\hat f_x(S)$ the sum $\sum_{q \in S} \hat f_x(q)$. 

\begin{defi}[Minimal flag decomposition]
We say that a flag decomposition $\Phi$ of $f$ is {\em minimal} if for every $x \in \P$ the affine space $V_x$ is spanned by the support of $f_{\preceq x}$ and $\Lambda_x$ is the minimal lattice containing the support of $\hat f_x$.    
\end{defi}

In Section \ref{s2} we introduced a notion of proper points of a convex flag. Given a flag decomposition, there is a natural way to define a set of proper points.

\begin{defi}[Proper points]\label{prop}
Let $\Phi$ be a flag decomposition of a function $f$. Let $\Omega_0$ be the set of points ${\bf q}$ of $\P$ such that $\hat f({\bf q}) > 0$ and let $\Omega = \c (\Omega_0)$. The set $\Omega$ is called {\em the set of proper points of the flag decomposition $\Phi$}.
\end{defi}


Let $\P$ be a convex flag with a set of proper points $\Omega$. 
Let $x \in \P$ and $\Gamma$ be a face of $P_x$. Define an element $x_\Gamma \in \P$ as follows:
\begin{equation}\label{formula}
    x_{\Gamma} := \sup_{{\bf q}: ~{\bf q}_x \in \Gamma} \inf \D^{\bf q},
\end{equation}
where the supremum is taken over all proper points ${\bf q}$ which are defined over $x$ such that ${\bf q}_x \in \Gamma$. In particular, we have $x_\Gamma \preceq x$.

\begin{defi}[Realized face]
    In the setup above, we say that a face $\Gamma \subset P_x$ is {\em realized} if $ \psi_{x, x_\Gamma} (P_{x_\Gamma}) \subset \Gamma$.
\end{defi}

Essentially, this definition means that the face $\Gamma$ is properly represented by the convex flag $(\P, \Lambda)$. For example, we could start with a trivial convex flag $(\P,\Lambda)$ where $\P = \{x_0\}$, $\Lambda=\Z^d$, and $P_{x_0} \subset \Q^d$ is an arbitrary polytope. Then, for every proper face $\Gamma$ of $P_{x_0}$, we have $x_\Gamma = x_0$ and so $\psi_{x_0, x_\Gamma}(P_{x_\Gamma}) = P_{x_0} \not\subset \Gamma$. Thus, the convex flag `does not know' about the faces of $P_{x_0}$. In the following sections, we will introduce an operation on convex flags which adds a new element $x_\Gamma$ to the poset $\P$ and makes $\Gamma$ realized in the new convex flag. 

In Example \ref{exmpl} where $(\P,\Lambda)$ is constructed from a convex polytope $P$, all faces $\Gamma$ of $P$ are realized by design (and moreover we even have a stronger property $P_{x_\Gamma} = \Gamma$). In applications, we will be able to guarantee that all `relevant' faces that we consider are realized. 

Note that the polytope $P_{x_\Gamma}$ corresponding to a face $\Gamma\subset P_x$ need not coincide with $\Gamma$. Namely, consider Example \ref{ex1} with the set of proper points defined as the convex hull of all vertices of polytopes $P_s$ with $|s| = d$. 
Let $\Gamma_0 = \{0\}, \Gamma_1 = \{1\}$ be the faces of $P_{s} = [0,1]$ for some string $s \in \P$ with $|s|\le d-1$. Then we have $x_{\Gamma_0} = s0$ and $x_{\Gamma_1} = s1$ and both faces $\Gamma_0$ and $\Gamma_1$ are realized. 

Similarly, consider Example \ref{ex2} with proper points defined in an analogous way. Then all faces of $P_a$ are realized in $\P$ using elements $b_1, \ldots,b_n$ for edges $E_1, \ldots, E_n$ and elements $c_1, \ldots,c_n$ for vertices $v_1, \ldots, v_n$, respectively. On the other hand, only edges of $P_{b_i}$ are realized by $P_{c_i}, P_{c_{i-1}}$, and none of the vertices of $P_{c_i}$ is realized.

\begin{defi}[Reduced element]
    In the setup above,  we say that an element $x \in \P$ is {\em reduced} if there exists a proper point ${\bf q}$ such that $\inf \D^{\bf q} = x$.
\end{defi}

It is easy to see that $x$ is reduced if and only if $x_{P_x} = x$.
We say that a convex flag $(\P, \Omega)$ is {\em reduced} if every element $x\in \P$ is reduced. Similarly, a flag decomposition $\Phi$ is reduced if the corresponding convex flag $(\P,\Omega)$ is reduced.
Since we primarily care about proper points of a convex flag, non-reduced elements do not give us any useful information. In the following section we show that one can simply remove these elements from $\P$ and maintain all desirable properties. 

\begin{defi}[Large face]\label{largef}
Let $\Phi$ be a flag decomposition, $\varepsilon>0$ and $x \in \P$. A face $\Gamma \subset P_x$ is called $\varepsilon$-large if $\hat f_x(\Gamma \cap \Lambda_x) \ge \varepsilon f^\Phi(V)$ and for every proper face $\Gamma' \subset \Gamma$ we have $\hat f_x(\Gamma' \cap \Lambda_x) \le (1-\varepsilon) \hat f_x(\Gamma \cap \Lambda_x)$. 

An element $x \in \P$ is called $\varepsilon$-large if $\hat f_x(P_x \cap \Lambda_x) \ge \varepsilon f^\Phi(V)$. Note that the fact that $x$ is $\varepsilon$-large does not imply that $P_x$ is an $\varepsilon$-large face.
\end{defi}

\begin{defi}[Gap function]
    For a flag decomposition $\Phi$ and a reduced element $x \in \P$ define the {\em gap function} $G(x)$ of $x$ to be the minimum of non-zero values of the function $\hat f_x(q)$ over all $q \in \Lambda_x$. 
\end{defi}

With these definitions we can now state the properties which we would like our flag decomposition to satisfy. 

\begin{defi}[Complete element]
Let $\Phi$ be a flag decomposition of $f$, $\delta > 0$, $T: \P \rightarrow \N$ and let $x \in \P$. Then $x$ is called $(T, \delta)$-{\em complete} if for every linear function $\xi \in V_x^*$ that is not constant on fibers of $\varphi_x$, the function $f_{\preceq x}$ is $(T(x), \delta)$-thick along $\xi$.
\end{defi}

To illustrate this definition, suppose that we have some element $x \in \P$ which is not $(T,\delta)$-complete for some bounded number $T$ and fairly small $\delta>0$. Furthermore, let us suppose that $(\P,\Lambda)$ is $K$-bounded for some bounded function $K$. 
This means that the support of the function $f_{\preceq x}$ is essentially contained in a strip $H(\xi,T(x)) \subset V_x$. Consider the combined map $\tilde \varphi_x=(\varphi_x, \xi): V_x \to \Lambda_x/p\Lambda_x \times \Z/p\Z$. Then, after small pruning, the support of $f_{\preceq x}$ is mapped by $\tilde \varphi_x$ to a bounded box in $\Lambda_x \times \Z$ of size at most $\max(K(x), T(x))$. We can then add a new element $\tilde x$ to $\P$ and define a new polytope $P_{\tilde x}$ by taking the convex hull of this image. In other words, we constructed a larger flag decomposition which takes into account the information that $\operatorname{supp}f_{\preceq x}$ lies in the strip $H(\xi,T(x))$. (This construction will be made precise in the following sections.)

So an element $x\in \P$ is complete if the $\FF_p$-representation $V\to (\P,\Lambda)$ `captures' the bounded part of the support of $f_{\preceq x}$. 

\begin{defi}[Complete decomposition] \label{cmp}
Let $\varepsilon, \delta > 0$, let $\Phi$ be a flag decomposition and let $T:\P\rightarrow \N$ be a function. A flag decomposition $\Phi$ is called $(T, \varepsilon, \delta)$-{\rm complete} if the following conditions hold:
\begin{itemize}
    \item $\Phi$ is minimal and reduced,
    \item every $\varepsilon$-large element $x \in \P$ is $(T(x), \delta)$-complete,
    \item for every $x\in \P$, every $\varepsilon$-large face $\Gamma \subset P_x$ is realized.
\end{itemize}
\end{defi}

Now we are ready to formulate the main result of this section. We say that a function $g:\N\rightarrow \N$ is {\em growing} if $g(n) > n$ for all $n\in \N$.

\begin{theorem}[Flag Decomposition Lemma]\label{fdl}
Fix $d \ge 0$, let $\varepsilon > 0$ and let $g: \N \rightarrow \N$ be a growing function. Then there are constants $p_0(d, \varepsilon, g)$ and $\delta \gg_{d, \varepsilon} 1$ such that the following holds for all primes $p>p_0(d, \varepsilon, g)$.

Let $V = \FF_p^d$ and let $f: V \rightarrow \N$ be an arbitrary function. Then there exists a flag decomposition $\Phi$ of $f$ and functions $T, K: \P \rightarrow \N$, with $K$ decreasing, such that the following holds:
\begin{itemize}
    \item $\Phi$ is $K$-bounded and $(T, \varepsilon, \delta)$-complete,
    \item for all $x \in \P$ we have $T(x) \ge g(K(x))$, $K(x) \ll_{g, d, \varepsilon} 1$ and $G(x) \ge \delta^3 K(x)^{-d} f(V)$ (where $G(x)$ denotes the gap function),
    \item we have $f^{\Phi}(\FF_p^d) \ge (1-\varepsilon) f(V)$ and $|\P| \ll_{d, \varepsilon} 1$.
\end{itemize}
\end{theorem}

In Sections \ref{clp} and \ref{fdl2}, we introduce several operations on flag decompositions. We apply them in Section \ref{s3} to prove Theorem \ref{fdl}. Later in the paper, we will apply Theorem \ref{fdl} as a black-box, so the content of Sections \ref{clp}-\ref{s3} will not be needed outside Section \ref{sfdl}.

\subsection{Clean-up lemmas}\label{clp}

This section contains simple operations on flag decompositions which allow us to make them reduced, make the gap function separated from 0 and optimize the lattice $\Lambda$. The completeness properties of decompositions will be preserved under these operations. 

In what follows, we will work with multiple flag decompositions of the same function at once. To avoid notational clutter, we will use the following convention: all objects related to a flag decomposition will be denoted by the same letters, and superscripts will be added to distinguish these objects between different decompositions. For instance, the convex flag corresponding to a flag decomposition $\Phi'$ will be denoted by $(\P',\Lambda')$, and similarly for other objects.

Recall that we fix $V = \FF_p^d$ and $f: V\to \R_{\ge 0}$. 
For convenience let us recall that a flag decomposition $\Phi$ of $f$ consists of the following data:
\begin{itemize}
    \item A convex poset $\P$, a set $\Omega$ of proper points of $\P$, functions $T, K: \P \rightarrow \N$, 
    \item For $y \preceq x$, we have a map $\psi_{x, y}: \A_y \rightarrow \A_x$,
    \item For $x \in \P$ we have a space $\A_x$, a polytope $P_x \subset \A_x$, a lattice $\Lambda_x \subset \A_x$, an affine basis $E_x$ of $\Lambda_x$, 
    \item For $x\in \P$ we have a subspace $V_x \subset V$, a surjective map $\varphi_x: V_x \rightarrow \Lambda_x /p\Lambda_x$ and a function $f_x: V_x \rightarrow \N$.
\end{itemize}

For a flag decomposition $\Phi$ of $f$, let $\P^{\text{red}}$ be the set of reduced elements of $\P$, that is, the elements $x\in \P$ for which there exists a proper point ${\bf q}$ with $\inf \D^{\bf q} =x$. Equivalently, the set $\mathcal Q$ of all points ${\bf q}$ supported on $x$ satisfies $\bigcap_{{\bf q}\in \mathcal Q} \D^{\bf q} = \P^x$.


\begin{claim}
The poset $\P^{\text{red}} \subset \P$ is convex.
\end{claim}

\begin{proof}
We need to check that for every $x, y \in \P^{\text{red}}$ the set $\{x, y\}$ has a supremum in $\P^{\text{red}}$. Let $z = \sup(x, y) \in \P$ and define $z' = z_{P_z}$, that is,
\begin{equation*}
    z' = \sup_{{\bf q}:~ z \in \D^{\bf q}} \inf \D^{\bf q},
\end{equation*}
where the supremum is taken over all proper points $\bf q$ which are defined on the element $z$. Any point $\bf q$ which is supported on $x$ or $y$ is also supported on $z$ and so we have $x_{P_x}, y_{P_y} \preceq z'$. Since $x = x_{P_x}$ and $y = y_{P_y}$ this implies that $z'$ is an upper bound for $\{x, y\}$. But $z'\preceq z$ and so we must have $z' = z$ and hence $z \in \P^{\text{red}}$. This shows that $\P^{\text{red}}$ is a convex poset.
\end{proof}

Now we can restrict all the data of the flag decomposition $\Phi$ to the convex subposet $\P^{\text{red}} \subset \P$. We claim that the resulting collection of data $\Phi^{\text{red}}$ is a flag decomposition of $f$. The transition from $\Phi$ to $\Phi^{\text{red}}$ does not affect the numerical parameters of the decomposition but makes it reduced.

\begin{lemma}[Reduced decomposition]\label{straig}
In the notation above, $\Phi^{\text{red}}$ is a flag decomposition of $f$ such that:
\begin{itemize}
    \item $\Phi^{\text{red}}$ is reduced and $K$-bounded, and $f^{\Phi^{\text{red}}} = f^{\Phi}$,
    \item If $x \in \P^{\text{red}}$ is a $(T, \delta)$-complete element in $\Phi$ then $x$ is $(T, \delta)$-complete in $\Phi^{\text{red}}$,
    \item Let $x \in \P^{\text{red}}$ and let $\Gamma \subset P_x$ be a face. Then $x$ or $\Gamma$ is $\varepsilon$-large in $\Phi$ if and only if it is $\varepsilon$-large in $\Phi^{\text{red}}$. If $\Gamma$ is realized in $\Phi$, then $\Gamma$ is realized in $\Phi^{\text{red}}$,
    \item For every $x \in \P^{\text{red}}$ we have $G^{\text{red}}(x) = G(x)$.
\end{itemize}
\end{lemma}

\begin{proof}
Note that if $x \not \in \P^{\text{red}}$ then there are no proper points ${\bf q}$ of $\P$ with $\inf\D^{\bf q} = x$. By the definition of the set of proper points of $\Phi$ we get that $\hat f_x(q) = 0$ for every $q \in \Lambda_x$ and so the function $f_x$ is 0 at every point of $V_x$. This implies that $f^{\Phi^{\text{red}}} = f^{\Phi}$.

Note that all properties of $\Phi$ which depend only on structures associated with a particular element $x \in \P$ hold automatically for $\Phi^{\text{red}}$. In particular, $(T, \delta)$-completeness is a property of the function $f_{\preceq x}$, which coincides on both $\Phi$ and $\Phi^{\text{red}}$. Similarly, one checks the properties of being $\varepsilon$-large, $K$-bounded and the equality $G^{\text{red}}(x) = G(x)$. 

It remains to check that if $\Gamma \subset P_x$ is realized in $\Phi$ then it is realized in $\Phi^{\text{red}}$. Note that $x_\Gamma$ is a reduced element: if a point ${\bf q}$ is supported on $x_\Gamma$ then since $\psi_{x, x_{\Gamma}} P_{x_\Gamma} \subset \Gamma$, we get ${\bf q}_x \in \Gamma$. Thus, a point ${\bf q}$ is supported on $x_\Gamma$ if and only if ${\bf q}_x \in \Gamma$. Thus, the supremums in the definitions of elements $x_{P_{x_\Gamma}}$ and $x_{\Gamma}$ are taken over the same set of points which implies $x_{P_{x_\Gamma}} = x_{\Gamma}$ and $x_{\Gamma}$ is reduced. So we have $x^{\text{red}}_\Gamma \preceq x_{\Gamma} \preceq x$ and $\psi_{x, x^{\text{red}}_\Gamma} P_{x^{\text{red}}_\Gamma} \subset \Gamma$. We conclude that $\Gamma$ is realized.
\end{proof}

The next lemma shows that one can always modify the functions $f_x$ a bit to make the gap function $G(x)$ separated from 0. The parameters of the flag decomposition do not change significantly after this operation. 

\begin{lemma}[Creating a large gap]\label{completion}
Let $\Phi$ be a reduced $K$-bounded flag decomposition of a function $f$. For every $\alpha > 0$ there exists a convex subposet $\P' \subset \P$ and a reduced $K$-bounded flag decomposition $\Phi'$ of $f$ with corresponding poset $\P'$ satisfying the following properties:
\begin{itemize}
    \item For every $x \in \P'$ the objects $V_x, \Lambda_x, E_x, \varphi_x$ are unchanged when passing from $\Phi$ to $\Phi'$ and we have $f'_x \le f_x$ pointwise and $P'_x \subset P_x$,
    \item For every $x \in \P'$ we have $G'(x) \ge \alpha (2 K(x)+1)^{-\dim V} |\P|^{-1} f^{\Phi}(V)$, where $G'$ denotes the gap function of $\Phi'$,
    \item $f^{\Phi'}(V) \ge (1-\alpha) f^\Phi(V)$,
    \item If an element $x \in \P'$ is $\varepsilon$-large and $(T, \delta)$-complete in $\Phi$ for some $T, \varepsilon, \delta$, then $x$ is $(T, \delta - \frac{\alpha}{\varepsilon})$-complete in $\Phi'$. 
    \item If $x \in \P'$ and a face $\Gamma \subset P_x$ is realized in $\Phi$ and $\Gamma' = \Gamma \cap P'_x$ is non-empty then $\Gamma'$ is realized in $\Phi'$.
\end{itemize}

\end{lemma}

\begin{proof}
Initially set $f_x' = f_x$ for all $x \in \P$ and perform the following procedure to the collection of functions $\{f_x'\}$. Suppose that there exists $x \in \P$ and a point $q \in \Lambda_x$ such that
\begin{equation}\label{thing}
0 < \hat f'_x(q) \le \alpha (2K(x)+1)^{-\dim V}|\P|^{-1} f^\Phi(V),
\end{equation}
where, as in Definition \ref{flagd}, $\hat f'_x(q)$ denotes the sum of $f'_y(\varphi_x^{-1}[q])$ over all $y \preceq x$ if $\|q\|_{\infty, E_x} \le \frac{p-1}{2}$ and 0 otherwise (note that if $\hat f'_x(q)>0$ then $\hat f_x(q)>0$ and we automatically fall into the former case).
In this case, for each $y \preceq x$ define a new function $f''_y$:
\[
f''_y(v) = \begin{cases}
f'_y(v), ~ \varphi_x(v) \neq [q],\\
0, ~\varphi_x(v) = [q],
\end{cases}
\]
where $[q]$ denotes the class of the point $q$ in $\Lambda_x/p\Lambda_x$. Replace $f'_y$ by the function $f''_y$ and repeat the step until there are no $x \in \P$ and $q \in \Lambda_x$ satisfying the inequality above. Note that each pair $(x, q)$ can appear at most once so this process eventually terminates at some collection of functions $f'_x$, $x \in \P$. First, observe that the functions $\hat f'_x$ defined as above satisfy the desired gap condition. It remains to define the corresponding flag decomposition with functions $f'_x$ and verify the rest of the properties of the lemma.

With functions $\hat f'_x$ already defined for all $x\in \P$, let $P_x'$ be the convex hull of the support of $\hat f'_x$. Then clearly $P_x' \subset P_x$ and $f'_x \le f_x$ pointwise. It is clear from the definition that, for every $y \preceq x$, we have $\psi_{x, y} P'_y \subset P'_x$, so this defines a convex flag $\P'$ on the same poset $\P$. By keeping $\Lambda_x, E_x$ unchanged, we get a lattice $\Lambda'$ on $\P'$ and $(\P', \Lambda')$ is $K$-bounded. Using the same functions $\varphi_x$ and spaces $V_x$, we can thus define a flag decomposition $\Phi'$ of $f$.

Note that each step of the procedure decreases $f^{\Phi'}(V)$ by at most the right hand side of (\ref{thing}) and since there are at most $(2K(x)+1)^{\dim V}$ points in the relevant cube, each element $x \in \P$ decreases $f^{\Phi'}(V)$ by at most $\alpha |\P|^{-1} f^{\Phi}(V)$ in total over all steps of the procedure. Since there are $|\P|$ elements in total, we conclude that $f^{\Phi'}(V) \ge (1-\alpha) f^{\Phi}(V)$ after the last step.

Note that the set of proper points $\Omega'$ of $\Phi'$ is contained in the set of proper points $\Omega$ of $\Phi$. Thus, if $\Gamma$ is a face of $P_x$ such that $\Gamma' = P'_x \cap \Gamma$ is non-empty then the supremum in the definition of $x'_{\Gamma'}$ is taken over a subset of points from $\Omega$ supported on  $\Gamma$. So we have $x'_{\Gamma'} \preceq x_\Gamma$ which implies that
$$
\psi_{x, x'_{\Gamma'}} P'_{x'_{\Gamma'}} \subset \psi_{x, x_\Gamma} P_{x_\Gamma} \subset \Gamma,
$$
and so $\Gamma'$ is realized in $\Phi'$. 

Let $x \in \P$ be an $\varepsilon$-large and $(T, \delta)$-complete element for some $T,\varepsilon, \delta$ and let $\xi \in V_x^*$ be a linear function that is not constant on fibers of $\varphi_x$. Then we have 
$$
f_{\preceq x}(V_x \setminus H(\xi, T)) \ge \delta f_{\preceq x}(V_x),
$$
Since $x$ is $\varepsilon$-large, i.e. $f_{\preceq x}(V_x) \ge \varepsilon f^\Phi(V)$, we obtain
$$
f'_{\preceq x}(V_x \setminus H(\xi, T)) \ge f_{\preceq x}(V_x \setminus H(\xi, T)) - \alpha f^\Phi(V) \ge \left(\delta - \frac{\alpha}{\varepsilon}\right) f'_{\preceq x}(V_x)
$$
Thus, $x$ is $(T, \delta - \alpha/\varepsilon)$-complete in $\Phi'$. We have checked all the required properties except that $\Phi'$ is reduced. By Lemma \ref{straig}, we conclude that $\Phi'^{\text{red}}$ satisfies all properties of the lemma and we are done.
\end{proof}

Lastly, we show that, provided that $p$ is large enough, we can always modify a flag decomposition in order to make it minimal. Namely, given a flag decomposition $\Phi$ of $f$, for $x \in \P$ let $V_x^{\text{min}}$ be the minimal affine subspace of $V_x$ containing the support of the function $f_{\preceq x}$, and let $\Lambda_x^{\text{min}} \subset \Lambda_x$ be the minimal lattice containing the support of $\hat f_x$.

\begin{lemma}[Minimal decomposition]\label{minlat}
Let $\Phi$ be a $K$-bounded flag decomposition of a function $f$. If $p \gg_{K, d}1$, then there exists a minimal flag decomposition $\Phi^{\text{min}}$ of $f$ on the convex flag $\P$, with functions $f_x$, lattice $\Lambda^{\text{min}} = (\Lambda_x^{\text{min}})_{x\in\P}$ and subspaces $V_x^{\text{min}}$. Moreover, $\Phi^{\text{min}}$ is $K'$-bounded where $K': \P\rightarrow \N$ satisfies $K'(x) \le A_d(K(x))$ for all $x \in \P$ and some function $A_d$ depending only on $d = \dim V$.

In particular, $f^{\Phi^{\text{min}}}(V) = f^{\Phi}(V)$, the realized faces of $\Phi^{\text{min}}$ and $\Phi$ are the same and every $(T, \delta)$-complete element $x \in \P$ in $\Phi$ is $(T, \delta)$-complete in $\Phi^{\text{min}}$.
\end{lemma}

\begin{proof}
In order to define a flag decomposition $\Phi^{\text{min}}$ we need to construct maps $\varphi^{\text{min}}_x: V^{\text{min}}_x \rightarrow \Lambda_x^{\text{min}} / p \Lambda_x^{\text{min}}$ and define an affine basis $E_x^{\text{min}}$ of $\Lambda_x^{\text{min}}$.

We have a natural map $\iota_x: \Lambda_x^{\text{min}}/p\Lambda^{\text{min}}_x \rightarrow \Lambda_x/p\Lambda_x$ which sends an equivalence class $[q]$ in $\Lambda_x^{\text{min}}/p\Lambda^{\text{min}}_x$ to the unique class $\iota_x([q])$ in $\Lambda_x/p\Lambda_x$ which contains it set-theoretically. In the basis $E_x$, the sublattice $\Lambda^{\text{min}}_x$ is defined by a finite collection of points with coordinates bounded by $K(x)$. Thus, if $p$ is large enough compared to $K(x)$ and $\dim \Lambda_x \le d$, then the quotient $\Lambda_x/\Lambda_x^{\text{min}}$ has no $p$-torsion and so the map $\iota_x$ is an injective affine map over $\FF_p$.
Let $\varphi_x^{\text{min}}$ be the composition of $\varphi_x$ with $\iota_x^{-1}$. It then follows that $\varphi^{\text{min}}_x|_{V_y^{\text{min}}} = \psi_{x, y}\varphi^{\text{min}}_y$ holds for all $y\preceq x$. The map $\varphi^{\text{min}}_x$ is surjective since its image contains the support of the function $\hat f_x$. Since $\Lambda_x^{\text{min}} \subset \Lambda_x$ is defined by a collection of points with coordinates at most $K(x)$, one can find an affine basis $E^{\text{min}}_x$ of $\Lambda_x^{\text{min}}$ such that the support of $\hat f_x$ is $A_d(K(x))$-bounded with respect to $E^{\text{min}}_x$ for some function $A_d$ depending only on $d$.
Lastly, we need to check that the functions $\hat f_x$ and $\hat f^{\text{min}}_x$ coincide with this choice of the basis $E^{\text{min}}_{x}$. Indeed, this follows from the fact that the norms induced by $E^{\text{min}}_x$ and $E_x|_{\Lambda_x^{\text{min}}}$ are equivalent up to a constant depending only on $K(x)$ and $d$. Hence every point $q$ satisfying $\|q\|_{\infty, E_x}\le K(x)$ automatically satisfies $\|q\|_{\infty, E_x^{\text{min}}} \le \frac{p-1}{2}$.
This finishes the proof.
\end{proof}

\subsection{Refinements}\label{fdl2}

A flag decomposition whose existence is guaranteed by Theorem \ref{fdl} has the property that all large faces are realized and all large elements are complete. The constructions in this section allow us to refine a current flag decomposition $\Phi$ and ensure that a given face becomes realized or a given element becomes complete. At the same time, all properties of the decomposition change in a controllable manner. Iterating this process will eventually bring us to the flag decomposition in Theorem \ref{fdl}.

For a poset $\P$ and an element $x \in \P$, it will be convenient to define an auxiliary poset $\P[x]$ as follows. As a set, $\P[x] = \P \times \{1\} \cup \P_x \times \{0\}$, where we denote $\P_x = \{y: ~y \preceq x\}$. For $(y, \alpha), (y', \alpha') \in \P_x$, we have $(y, \alpha) \preceq (y', \alpha')$ if $y \preceq y'$ in $\P$ and $\alpha \le \alpha'$. Note that if $\P$ is a convex poset, then $\P[x]$ is convex as well.

Recall that a face $\Gamma \subset P_x$ is realized if $\psi_{x, x_\Gamma} P_{x_\Gamma} \subset \Gamma$. If a face $\Gamma$ is unrealized and we want to fix this, then we can try to add a new element $x'_\Gamma$ to the poset $\P$ that plays the role of $x_\Gamma$ in the new flag decomposition. More precisely, we have the following construction.

\begin{lemma}[Realized faces]\label{oper1}
Let $\Phi$ be a $K$-bounded flag decomposition of $f$, let $x \in \P$ and $\Gamma \subset P_x$ be a face. Suppose that $p \gg_{K, d} 1$.
Then there exists a flag decomposition $\Phi[x]$ on the poset $\P[x]$ with the following properties:
\begin{itemize}
    \item For every $(y, 1) \in \P[x]$ the objects $V_{(y, 1)}, \Lambda_{(y, 1)}, P_{(y, 1)}, E_{(y, 1)}, \varphi_{(y, 1)}$ coincide with the corresponding objects for $y \in \P$,
    \item We have $f^{\Phi[x]}(V) = f^\Phi(V)$ and $f_y = f_{(y, 1)} + f_{(y, 0)}$ for all $y \preceq x$, 
    \item $\Phi[x]$ is $K$-bounded, where $K: \P[x] \rightarrow \N$ satisfies $K((y, 1)) = K(y)$ for every $y \in \P$ and $K((y, 0)) \le A_d(K(y))$ for all $y \preceq x$ and some function $A_d: \N\rightarrow \N$ depending only on $d$,
    \item If $y \in \P$ is $(T, \delta)$-complete in $\Phi$ then $(y, 1) \in \P[x]$ is $(T, \delta)$-complete in $\Phi[x]$. If $\Gamma' \subset P_y$ is realized in $\Phi$ for some $y \in \P$ then $\Gamma' \subset P_{(y, 1)}$ is realized in  $\Phi[x]$,
    \item The face $\Gamma \subset P_{(x,1)}$ is realized in $\Phi[x]$. 
\end{itemize}
\end{lemma}

\begin{proof}
For each $y\preceq x$ we let $\Lambda_{(y, 0)} \subset \Lambda_y$ be the lattice obtained as the intersection of $\Lambda_y$ with the affine hull of the set of points $q \in \Lambda_y$ such that $\psi_{x, y}(q) \in \Gamma$ and $\hat f_y(q) > 0$. Since the quotient $\Lambda_y / \Lambda_{(y,0)}$ has no torsion, the natural map $\theta_y: \Lambda_{(y, 0)} / p \Lambda_{(y, 0)} \rightarrow \Lambda_y/p\Lambda_y$ of affine spaces over $\FF_p$ is injective. 
Let $V_{(y, 0)} \subset V_y$ be the preimage of $\Lambda_{(y, 0)}$ in $V_y$, that is, $V_{(y, 0)} = \varphi_y^{-1} \operatorname{Im}\theta_y$. Thus, we get a map $\varphi_{(y, 0)}: V_{(y, 0)} \rightarrow \Lambda_{(y, 0)} / p \Lambda_{(y, 0)}$ by restricting $\varphi_y$ to the subspace $V_{(y, 0)}$. For $y \preceq y' \preceq x$, it is clear that the map $\psi_{y', y}$ maps $\Lambda_{(y, 0)}$ into $\Lambda_{(y', 0)}$.

Note that the support of the function $\hat f_y$ is $K(y)$-bounded in the basis $E_y$ and the sublattice $\Lambda_{(y, 0)}$ is defined by a subset of the support of $\hat f_y$. Thus, by a compactness argument, there is some function $A_d$ and an affine basis $E_{(y, 0)}$ of $\Lambda_{(y, 0)}$ such that the restriction of $\hat f_y$ on $\Lambda_{(y, 0)}$ is $A_d(K(y))$-bounded in $E_{(y, 0)}$.
For $y \preceq x$ we define $f_{(y, 0)}$ as the restriction of $f_y$ on the subspace $V_{(y, 0)} \subset V_y$ and $f_{(y, 1)} = f_y - f_{(y, 0)}$. Since $p$ is assumed to be sufficiently large with respect to $K(y)$ and $d$, it follows that the restriction of $\hat f_y$ on $\Lambda_{(y, 0)}$ coincides with the function $\hat f_{(y, 0)}$ defined by the collection of functions $\{f_{(y', 0)}\}$. Indeed, we use the fact that the $l_\infty$-norms defined by the bases $E_y$ and $E_{(y,0)}$ are equivalent up to a constant depending only on $d$ and $K(y)$. Finally, let $P_{(y, 0)}$ be the convex hull of the support of $\hat f_{(y,0)}$.

Let $\A_{(y, 0)} = \A_{(y, 1)} = \A_y$ and let $\psi_{(y, 1), (y, 0)}$ be the identity map. 
We have now described all data required to define a flag decomposition $\Phi[x]$ of $f$ on the poset $\P[x]$ and it remains to verify the claimed properties of this construction. The first three bullet points in Lemma \ref{oper1} follow directly from the construction. 
Since this construction leaves all structures of points $y \in \P$ unchanged and $\hat f_{(y, 1)} = \hat f_{y}$, the completeness and realization properties for elements $(y, 1) \in \P[x]$ follow automatically, which verifies the fourth point.

Lastly, we check that $\Gamma$ is a realized face in $\Phi[x]$. Note that by construction we have $P_{(x, 0)} = \Gamma$ and so $\psi_{(x, 1), (x, 0)} P_{(x, 0)} \subset \Gamma$. Thus, it is enough to check that for every proper point ${\bf q}$ such that ${\bf q}_{(x, 1)} \in \Gamma$ we have $(x, 0) \in \D^{\bf q}$. By definition of the set of proper points of a flag decomposition, every proper point ${\bf q}$ is a convex combination ${\bf q} = \sum_{i=1}^n \alpha_i {\bf q}_i$
of integer points ${\bf q}_i$ such that $\hat f({\bf q}_i) > 0$. Let $(y_i, a_i) \in \P[x]$ be the element $\inf \D^{{\bf q}_i}$ and let $q_i = {\bf q}_{i,(y_i,a_i)} \in \Lambda_{(y_i, a_i)}$. Then we have 
\begin{equation}\label{nonsense}
\hat f_{(y_i, a_i)}(q_i) =\hat f({\bf q}_i) >0.
\end{equation}
On the other hand, if we have ${\bf q}_{(x,1)} \in \Gamma$ then for every $i=1, \ldots, n$ we get $(x,1) \in \D^{{\bf q}_i}$ and ${\bf q}_{i, (x, 1)} \in \Gamma$. That is, $\psi_{x, y_i}q_i \in \Gamma$ and so $q_i$ belongs to the lattice $\Lambda_{(y_i, 0)}$. This means that the preimage of the class $[q_i] \in \Lambda_{y_i} / p\Lambda_{y_i}$ under the function $\varphi_{y_i}$ belongs to the space $V_{(y_i, 0)}$. By definition, the function $f_{(y_i, 1)}$ has zero support on $V_{(y_i, 0)}$ and so we get 
$$
\hat f_{(y_i, 1)}(q_i) = f_{(y_i, 1)}(\varphi_{y_i}^{-1}[q_i]) = 0,
$$
which combined with (\ref{nonsense}) implies that $a_i = 0$ and, thus, $(x, 0) \in \D^{{\bf q}_i}$, as desired. We conclude that $\Gamma$ is a realized face in $\Phi[x]$.
\end{proof}

The second operation allows us to make a particular element $x\in \P$ $(T, \delta)$-complete. Recall that $x$ is $(T, \delta)$-complete if the function $f_{\preceq x}$ is $(T, \delta)$-thick along every linear function $\xi$ that is not constant on fibers of $\varphi_x$.
The basic idea behind this construction is that if $x$ is not thick along some $\xi$ then we can make $f_{\preceq x}$ supported on a strip of width $T$ by removing a few elements from its support and then use this strip to modify the flag decomposition.

\begin{prop}[Complete elements] \label{oper2}
Let $\Phi$ be a minimal $K$-bounded flag decomposition of $f$, let $x \in \P$, and fix a growing function $g: \N \rightarrow \N$ and a constant $\delta > 0$. Suppose that $p \gg_{d, K, g} 1$. Then there exists a flag decomposition $\Phi[x]$ on the poset $\P[x]$ with the following properties:
\begin{itemize}
    \item For every $y \in \P$ the objects $V_{(y, 1)}, \Lambda_{(y, 1)}, E_{(y, 1)}, \varphi_{(y, 1)}$ coincide with the corresponding objects for $y \in \P$ and we have $P_{(y, 1)} \subset P_y$,
    \item We have $f^{\Phi[x]}(V) \ge (1-3^{d+1}\delta )f^{\Phi}(V)$ and for every $y\preceq x$ we have $f_{(y, 1)} = 0$.
    \item $\Phi[x]$ is $K$-bounded, where $K:\P[x] \rightarrow \N$ satisfies $K((y, 1)) = K(y)$ and for $y \preceq x$ we have $K((y, 0)) \le g'(\max\{K(x), K(y)\})$ for some growing function $g':\N\to \N$ depending on $d$ and $g$,
    \item If $y\in \P$ is $(T, \alpha)$-complete in $\Phi$ then $(y, 1) \in \P[x]$ is $(T, \alpha')$-complete in $\Phi[x]$ where 
    $$
    \alpha' = \alpha - 3^{d+1}\delta \frac{f_{\preceq x}(V)}{f_{\preceq y}(V)}.
    $$
    \item If $\Gamma \subset P_y$ is realized in $\Phi$ for some $y \in \P$ and $\Gamma' = \Gamma\cap P_{(y, 1)}$ is non-empty then $\Gamma' \subset P_{(y, 1)}$ is realized in $\Phi[x]$, 
    \item The element $(x, 0)$ is $(g(K((x,0))), \delta)$-complete in $\Phi[x]$. 
\end{itemize}
\end{prop}

\begin{proof}
Let $W \subset V^*_x$ be the space of linear functions $\xi$ on $V_x$ which are constant on fibers of $\varphi_x$. In other words, every $\xi \in W$ has the form $\xi(v) = \eta\varphi_x(v)$ for a linear function $\eta: \Lambda_x/p\Lambda_x \rightarrow \FF_p$. Note that $W$ contains the 1-dimensional subspace of constant functions. 
Let $\tilde g: \N\to \N$ be a growing function such that $\tilde g(K) \ge g(A_d(K))$ holds for all $K\ge 1$ and some growing function $A_d$ which we will specify below. 
Recall that $f_{\preceq x} = \sum_{y\preceq x} f_y$ and consider a maximal sequence of linear functions $\xi_1, \ldots, \xi_k \in V_x^*$ such that:
\begin{itemize}
    \item The function $f_{\preceq x}$ is $(\tilde g^i(K(x)), 3^i\delta)$-thin along $\xi_i$,
    \item The dimension of the space $W' = \langle W, \xi_1,\ldots, \xi_k \rangle$ equals $\dim W + k$.
\end{itemize}
By definition, for every $\xi \not \in W'$, the function $f_{\preceq x}$ is $(\tilde g^{k+1}(K(x)), 3^{k+1}\delta)$-thick along $\xi$. Let $\Pi \subset V_x$ be the set of vectors $v$ such that $\xi_i(v) \in [-\tilde g^i(K(x)), \tilde g^i(K(x))]$, for all $i=1, \ldots, k$. For $y \preceq x$ define $f_{(y, 0)}$ to be the restriction of $f_y$ on $\Pi$ and let $f_{(y, 1)} = 0$. For $y\not \preceq x$ we let $f_{(y, 1)} = f_y$.

For $y \preceq x$ put $\A_{(y, 0)} = \A_y \times \Q^k$ and define a new lattice $\Lambda_{(y, 0)} \subset \Lambda_y \times \Z^k$ to be the minimal lattice containing all vectors of the form:
$$
(\tilde \varphi_y(v), \tilde \xi_1(v), \ldots, \tilde \xi_k(v)) \in \Lambda_y \times \Z^k,
$$
where $v \in V_y$ is such that $f_{\preceq (y, 0)}(v) > 0$, $\tilde \varphi_y(v) \in \Lambda_y$ denotes the unique lifting of $\varphi_y(v)$ such that $\| \tilde \varphi_y(v) \|_{\infty, E_y} \le K(y) < p/2$ and similarly $\tilde \xi_i(v) \in [-\tilde g^i(K(x)), \tilde g^i(K(x))]$ is the lifting of the element $\xi_i(v) \in \FF_p$. 

Now we can define a natural map $\varphi_{(y, 0)}: V_{y} \rightarrow \Lambda_{(y, 0)}/p\Lambda_{(y,0)}$ sending $v$ to the vector 
$$
(\varphi_y(v), \xi_1(v), \ldots, \xi_k(v)) \in \Lambda_y/p\Lambda_y \times \FF_p^k
$$ 
which then can be identified with an element of $\Lambda_{(y, 0)}/p\Lambda_{(y,0)}$. More precisely, we use the fact that $\Phi$ is minimal, so that the support of $f_{\preceq (y,0)} = f_{\preceq y}$ affinely spans $V_y$ and so $\varphi_{(y, 0)}$ can be first defined on the support of $f_{\preceq (y, 0)}$ in the obvious way and then extended by linearity on the whole space $V_y$.

Note that the lattice $\Lambda_{(y, 0)} \subset \Lambda_y \times \Z^k$ is defined by vectors with coordinates bounded by $K(y)$ and $\tilde g^k(K(x))$. Therefore, one can find an affine basis $E_{(y, 0)}$ in which the support of $\hat f_{(y, 0)}$ is $A_d(M)$-bounded, where $M=\max\{K(y), \tilde g^d(K(x))\}$ and $A_d$ is a function depending only on $d$. This is the function $A_d$ that we use in the definition of $\tilde g$. 
For $y' \preceq y \preceq x$, define a map $\psi_{(y, 0), (y',0)}: \A_{(y', 0)} \rightarrow \A_{(y, 0)}$ as $\psi_{y,y'}$ on the first coordinate and as the identity on the second coordinate. This maps the lattice $\Lambda_{(y', 0)}$ into $\Lambda_{(y, 0)}$.

For all $y\in \P$ we let $\Lambda_{(y, 1)} = \Lambda_y$, $\A_{(y, 1)} = \A_y$ and $E_{(y, 1)} = E_y$. Let $\psi_{(y, 1), (y, 0)}: \A_{(y, 0)} \rightarrow \A_{y}$ be the projection on the first coordinate.
Define the functions $\hat f_{(y, \alpha)}: \Lambda_{(y, \alpha)} \rightarrow \N$ by using the functions $f_{(y, 0)}$, $f_{(y, 1)}$ and bases $E_{(y, \alpha)}$ defined above and let $P_{(y,0)}$ and $P_{(y, 1)}$ be the convex hulls of the supports of these functions.

We claim that we constructed a flag decomposition $\Phi[x]$ of $f$ on the poset $\P[x]$. Indeed, the axioms of a convex flag and of a flag decomposition are satisfied by this construction. It remains to verify the properties of $\Phi[x]$ claimed in the statement of the lemma. The first bullet point follows directly from the construction. Since $f_{\preceq x}$ is $(\tilde g^{i}(K(x)), 3^i \delta)$-thin along $\xi_i$, we get
\begin{align*}
f^{\Phi}(V)- f^{\Phi[x]}(V) \le f_{\preceq x}(V_x \setminus \Pi) \le \sum_{i=1}^k f_{\preceq x}(V_x \setminus H(\xi_i, \tilde g^i(K(x)))) \le \\
\le \sum_{i=1}^k 3^i \delta f_{\preceq x}(V_x) \le \frac12 3^{k+1}\delta f_{\preceq x}(V_x),
\end{align*}
which implies the second point. For every $y \preceq x$, the polytope $P_{(y, 0)}$ is contained in the product $P_y \times [-\tilde g^k(K(x)), \tilde g^k(K(x))]^k$. Hence $P_{(y, 0)}$ is $A_d(\max \{K(y), \tilde g^k(K(x))\})$-bounded with respect to the basis $E_{(y, 0)}$. Note that this is an affine basis of the lattice $\Lambda_{(y,0)}$, so we incur the loss $A_d$ when changing coordinates. Thus, we get the third point with $g'$ defined in terms of the $d$-th iterate of $\tilde g$ and using the fact that $k\le d$.   

Let $y \in \P$ be a $(T, \alpha)$-complete element of $\Phi$. Let $\xi$ be a function not constant on fibers of $\varphi_{(y, 1)} = \varphi_y$. Then by definition we have $f_{\preceq y}(V_y \setminus H(\xi, T)) \ge \alpha f_{\preceq y}(V_y)$ which gives 
$$
f_{\preceq (y, 1)}(V_y \setminus H(\xi, T)) \ge \alpha f_{\preceq y}(V_y) - 3^{d+1}\delta f_{\preceq x}(V_x) \ge \alpha' f_{\preceq y}(V_y),
$$
and so $(y, 1)$ is $(T, \alpha')$-complete in $\Phi[x]$. 

Suppose that $\Gamma \subset P_y$ is a realized face in $\Phi$ for some $y\in \P$ and that $\Gamma' =\Gamma \cap P_{(y, 1)}$ is non-empty. Let ${\bf q}$ be a proper point of $\Phi[x]$ such that ${\bf q}_{(y, 1)} \in \Gamma$. This point corresponds to a proper point ${\bf q}'$ of $\Phi$: define $\D^{{\bf q}'} = \{z \in \P ~|~ (z, 1) \in \D^{{\bf q}}\}$ and set ${\bf q}'_z = {\bf q}_{(z, 1)}$ for every $z \in \D^{{\bf q}'}$. One can check that ${\bf q}'$ is indeed a proper point of $(\P, \Omega)$. By definition, we have $y_\Gamma \in \D^{{\bf q}'}$ and $\psi_{y,y_{\Gamma}}(P_{y_\Gamma}) \subset \Gamma$. It follows that $(y_\Gamma,1) \in \D^{{\bf q}}$ and 
\[
\psi_{(y,1), (y_\Gamma, 1)}(P_{(y_\Gamma,1)}) \subset \psi_{(y,1), (y_\Gamma, 1)}(P_{y_\Gamma}) \subset \Gamma.
\]
Since $\psi_{(y,1), (y_\Gamma, 1)}$ is a map of polytopes, we also have $\psi_{(y,1), (y_\Gamma, 1)}(P_{(y_\Gamma,1)}) \subset P_{(y,1)}$. Hence this image is contained in $\Gamma'$. Clearly, $(y, 1)_{\Gamma'} \preceq (y_\Gamma,1)$, and $\psi_{(y,1), (y, 1)_{\Gamma'}}(P_{(y,1)_\Gamma'}) \subset \psi_{(y,1), (y_\Gamma, 1)}(P_{(y_\Gamma,1)})$. Thus, $\Gamma'$ is a realized face in $\Phi[x]$, as desired.

Let us check that $(x, 0)$ is $(\tilde g^{k+1}(K(x)), \delta)$-complete in $\Phi[x]$. Let $\xi \in V^*_{(x, 0)} = V^*_x$ be a linear function not constant on fibers of $\varphi_{(x, 0)}$. This is equivalent to the condition that $\xi \not \in W' = \langle W, \xi_1, \ldots, \xi_k\rangle$. Therefore, $f_{\preceq x}$ is $(\tilde g^{k+1}(K(x)), 3^{k+1}\delta)$-thick along $\xi$. Using the bound on $f^{\Phi}(V) - f^{\Phi[x]}(V)$, we get
\begin{align*}
f_{\preceq (x, 0)}(V_x \setminus H(\xi, \tilde g^{k+1}(K(x)))) \ge f_{\preceq x}(V_x \setminus H(\xi, \tilde g^{k+1}(K(x)))) - \frac12 3^{k+1}\delta f_{\preceq x}(V_x) \ge\\
\ge \frac12 3^{k+1} \delta f_{\preceq x}(V_x) \ge \delta f_{\preceq (x,0)}(V_x).
\end{align*}
By the choice of $\tilde g$, we have $\tilde g^{k+1}(K(x))  \ge g ( A_d(\tilde g^k(K(x))) ) \ge g(K(x,0)) $. Thus, we get the last bullet point of the proposition. This completes the proof.
\end{proof}

\subsection{Final preparations for the proof of Theorem \ref{fdl}}

In this section we collect some additional results needed in the proof of Theorem \ref{fdl}. Let $P \subset \Q^d$ be a polytope, $\mu$ be a finite measure on $P$ and $\varepsilon > 0$. A face $\Gamma \subset P$ is called $\varepsilon$-large with respect to $P$ and $\mu$ if $\mu(\Gamma) \ge \varepsilon \mu(P)$ and for every proper face $\Gamma' \subset \Gamma$ we have $\mu(\Gamma') < (1-\varepsilon)\mu(\Gamma)$. 

\begin{prop}\label{large}
    Let $P_1 \supset P_2 \supset \ldots \supset P_N$ be a sequence of polytopes in $\Q^d$, let $\mu$ be a finite measure on $\Q^d$ and let $\varepsilon >0$. Suppose that $\mu(P_N) \ge \varepsilon \mu(P_1)$.
    Suppose that for $i=1, \ldots, N$, $\Gamma_i \subset P_i$ is an $\varepsilon$-large face with respect to $P_i$ and $\mu$ and that for every $1\le i < j \le N$ we have $\Gamma_i \cap P_j \neq \Gamma_j$. Then we have $N \le (\varepsilon^{-3}+d+2)^{d+2}$.
\end{prop}

\begin{proof}
For an integer $t\ge 1$, let $N_t$ be the maximum number $n$ such that there exists a set of $t$ affinely independent points of $\Q^d$ which is contained in at least $N_t$ faces $\Gamma_i$. Note that
$$
\sum_{i=1}^N \mu(\Gamma_i) \ge \sum_{i=1}^N \varepsilon\mu(P_i)\ge \varepsilon^2 N \mu(P_1), 
$$
so by the pigeonhole principle, there exists a point $q \in P_1$ which belongs to at least $\lfloor \varepsilon^2 N\rfloor$ faces $\Gamma_i$. In particular, we get $N_1 \ge \lfloor\varepsilon^2 N\rfloor$. On the other hand, since there are no $d+2$ affinely independent points in $\Q^d$, we trivially have $N_{d+2} = 0$. Now let $1\le t \le d+1$ be arbitrary and consider a $t$-element affinely independent set $S= \{q_1, \ldots, q_t\}$ which is contained in $N_t$ faces $\Gamma_{i_1}, \ldots, \Gamma_{i_{N_t}}$ for some indices $i_1 < \ldots < i_{N_t}$. For each $j=1, \ldots, N_t$, let $\Gamma'_{i_j}$ be the minimal face of $P_{i_j}$ containing $S$. For every $j\le j'$ we have $\Gamma'_{i_{j'}} \subset \Gamma_{i_j} \cap P_{i_{j'}}$.

Note that if for some $j'$ we have $\Gamma'_{i_{j'}} = \Gamma_{i_{j'}}$ then for every $j < j'$ we get $\Gamma_{i_{j'}} \subset \Gamma'_{i_{j'}} \subset \Gamma_{i_j} \cap P_{i_{j'}}$. By the assumption, we have $\Gamma_{i_{j'}} \neq \Gamma_{i_j} \cap P_{i_{j'}}$, therefore, $\Gamma_{i_{j'}}$ is a proper face in $\Gamma_{i_j} \cap P_{i_{j'}}$ and, in particular, we get that $\dim \Gamma_{i_{j'}} < \dim \Gamma_{i_j}$. Thus, there are at most $d+1$ indices $j \in [N_t]$ such that $\Gamma'_{i_j} = \Gamma_{i_j}$. Denote by $J \subset [N_t]$ the set of all indices $j$ such that $\Gamma'_{i_j} \neq \Gamma_{i_j}$. 

Let $j \in J$ and note that since $\Gamma_{i_j}$ is $\varepsilon$-large, we have 
$$
\mu(\Gamma_{i_j} \setminus \Gamma'_{i_j}) \ge \varepsilon \mu(\Gamma_{i_j}) \ge \varepsilon^2 \mu(P_{i_j}) \ge \varepsilon^3 \mu(P_1).
$$
Thus, $\sum_{j\in J} \mu(\Gamma_{i_j} \setminus \Gamma'_{i_j}) \ge \varepsilon^3 (N_t - d-1) \mu(P_1)$ and by the pigeonhole principle there exists a point $q$ belonging to at least $M = \lfloor \varepsilon^3 (N_t - d-1)\rfloor$ sets $\Gamma_{i_j} \setminus \Gamma'_{i_j}$, $j\in J$. Suppose that $M > 0$ and let $j$ be such an index. Then $q$ does not belong to the affine hull $V$ of the set $S$. Indeed, otherwise $q \in V \cap P_{i_j}$ and so $q$ lies in the minimal face of $P_{i_j}$ containing the set $S$, that is $q \in \Gamma'_{i_j}$, which is a contradiction. We conclude that the set $S' = S \cup \{q\}$ is affinely independent and is contained in $M$ faces $\Gamma_i$ for $1\le i\le N$. Thus, we get $N_{t+1} \ge M = \lfloor \varepsilon^3 (N_t - d-1)\rfloor$ which implies that
$$
\varepsilon^{-3} (N_{t+1}+1) + d+1 \ge N_t
$$
holds for all $t=1, \ldots, d+1$. Using $N_{d+2} = 0$ and chaining these inequalities together, we get an upper bound $N_1 \le (\varepsilon^{-3} + d + 2)^{d+1}$. Combined with the bound $N \le \varepsilon^{-2} (N_1+1)$ and some simplifications, this leads to the desired estimate on $N$. 
\end{proof}

Let $\Phi$ be a flag decomposition of a function $f: V=\FF_p^d \rightarrow \N$. For an element $x \in \P$ define a function $l_\Phi(x) \in [d]^2$, which we call the \emph{level} of $x$, to be the pair $(\codim V_x, \dim \Lambda_x)$. We view $[d]^2$ as a linearly ordered set with respect to the lexicographical order $\preceq_{lex}$. Note that for every $y \preceq x$ we always have $l_\Phi(y) \succeq_{lex} l_\Phi(x)$, that is, either $\codim V_y >  \codim V_x$ or $\dim V_y = \dim V_x$ and $\dim \Lambda_y \ge \dim \Lambda_x$. Moreover, we have $l_\Phi(x) = l_\Phi(y)$ only if $V_x = V_y$, lattices $\Lambda_x$ and $\Lambda_y$ have equal dimensions and $\psi_{x, y}$ is an injection.

\begin{obs}\label{level1}
Let $\Phi$ be a flag decomposition of $f$, let $x \in \P$ and let $\Gamma \subset P_x$ be a proper face of $P_x$. Let $\Phi[x]$ be a  flag decomposition given by Lemma \ref{oper1} applied to the face $\Gamma$. Then we have $l_{\Phi[x]}(y, 0) \succ_{lex} l_{\Phi}(x)$ for all $y \preceq x$.  Moreover, for every $y \in \P$ we have $l_{\Phi[x]}(y, 1) = l_\Phi(y)$.
\end{obs}

\begin{proof}
Note that the proof of Lemma \ref{oper1} in fact gives that $(x, 1)_\Gamma \preceq (x,0)$, $\psi_{(x, 1), (x, 0)} P_{(x,0)} \subset \Gamma$ and the image of $\Lambda_{(x,0)}$ is contained in the affine hull of $\Gamma$. The space $V_{(x,0)}$ is defined as the preimage of the lattice $\Lambda_{(x, 0)}$ obtained as the intersection of $\Lambda_x$ with the affine hull of $\Gamma$. If $\Gamma$ is a proper face then it follows that $\dim V_{(x,0)} < \dim V_x$. This implies that $l_{\Phi[x]}(y, 0) \succeq_{lex} l_{\Phi[x]}(x, 0) \succ_{lex} l_\Phi(x)$ for all $y \preceq x$.

The last assertion follows from the fact that $V_{(y, 1)} = V_y$ and $\Lambda_{(y, 1)} = \Lambda_y$ for $y\in \P$.
\end{proof}

\begin{obs}\label{level2}
Let $\Phi$ be a minimal flag decomposition of $f$, let $x \in \P$ and suppose that $x$ is not $(g(K(x)), \delta)$-complete in $\Phi$. Let $\Phi[x]$ be a flag decomposition given by Lemma \ref{oper2} applied to the element $x$ and the same $g, \delta$. Then we have $l_{\Phi[x]}(y, 0) \succ_{lex} l_{\Phi}(x)$ for all $y \preceq x$. Moreover, for every $y \in \P$ we have $l_{\Phi[x]}(y, 1) = l_\Phi(y)$.
\end{obs}

\begin{proof}
    Since $x$ is not $(g(K(x)),\delta)$-complete, in the proof of Lemma \ref{oper2} we have $k\ge 1$. 
    In the proof of Lemma \ref{oper2} we defined $\Lambda_{(x, 0)}$ as a minimal lattice in $\Lambda_x \times \Z^k$ containing all points of the form $(\tilde \varphi_x(v), \tilde\xi_1(v), \ldots, \tilde \xi_k(v))$ over all $v$ such that $f_{\preceq x}(v) >0$. By assumption, points $v$ with $f_{\preceq x}(v)>0$ span $V_x$ and, by construction, the map $(\varphi_x, \xi_1, \ldots, \xi_k): V_x \rightarrow \Lambda_x/p\Lambda_x \times \FF_p^k$ is surjective. It follows that the map 
    $$
    \Lambda_{(x, 0)}/p\Lambda_{(x, 0)} \rightarrow \Lambda_x/p\Lambda_x \times \FF_p^k
    $$
    is an isomorphism and, in particular, $\dim \Lambda_{(x, 0)} > \dim \Lambda_{x}$. Since $\dim V_{(x, 0)} = \dim V_x$ we conclude that $l_{\Phi[x]}((x,0)) \succ_{lex} l_{\Phi}(x)$.

    The last assertion follows from the fact that $V_{(y, 1)} = V_y$ and $\Lambda_{(y, 1)} = \Lambda_y$ for $y\in \P$.
\end{proof}

\begin{obs}\label{level3}
    Let $\Phi$ be a flag decomposition of $f$ and let $\Phi'$ be the result of applying one of Lemmas \ref{straig}, \ref{completion}, \ref{minlat}. Then for every $x \in \P'$ we have $l_{\Phi'}(x) \succeq_{lex} l_{\Phi}(x)$.
\end{obs}

\begin{proof}
    In Lemma \ref{straig}, we only remove elements from $\P$, so the value of $l_\Phi(x)$ is not affected for the remaining elements. Similarly, Lemma \ref{completion} does not affect $l_\Phi(x)$. In Lemma \ref{minlat}, we replace $\Lambda_x$ and $V_x$ by the minimal lattice and subspace, respectively, that contain the supports of $\hat f_x$ and $f_{\preceq x}$. We note that if $\dim \Lambda_x^{\text{min}} < \dim \Lambda_x$, then we also have $\dim V_{x}^{\text{min}} <  \dim V_x$. So either both dimensions stay the same, or $\dim V_x$ decreases. This implies that $l_{\Phi^{\text{min}}}(x) \succeq_{lex} l_{\Phi}(x)$ holds for all $x \in \P$.
\end{proof}

Lastly, we will use the following result from Ramsey theory.

\begin{claim}\label{konig}
For integers $N, k \ge 1$ let $\chi: [N] \rightarrow [k]$ be a coloring of the set consisting of the first $N$ natural numbers in $k$ colors. Let $h:\N \rightarrow \N$ be an arbitrary function. If $N \gg_{h, k} 1$ then for some $l \in [k]$ there exists an interval $J = [j_0, j_1] \subset [N]$ such that $\chi(j) \in [l, k]$ for every $j \in J$ and $\chi(j) = l$ for at least $h(j_0)$ elements $j \in J$. 
\end{claim}

\begin{proof}
By K\"onig's tree lemma, it is enough to prove the analogous statement with $[N]$ replaced by $\N$. If $\chi: \N \rightarrow [k]$ is an arbitrary coloring then we let $l$ be the least element of $[k]$ such that color $l$ appears in $\chi$ infinitely many times. Since all colors $l' < l$ appear only finitely many times there is some $j_0 \in \N$ such that $\chi(j) \ge l$ for every $j \ge j_0$. Now let $j_1 \ge j_0$ be an element such that the interval $[j_0, j_1]$ contains at least $h(j_0)$ elements $j$ such that $\chi(j) = l$. Then $J = [j_0, j_1]$ is the desired interval.
\end{proof}

\subsection{Proof of Flag Decomposition Lemma}\label{s3}

Now we turn to the proof of Theorem \ref{fdl}. Let $f: V = \FF_p^d \rightarrow \N$ be an arbitrary nonzero function, $\varepsilon >0$ and let $g$ be a growing function such that $p \gg_{d, \varepsilon, g} 1$. We are going to apply the lemmas from the previous sections repeatedly to build the desired flag decomposition. 

Let $\delta_0 >0$ be a sufficiently small constant depending on $d, \varepsilon$ and for $i\ge 0$ denote $\delta_i = 3^{-2d i} \delta_0$ and $\varepsilon_i = \varepsilon + 2^{-i} \varepsilon$.

\paragraph{Initialization.} Let $\Phi^0$ be a flag decomposition of $f$ defined in the following way. Let $\P^0 = \{x_0\}$ be a one-element poset, let $V^0_{x_0} = V$, $f^0_{x_0} = f$, let $\A^0_{x_0}$ be a 0-dimensional space and $\Lambda^0_{x_0} = \A^0_{x_0}$. Since $\Lambda^0_{x_0} / p\Lambda^0_{x_0}$ consists of a single element, there is a unique map $\varphi^0_{x_0}: V^0_{x_0} \rightarrow \Lambda^0_{x_0} / p\Lambda^0_{x_0}$. Finally, let $E_{x_0}^0$ be the only affine basis of $\A^0_{x_0}$ which consists of a single origin point and let $P_{x_0}^0 = \A^0_{x_0}$ be the 1-point polytope. This flag decomposition is trivially $K_0$-bounded with $K_0(x_0) = 1$. 

\paragraph{Step $i$.} Suppose that we are given a flag decomposition $\Phi^{i-1}$ of $f$. We construct a new flag decomposition $\Phi^i$ or finish the process according to the following cases. In each case we assume that all previous cases do not apply.
\begin{itemize}
    \item[(i)] If $\Phi^{i-1}$ is not reduced then let $\Phi^i = \Phi^{i-1, \text{red}}$ and proceed to step $i+1$. Otherwise, if $\Phi^{i-1}$ is not minimal, let $\Phi^i = \Phi^{i-1, \text{min}}$ and proceed to step $i+1$.
    Otherwise, if there exists $x \in \P^{i-1}$ such that $G(x) < \delta_i^2 \varepsilon K_{i-1}(x)^{-d} f(V)$, then apply Lemma \ref{completion} to $\Phi^{i-1}$ and $\alpha=\delta_i^2 \varepsilon$, let $\Phi^{i}$ be the resulting flag decomposition and proceed to step $i+1$.
    \item[(ii)] Suppose that there is an element $x \in \P^{i-1}$ and a face $\Gamma \subset P_x$ such that $\Gamma$ is $\varepsilon$-large and unrealized in $\Phi^{i-1}$. Choose such $x$ so that $l_{\Phi^{i-1}}(x)$ is minimal possible and apply Lemma \ref{oper1} to $\Phi^{i-1}$ with parameters $x$ and $\Gamma$. Let $\Phi^i$ be the resulting flag decomposition and proceed to step $i+1$.
    \item[(iii)] Suppose that the previous case does not apply and there exists an $\varepsilon$-large element $x \in \P^{i-1}$ which is not $(g(K_{i-1}(x)), \delta_i)$-complete. Choose such $x$ so that $l_{\Phi^{i-1}}(x)$ is minimal possible and apply Lemma \ref{oper2} to $\Phi^{i-1}$ with parameters $x$ and $g$, $\delta_i$. Let $\Phi^i$ be the resulting flag decomposition and proceed to step $i+1$.
    \item[(iv)] If none of the above applies, stop the procedure and let $\Phi = \Phi^{i-1}$.
\end{itemize}

We will show that the procedure stops in $N \ll_{d, \varepsilon} 1$ steps. Let us first see how this would imply Theorem \ref{fdl}. Indeed, let $\Phi$ denote the resulting flag decomposition of $f$ after the procedure stops in $N \ll_{d, \varepsilon} 1$ steps. Denote $\delta = \delta_N \gg_{d, \varepsilon}1$.  
Since (i) is not applicable, $\Phi$ is reduced, minimal and $G(x) \ge \delta^3 K(x)^{-d} f(V)$. Since (ii) is not applicable, every $\varepsilon$-large face of $\Phi$ is realized. Since (iii) is not applicable, every $\varepsilon$-large element $x \in \P$ is $(g(K(x)), \delta_N)$-complete. 
We conclude that $\Phi$ is $(T, \varepsilon, \delta)$-complete for some function $T:\P\rightarrow \N$ such that $T(x) \ge g(K(x))$ for all $x \in \P$. 

At each step, the number of elements of $\P$ increases by a factor of at most 2, so we have $|\P| = |\P_N| \le 2^N \ll_{d, \varepsilon} 1$. At each step in which (i) is applied, Lemma \ref{completion} gives $f^{\Phi^i}(V) \ge (1-\delta_i^2 \varepsilon)f^{\Phi^{i-1}}(V)$. At each step in which (iii) is applied, Lemma \ref{oper2} gives $f^{\Phi^i}(V) \ge (1-3^{d+1}\delta_i)f^{\Phi^{i-1}}(V)$. Finally, the value of $f^{\Phi^{i-1}}(V)$ does not change when we apply (ii), by Lemma \ref{oper1}. It is easy to see that with our choice of parameters we get
$$
f^{\Phi}(V) \ge f^{\Phi^{N}}(V) \ge (1-\varepsilon) f^{\Phi^0}(V) = (1-\varepsilon) f(V).
$$
At each step the function $K_i(x)$, which is responsible for the boundedness of the flag decomposition $\Phi^i$, gets modified in a controlled manner: either it stays the same or gets replaced with a new value $A_d(K(x))$ (as in Lemma \ref{minlat}). When new elements are added to $\P$ (in Lemma \ref{oper1} and Proposition \ref{oper2}), the function $K$ is upper bounded in terms of values of $K$ on old elements, some iteration of the function $g$ and some bounded function $A_d$ depending on $d$. Altogether, we have the estimate $K(x) \ll_{g, d,\varepsilon} 1$ for every $x \in \P$.

We checked all properties claimed in Theorem \ref{fdl} and it remains to show that the procedure above terminates in $N\ll_{d, \varepsilon}1$ steps.

Assume that the process did not stop in $N$ steps and let us arrive at a contradiction provided that $N$ is sufficiently large. Given an element $l \in [d]^2$ we denote $\bar l = (d+1) l_1 + l_2 \in [(d+1)^2]$ so that it defines an embedding of the lexicographical order on $[d]^2$ in $\N$.
Let $h: \N\rightarrow \N$ be a growing function depending on $d, \varepsilon$ which will be determined and define a coloring $\chi$ of the interval $[N]$ as follows:
\begin{itemize}
    \item If (i) is applied at step $i$ and `reduced elements' lemma was applied then we let $\chi(i) = 2(d+1)^2+2$,
    \item If (i) is applied at step $i$ and `minimal lattice' lemma was applied then we let $\chi(i) = 2(d+1)^2+1$,
    \item If (i) is applied at step $i$ and `large gaps' lemma was applied then we let $\chi(i) = 2(d+1)^2$,
    \item If (ii) is applied to an element $x$ at step $i$ then we let $\chi(i) = 2 \bar l_{\Phi^{i-1}}(x) + 1$,
    \item If (iii) is applied to an element $x$ at step $i$ then we let $\chi(i) = 2 \bar l_{\Phi^{i-1}}(x)$.
\end{itemize}

If $N$ is large enough compared to $d$ and $h$, then Claim \ref{konig} implies that there is some $c \in \N$ and $1\le j_0 \le j_1 \le N$ such that $\chi(j) \ge c$ for every $j \in [j_0, j_1]$ and $\chi(j) = c$ for at least $h(j_0)$ elements $j \in [j_0, j_1]$. For each value of $c$ we will show that $h(j_0)$ cannot be arbitrarily large and thus arrive at a contradiction. 

\paragraph{Case $c = 2(d+1)^2 +1$ or $c=2(d+1)^2 +2$.} In this case, for every $j \in [j_0, j_1]$, we apply (i) and only invoke the ``reduced elements'' and ``minimal lattice'' subcases.
Observe that the properties of $\Phi$ being reduced or minimal are preserved by applications of Lemma \ref{straig} and Lemma \ref{minlat}. Thus, the interval $[j_0, j_1]$ contains at most 2 elements in this case and we conclude that $h(j_0)\le 2$. 
\paragraph{Case $c = 2(d+1)^2$.} In this case, (i) is applied at step $j$ for every $j \in [j_0, j_1]$. Note that in Lemmas \ref{straig} and \ref{minlat} the function $K(x)$ can only increase. Thus, if $G(x) \ge \delta_j^3 K_{j-1}(x)^{-d} f(V)$ is true for all $x\in \P$ at some step $j \in [j_0, j_1]$ then it remains true for all $j' \in [j, j_1]$. Thus, the `large gap' sub-case of (i) is applied at most once on the interval $[j_0, j_1]$. On the other hand, Claim \ref{konig} guarantees that there are at least $h(j_0)$ such values, a contradiction.
\paragraph{Case $c = 2\bar l+1$.} In this case, for some $l \in [d]^2$ and every $j \in [j_0, j_1]$, only the following operations could be performed: 
\begin{itemize}
    \item all subcases of (i),
    \item case (ii) applied to an element $x$ with $l_{\Phi^{j-1}}(x) \succeq_{lex} l$; moreover, this case was applied to at least $h(j_0)$ elements $x$ with $l_{\Phi^{j-1}}(x) = l$,
    \item case (iii) applied to an element $x$ with $l_{\Phi^{j-1}}(x) \succ_{lex} l$.
\end{itemize} 

Note that none of these operations increases the number of elements $x \in \P^j$ with $l_{\Phi^j}(x) \preceq_{lex} l$. Indeed, for fixed $x$, the sequence of elements $l_{\Phi^j}(x)$ is increasing in $j$, where we identify $x \in \P^j$ with elements of the form $(x, 1) \in \P^{j+1}$, $((x, 1), 1) \in \P^{j+2}$ and so on. Moreover, by Observations \ref{level1} and \ref{level2}, the newly added elements to $\P^j$ always have level strictly larger than $l$.
Since $|\P^{j_0}| \le 2^{j_0}$, there are at most $2^{j_0}$ elements $x$ with level at most $l$ appearing in one of the posets $\P^j$, $j \in [j_0, j_1]$. In fact, all of these elements already appear in $\P^{j_0}$. By the pigeonhole principle, there exists an element $x \in \P^{j_0}$ with level $l$ such that case (ii) was applied to $x$ and to some $\varepsilon_j$-large face of $P_x^j$ for at least $M = \lfloor h(j_0) 2^{-j_0}\rfloor$ different indices $j \in [j_0, j_1]$.
Let $i_1 <  \ldots < i_M$ be the list of such indices and for $t=1, \ldots, M$ let $\Gamma_t \subset P^{i_t}_x$ be the $\varepsilon$-large face to which case (ii) was applied. In particular, $\Gamma_t$ is unrealized in $\Phi^{i_t-1}$ and it is realized in $\Phi^{i_t}$. 

Since $\Gamma_t$ is $\varepsilon$-large in $\Phi^{i_t-1}$, the weight of the function $\hat f_{x}^{i_t-1}$ on $\Gamma_t$ is at least $\varepsilon f^{\Phi^{i_t-1}}(V)$. It follows that, for all $j \ge i_t$, the restriction of  $\hat f_{x}^{j}$ to $\Gamma_t$ is non-zero since the total weight removed from $f$ at all these steps is smaller than $\varepsilon f^{\Phi^{i_t-1}}(V)$. In particular, for every $j \in [i_t, j_1]$, we have $\Gamma_t \cap P^j_x \neq \emptyset$. Thus, our clean-up and refinement lemmas imply that $\Gamma_t \cap P^j_x$ is a realized face in $\Phi^j$ for all $j \in [i_t, j_1]$.
But for every $t' > t$, $\Gamma_{t'}$ is unrealized in $\Phi^{i_{t'}-1}$, and so $\Gamma_{t'} \neq \Gamma_t \cap P_{x}^{i_{t'}-1}$. Let $\mu$ be the measure on $\Lambda_x$ corresponding to the function $\hat f^{j_1}_x$. Then, using the inequalities between $\hat f^j_x(\Lambda_x)$ for different $j\in [j_0, j_1]$, one can show that, for every $t=1, \ldots, M$, the face $\Gamma_t \subset P_x^{i_t-1}$ is $\varepsilon/2$-large with respect to $P_x^{i_t-1}$ and the measure $\mu$. Thus, Proposition \ref{large} can be applied and we get that $M \ll_{d, \varepsilon} 1$ holds. We conclude that $h(j_0) \ll_{d, \varepsilon} 2^{j_0}$, which gives a contradiction provided that $h$ grows fast enough, depending only on $\varepsilon$ and $d$. 

\paragraph{Case $c = 2 \bar l$.} In this case, for some $l \in [d]^2$ and every $j \in [j_0, j_1]$, only the following operations could be performed: 
\begin{itemize}
    \item all subcases of (i),
    \item case (ii) applied to an element $x$ with $l_{\Phi^{j-1}}(x) \succeq_{lex} l$, 
    \item case (iii) applied to an element $x$ with $l_{\Phi^{j-1}}(x) \succeq_{lex} l$; moreover, this case was applied to at least $h(j_0)$ elements $x$ with $l_{\Phi^{j-1}}(x) = l$.
\end{itemize} 

If case (iii) is applied to some element $x$ at step $j \in [j_0, j_1]$ then by Lemma \ref{oper2}, $x$ (which we identify with the element $(x, 0)$) is not a reduced element in $\Phi^j$ and so $x$ is removed from $\P^{j}$ on the next step of the procedure. Thus, case (iii) can be applied to a given element $x$ only once. Since the number of elements on level at most $l$ does not increase on the interval $[j_0, j_1]$, we conclude that $h(j_0) \le 2^{j_0}$ which is a contradiction provided that $h$ grows fast enough.


\section{Relative set expansion}\label{set_expansion}

\subsection{Additive combinatorics tools}\label{exp}

For a non-constant linear\footnote{Since we are working with affine spaces we allow $\xi$ to have a constant term.} function $\xi: \FF_p^d \rightarrow \FF_p$ and a positive integer $K > 0$ we define a $K$-slab $H(\xi, K)$ to be the set $\{v \in \FF_p^d: ~ \xi( v ) \in [-K, K] \}$. 

\begin{defi}\label{defthick}
Let  $K \ge 1$ be an integer and $\varepsilon > 0$. We say that a multiset $X \subset \FF_p^d$ is {\em $(K, \varepsilon)$-thick along $\xi$} if we have $|X \cap H(\xi, K)| \le (1 - \varepsilon)|X|$.
We say that $X$ is {\em $(K, \varepsilon)$-thick} if $X$ is $(K, \varepsilon)$-thick along $\xi$ for every linear function $\xi$.  Otherwise we say that $X$ is {\em $(K, \varepsilon)$-thin along $\xi$} or {\em $(K, \varepsilon)$-thin}, respectively. 

A $K$-slab $H=H(\xi, K)$ is called {\em centrally symmetric} if the linear function $\xi$ has no constant term.
\end{defi}

The next two lemmas are similar to the main tools Alon and Dubiner \cite[Propositions 2.4 and 2.1, respectively]{AD} used in their proof of the bound (\ref{adin}). The statements that we need are slightly different from their analogues in \cite{AD}, so for the reader's convenience we include full proofs.

\begin{lemma}\label{ad}
Fix $d\ge 1, K \ge 100$ and $\varepsilon > 0$, and let $p$ be a prime such that $p > 100 K$. Let $A$ be a sequence of elements of $\FF_p^d$ and suppose that every centrally symmetric $K$-slab contains at most $(1 - \varepsilon)|A|$ members of $A$. Then, for every subset $Y \subset \FF_p^d$ of at most $p^d/2$ elements
there is an element $a \in A$ such that $|(Y+a) \cup Y| \ge (1+\frac{K\varepsilon}{c_0 p})|Y|$. Here one can take $c_0 = 200$.
\end{lemma}
\begin{proof} 
Let $m = \lceil\frac{40 p}{K}\rceil$. Let $f:\FF_p^d \rightarrow \Z$ be the characteristic function of the sequence $A$ and define a function $f_m$ as follows:
$$
f_m(x) = \sum_{t = 1}^m f(x/t).
$$
Consider a weighted graph $G$ on $\FF_p^d$, where vectors $x, y \in \FF_p^d$ are connected by an edge of weight $f_m(x-y) + f_m(y-x)$. Then $G$ is a regular graph of degree $\Delta = 2 m |A|$. Let $\lambda_2(G)$ denote the second largest eigenvalue of the graph $G$. By Alon--Milman inequality \cite{AM}, for every set $Y \subset \FF_p^d$ of size at most $p^d/2$ we have 
$$
E(Y, \overline{Y}) \ge (\Delta - \lambda_2(G)) \frac{|Y| |\overline{Y}|}{p^d} \ge (\Delta - \lambda_2(G)) |Y| / 2,
$$
where $E(Y, Z)$ denotes the weight of edges between $Y$ and $Z$. Suppose that we have shown that $\lambda_2(G) \le (1-\varepsilon/2) \Delta$, then the above inequality gives
$$
E(Y, \overline{Y}) \ge \varepsilon m |A| |Y| / 2.
$$
The graph $G$ is a union (with multiplicities) of matchings $\{(x, x+ t a), ~ x\in \FF_p^d  \}$ over all possible choices $a \in A$ and $t =1, \ldots, m$. Therefore, by the pigeonhole principle, there exists $a \in A$ and $t \in [1, m]$ such that 
$$
|(Y + t a) \setminus Y| \ge \varepsilon |Y|/4.
$$
On the other hand, we have a simple inequality 
$$
|(Y + t a) \setminus Y| \le t |(Y + a) \setminus Y| \le m |(Y + a) \setminus Y|,
$$
and so together these bounds imply that $|Y \cup (Y+a)| \ge (1+ \frac{\varepsilon K}{200 p})|Y|$, as desired.

To finish the proof it remains to check that we indeed have the upper bound $\lambda_2(G) \le (1-\varepsilon/2)\Delta$ on the second eigenvalue of $G$. Recall that the eigenvalues of the weighted Cayley graph $G$ are given by the (normalized) Fourier coefficients of the function $f_m(x) + f_m(-x)$.

For an arbitrary function $h: \FF_p^d \rightarrow \C$ and a linear function $\xi: \FF_p^d \rightarrow \FF_p$ define the Fourier coefficient 
\begin{equation}\label{fourier}
    \hat h(\xi) = \sum_{x \in \FF_p^d} h(x) e(\xi(x)),    
\end{equation}
where we write $e(y) = e^{\frac{2\pi i y}{p}}$. The normalization in (\ref{fourier}) is not quite standard but it is more convenient for our purposes: the second largest eigenvalue of $G$ is now precisely 
$$
\lambda_2(G) = \max_{\xi \neq 0}\left( \hat f_m(\xi) + \overline{\hat f_m(\xi)}\right) \le 2 \max_{\xi \neq 0} |\hat f_m(\xi)|.
$$
Applying (\ref{fourier}) to $f_m$ we get
$$
\hat f_m(\xi) =  \sum_{x} f_m(x) e(\xi(x)) =\sum_{x} \sum_{t=1}^m f(x/t)  e(\xi(x)) = \sum_{x} \left( f(x) \cdot \sum_{t=1}^m e(t \xi(x))\right).
$$
Summing over the geometric progression, for $\xi(x) \neq 0$, we get
$$
\hat f_m(\xi) = \sum_{x} f(x) e(\xi(x)) \frac{1 -e(m \xi(x))}{1-e(\xi(x))}.
$$
For $x \in \FF_p^d$ such that $\xi(x) \in [-K, K]$ we bound the term on the right hand side by $m f(x)$. For $x \in \FF_p^d$ such that $\xi(x) \not \in [-K, K]$ we have
$$
\left |e(\xi(x)) \frac{1 -e(m \xi(x))}{1-e(\xi(x))}\right| \le \frac{2}{|1- e^{\frac{2\pi i K}{p}}|} \le \frac{20 p}{K},
$$
and so by the triangle inequality we can estimate $\hat f_m(\xi)$ for $\xi \neq 0$:
$$
|\hat f_m(\xi)| \le |A \cap H(\xi, K)| m + |A \setminus H(\xi, K)| \frac{20p}{K},
$$
where the first term accounts for the contribution of vectors $x$ such that $\xi(x) \in [-K, K]$ and the second term accounts for $x$ such that $\xi(x) \not \in [-K, K]$. Recall that we picked $m = \lceil\frac{40 p}{K}\rceil$ and by the assumption on the set $A$ we have $|A \setminus H(\xi, K)| \ge \varepsilon |A|$. So we have
$$
|\hat f_m(\xi)| \le |A \cap H(\xi, K)| m + |A \setminus H(\xi, K)| m/2 \le |A| m (1- \varepsilon/2).
$$
This gives us the desired bound on the second eigenvalue of $G$ and completes the proof.
\end{proof}

\begin{lemma}\label{ad2}
Let $A \subset \FF_p^d$ be a non-empty subset such that $|A| = x^d \le (p/2)^d$. Let $E$ be a basis of $\FF_p^d$. Then there is an element $v \in E$ such that $|A \cup (A+v)| \ge (x + \frac{1}{3d})^d$.
\end{lemma}
\begin{proof}
The proof is based on a discrete version of the Loomis--Whitney inequality \cite{LW}:
\begin{prop}\label{LW}
Let $A \subset \R^d$ be a finite set. Let $A_i$ be the projection of $A$ on the $i$-th coordinate hyperplane $\{(x_1, \ldots, x_d)~|~x_i=0\}$. Then one has an inequality $|A|^{d-1} \le \prod_{i = 1}^d |A_i|$.
\end{prop}

Let $A \subset \FF_p^d$ and $|A| = x^d \le (p/2)^d$. We may assume that $E$ is the standard basis of $\FF_p^d$. 
Now consider the standard embedding of $\FF_p^d$ in $\Z^d$ and apply Proposition \ref{LW} to the image of $A$. It follows that there is $i \in \{1, \ldots, d\}$ such that $|A_i| \ge x^{d-1}$. This means that at least $x^{d-1}$ lines of the form $l_v = \{v + te_i\} \subset \FF_p^d$ intersect $A$. For every line $l_v$ intersecting $A$, we have either $|(A \cup (A+e_i)) \cap l_v| > |A \cap l_v|$ or $l_v \subset A$. But the number of lines $l_v$ contained in $A$ is at most $|A|/p$ 
so there are at least $x^{d-1} - x^d/p$ lines $l_v$ which intersect $A$ and are not contained in it. Thus,  
$$
|(A+e_i) \setminus A| \ge x^{d-1} - x^d/p \ge x^{d-1}/2.
$$
Finally, it is easy to verify that for every $x, d \ge 1$ the following inequality holds: $x^d + x^{d-1}/2 \ge (x + \frac{1}{3d})^d$. 
\end{proof}

\subsection{Set Expansion argument}\label{aldu}

This section is devoted to a proof of the following technical result.

\begin{theorem}\label{thexp}
    Fix integers $d, t \ge 0$, an integer $K \ge 1$, and a real number $\delta > 0$. There exist integers $T_0, p_0 \ge 0$ depending on all these parameters such that the following holds for every integer $T \ge T_0$ and every prime $p \ge p_0$. 

    Let $S \subset [-K, K]^d$ be a set of integer vectors. Suppose that for every $q \in S$ there exists a multiset $X_q \subset \FF_p^t$ and an integer $\alpha_q$ such that the following conditions hold:
    \begin{itemize}
        \item We have
    \[
    \sum_{q \in S} \alpha_q q = 0 \text{  and  } \sum_{q \in S} \alpha_q = p.
    \] 
    \item For every $q \in S$, we have $\delta p \le \alpha_q \le |X_q|-\delta p$.
    \item Denote $X = \bigcup_{q \in S} \{q\} \times X_q \subset \FF_p^{d+t}$. Then $X$ is $(T, \delta)$-thick along every linear function $\xi: \FF_p^{d+t} \rightarrow \FF_p$ which is not constant on $\{0\}\times \FF_p^t$.
    \end{itemize}

    Then $X$ contains $p$ distinct elements with sum zero modulo $p$.
\end{theorem}

Before we give a proof, let us discuss some special cases of Theorem \ref{thexp}. If we take $t=0$, then the statement is trivial: the set $X$ is obtained from $S$ by taking each point $q$ with multiplicity $|X_q|$. Since $\alpha_q \le |X_q|$, the coefficients $\alpha_q$ provide the desired $p$ elements with zero sum.
If $d = 0$ and the set $S$ consists of a single point $q$ in a 0-dimensional space, then $X = X_q$ is assumed to be $(T, \delta)$-thick along every non-zero linear function $\xi$ for some $T \ge T_0(\delta)$. This is more or less the setting of the argument of Alon--Dubiner \cite{AD}, in which they prove a linear upper bound on the Erd{\H o}s--Ginzburg--Ziv function using the additive combinatorics tools from Section \ref{exp}. Our proof generalizes their argument to arbitrary values of the parameters $d$ and $t$.

Finally, let us remark that our proof gives a slightly more general version of Theorem \ref{thexp}: if we replace the condition $\sum_{q\in S}\alpha_q = p$ with $\sum_{q\in S}\alpha_q = m$ for an arbitrary $m$ then, under the same conditions, $X$ contains $m$ distinct elements with sum zero modulo $p$.

\begin{proof}[Proof of Theorem \ref{thexp}]
Fix the data as stated in the theorem.
Let $\Lambda \subset \Z^S$ be the dependence lattice of the set of points $S \subset \Z^d$, namely,
\begin{equation}\label{defl}
\Lambda = \left \{ (\beta_q)_{q \in S}~|~ \sum \beta_q q = 0,~\sum \beta_q = 0,~ \beta_q \in \Z \right\}.
\end{equation}

Note that $\Lambda$ is defined by a system of equations with coefficients bounded by $K$. Basic facts from linear algebra imply that we can choose a basis $e_1, \ldots, e_k$ of the lattice $\Lambda$ such that $\|e_i\|_\infty \le K_1$ for all $i = 1, \ldots, k$ and some $K_1 \ll_{K, d} 1$. Here and in what follows we consider the $\ell_s$-norms $\|\cdot \|_s$ on spaces $\R^d$ and $\R^{S}$ taken with respect to the natural bases. 

Let $K_2 \ggg K_1$ be a sufficiently large function of $K, d$ and consider the set 
$$
\Phi = \{ (\lambda, \lambda') \in \N^S \times \N^S ~|~ \lambda - \lambda' \in \Lambda,~~ \| \lambda \|_\infty, \| \lambda' \|_\infty \le  K_2 \}.
$$
For $(\lambda, \lambda') \in \Phi$ we define $\J^{\lambda, \lambda'}$ to be the family of all pairs $(J, J')$ where $J, J' \subset X$, $J \cap J' = \emptyset$ and for every $q \in S$ we have
\begin{align}\label{defj}
    |J \cap (\{q\} \times X_q) | = \lambda_q, ~~~
    |J' \cap (\{q\} \times X_q) | = \lambda_q'.
\end{align}
For $(J, J') \in \J^{\lambda, \lambda'}$ we denote 
\begin{equation}\label{sigma}
\sigma(J, J') = \sum_{x \in J} x - \sum_{x' \in J'} x' \in \FF_p^{d+t}    
\end{equation}
Since $\lambda-\lambda' \in \Lambda$, for every $(J, J') \in \J^{\lambda, \lambda'}$ we have $|J| = |J'|$. Also, by (\ref{defl}) and (\ref{defj}), the projection of $\sigma(J, J')$ on the first $d$ coordinates is 0. Thus, $\sigma(J,J') \in \{0\} \times \FF_p^t$. For $(\lambda, \lambda') \in \Phi$, define a function $\nu_{\lambda, \lambda'}: \FF_p^t \rightarrow \R_{\ge 0}$ as follows:
\begin{equation}\label{nu}
    \nu_{\lambda, \lambda'}(v) :=   \frac{| \{(J, J') \in \J^{\lambda, \lambda'}:~\sigma(J, J') = (0,v)\}|}{|\J^{\lambda, \lambda'}|}, 
\end{equation}
so, in particular, we have $\nu_{\lambda, \lambda'}(\FF_p^t) = 1$. Put $\nu = \sum_{(\lambda, \lambda') \in \Phi} \nu_{\lambda, \lambda'}$.

We use the terminology of Definition \ref{tt} for functions on $\FF_p^t$.
\begin{lemma}\label{nuu}
The function $\nu: \FF_p^t \rightarrow \R_{\ge 0}$ is $(T/K_3, \delta/K_3)$-thick along every centrally symmetric non-zero linear function $\xi$ on $\FF_p^t$. Here the constant $K_3$ is a bounded function of $d, K, K_1, K_2, \delta$.
\end{lemma}

\begin{proof}
Suppose that there is a linear function $\xi$ such that $\nu$ is $(T', \delta')$-thin along $\xi$ and $\xi(0)=0$ for some $T' > 1$ and $\delta' > 0$. Denote $H = H(\xi, T') \subset \FF_p^t$. Our strategy is to deduce that $X$ is $(T'', \delta'')$-thin along some linear function $\eta$ on $\FF_p^{d+t}$ which coincides with $\xi$ on $\{0\}\times \FF_p^t$. With the right choice of parameters, this will contradict our initial assumption on $X$ and thus prove the lemma.

Let $\Phi' \subset \Phi$ be the set of pairs $(\lambda, \lambda') \in \Phi$ such that $\nu_{\lambda, \lambda'}$ is $(T', 2\delta')$-thin along $\xi$. It follows that
$$
\nu(\FF_p^t) \delta'  \ge \nu(\FF_p^t \setminus H) = \sum_{(\lambda, \lambda') \in \Phi} \nu_{\lambda, \lambda'}(\FF_p^t \setminus H) \ge \sum_{(\lambda, \lambda') \in \Phi \setminus \Phi'}  \nu_{\lambda, \lambda'}(\FF_p^t) 2\delta' ,
$$
so, since $\nu_{\lambda, \lambda'}(\FF_p^t) = 1$ we get 
\begin{equation}\label{lamin}
|\Phi'| \ge \frac12 |\Phi|.
\end{equation}

As a first step, we show that the values of $\xi$ on sets $X_q \subset \FF_p^t$ should be concentrated on short intervals. 

\begin{prop}\label{jb}
    For every $q \in S$ there exists a number $r_q \in \FF_p$ such that $\xi(x) - r_q \in [-2T', 2T']$ for all but at most $6\delta' |X_q|$ elements $x \in X_q$.
\end{prop}

\begin{proof}
We claim that there exists $(\lambda, \lambda')\in \Phi'$ such that $(\lambda_q, \lambda'_q) \neq (0, 0)$. Indeed, the set of $(\lambda, \lambda') \in \Phi$ such that $(\lambda_q, \lambda'_q) = (0, 0)$ is contained in $\Phi \cap V$ for some hyperplane $V \subset \R^{S} \times \R^{S}$. Since $(\lambda, \lambda) \in \Phi$ for every $\|\lambda\|_\infty \le K_3$, the set $\Phi$ is not contained in $V$. It follows that if we choose $K_2$ sufficiently large compared to $d, K, K_1$ then $|\Phi \cap V| \le 0.1 |\Phi|$. By (\ref{lamin}), we conclude that $\Phi' \not \subset V$ and so there exists $(\lambda, \lambda')\in \Phi'$ such that $(\lambda_q, \lambda'_q) \neq (0, 0)$. Without loss of generality let us assume that $\lambda_q \neq 0$.
    
Let $G$ be a graph on the vertex set $X_q$ where elements $x, x' \in X_q$ are connected by an edge if $\xi(x) - \xi(x') \not \in [-2T', 2T']$. Note that if the independence number of $G$ is at least $(1-6\delta') |X_q|$, then the statement of the proposition follows (take $r_q = \xi(x)$ for every member $x$ of the independent set). Thus, we may assume that $G$ has no independent set of size $(1-6\delta') |X_q|$. Hence, we can find $\ell = \lceil 3\delta'|X_q| / \lambda_q \rceil$ pairwise disjoint edges $(x_1, y_1), \ldots, (x_\ell, y_\ell)$ in $G$.

For $x \in X_q$, denote by $\J_x \subset \J^{\lambda, \lambda'}$ the set of pairs $(J, J')$ such that $(q,x) \in J$. By symmetry, we have $|\J_x| = \frac{\lambda_q}{|X_q|} |\J^{\lambda, \lambda'}|$ and for every distinct $x, y \in X_q$ we have $|\J_x \cap \J_y| \le  \left( \frac{\lambda_q}{|X_q|} \right)^2|\J^{\lambda, \lambda'}|$. 

For some $x, y$, let $(J, J') \in \J_x \setminus \J_y$. If we let $J'' = (J \setminus \{(q,x)\}) \cup \{(q,y)\}$ then $(J'', J') \in \J_y \setminus \J_x$ and by (\ref{sigma})
$$
\xi(\sigma(J, J')) - \xi(\sigma(J'', J')) = \xi(x) - \xi(y).
$$
Thus, if $(x, y)$ is an edge in $G$ then one of the sums $\xi(\sigma(J, J'))$ or $\xi(\sigma(J'', J'))$ does not belong to $[-T', T']$. Let $\mathcal I$ be the set of pairs $(J, J') \in \J^{\lambda, \lambda'}$ with $\xi(\sigma(J, J')) \not \in [-T', T']$. 
Thus, for every edge $(x, y) \in E(G)$, we get at least $|\J_x \setminus \J_y|$ such pairs $(J, J') \in \J_x \cup \J_y$. Using the Bonferroni inequality, we get
\begin{align*}
|\mathcal I| \ge \sum_{i=1}^\ell |\J_{x_i} \setminus \J_{y_i}| - \sum_{i < j} |(\J_{x_i} \cup \J_{y_i}) \cap (\J_{x_j} \cup \J_{y_j})| \ge\\
\ge \left(\ell \frac{\lambda_q}{|X_q|} - 2 \ell^2 \left (\frac{\lambda_q}{|X_q|}\right)^2\right) |\J^{\lambda, \lambda'}|.
\end{align*}
Since $\ell = \lceil 3\delta' |X_q|/\lambda_q \rceil \gg 1$, it follows that the right hand side is at least $2\delta' |\J^{\lambda, \lambda'}|$ (provided that $\delta' < 0.01$). This contradicts the assumption that $(\lambda, \lambda') \in \Phi'$, i.e. that $\nu_{\lambda, \lambda'}$ is $(T', 2\delta')$-thin along $\xi$. This concludes the proof of the proposition.
\end{proof}

For $q \in S$ denote $Z_q \subset X_q$ the subset of elements $x$ such that $\xi(x)-r_q \in [-2T', 2T']$. By Proposition \ref{jb} we have $|Z_q| \ge (1-6\delta')|X_q|$. 
Let $Z= \bigcup_{q\in S} \{q\} \times Z_q$ and
define a family $\tilde \J^{\lambda, \lambda'}$ consisting of all pairs $(J, J') \in \J^{\lambda, \lambda'}$ such that $J, J' \subset Z$. Then one can easily check
$$
|\tilde \J^{\lambda, \lambda'}| / |\J^{\lambda, \lambda'}| \ge 1 - 20 K_2 |S| \delta' \ge 0.5,
$$
where the last inequality holds provided that $\delta' |S| K_2 \le 0.01$. Thus, if $(\lambda, \lambda') \in \Phi'$ then there exists a pair $(J, J') \in \tilde \J^{\lambda, \lambda'}$ such that $\xi(\sigma(J, J')) \in [-T', T']$. Expanding the definition of $\sigma$ and using $\xi(x) - r_q \in [-2T', 2T']$ for $x \in Z_q$ gives 
\begin{equation*}
\sum_{q \in S} (\lambda_q - \lambda'_q) r_q \in [-K' T', K' T'] \pmod p
\end{equation*}
where $K'=10 |S| K_2$, which holds for every $(\lambda, \lambda') \in \Phi'$. By (\ref{lamin}), there exist $I, M \ll_{K_2, S} 1$ such that every vector in $\Phi$ can be expressed as a sum of at most $M$ vectors in $\Phi'$ divided by $I$ (pick a maximal linearly independent collection in $\Phi'$ and use triangle inequality). Thus, we get that for every $(\lambda, \lambda') \in \Phi$ we have
\begin{equation}\label{kp}
I \sum_{q \in S} (\lambda_q - \lambda'_q) r_q \in [-K'' T', K'' T'] \pmod p,
\end{equation}
where $K'' = M K'$. Let $S' \subset S$ be a minimal subset whose affine hull coincides with the affine hull of $S$. Choose $a \in \Z^d, b \in \Z$ such that $r_q = \langle a, q \rangle + b \pmod p$ for all $q \in S'$. Since the set $S'$ is affinely independent and $p$ is large enough compared to $K$, we can always solve this system modulo $p$. 

For every $q \in S \setminus S'$ there exists a unique (up to a constant) vector $u(q) \in \Lambda$ with support in $S' \cup \{q\}$ and $u(q)_q \neq 0$. Since $S \subset [-K, K]^d$, we can choose this vector in such a way that $\|u(q)\|_\infty \ll_{K, d} 1$. Since $K_2$ is taken sufficiently large, we have $\|u(q)\|_\infty \le K_2$. Thus, we can apply (\ref{kp}) to the pair $(u(q), 0) \in \Phi$ and obtain
$$
I u(q)_q r_q  + I \sum_{q' \in S'} u(q)_{q'} r_{q'} \in [-K'' T', K'' T'] \pmod p.
$$
On the other hand, since $u(q) \in \Lambda$, we have 
$$
0 = u(q)_q (\langle a, q\rangle +b) + \sum_{q' \in S'} u(q)_{q'} (\langle a, q'\rangle +b) = u(q)_q (\langle a, q\rangle +b) + \sum_{q' \in S'} u(q)_{q'} r_{q'},
$$
where in the last transition we used the definition of $a$ and $b$.

Let $\tilde I = I K_2!$. By subtracting the two equations above with an appropriate coefficient, we conclude that for every $q\in S$ we have, for some $\tilde K \ll_{K,d,\delta} 1$, 
$$
\tilde I (r_q - \langle a, q \rangle - b) \in [-\tilde K T', \tilde K T'] \pmod p.
$$
Recall that $\xi(x) - r_q \in [-2T', 2T']$ for every $x \in Z_q$, so by the triangle inequality we get for every $x \in Z_q$
$$
\tilde I \xi(x) - \tilde I (\langle a, q\rangle + b) \in [-\tilde K' T', \tilde K' T'] \pmod p,
$$
where $\tilde K' = \tilde K + 2\tilde I$.
Let $\eta: \FF_p^d \times \FF_p^t \rightarrow \FF_p$ be a function defined as
$$
\eta(y, x) = \tilde I (\xi(x) - (\langle a, y\rangle + b)),
$$
where $y \in \FF_p^d$, $x \in \FF_p^t$. Since $\xi \neq 0$ and $p > \tilde I$, the restriction of $\eta$ on $\{0\}\times \FF_p^t$ is non-constant. Thus, by our assumption, the set $X$ is $(T, \delta)$-thick along $\eta$. On the other hand, we showed that for every $v = (q, x) \in Z_q \subset Z$ we have $\eta(v) \in [-\tilde K' T', \tilde K' T']$. Recall that by Proposition \ref{jb} we have $|Z| \ge (1-6\delta') |X|$. If we choose $T' = T/\tilde K'$ and $\delta' = \min \{\delta/6, 0.01 |S|^{-1} K_2^{-1}\}$ then these conditions contradict each other (and the second term in the minimum makes the argument go through). This completes the proof of the lemma.
\end{proof}

Denote $T' = T/K_3, \delta' = \delta/K_3$ so that $\nu$ is $(T', \delta')$-thick as in Lemma \ref{nuu}. 
Now we can apply Lemmas \ref{ad} and \ref{ad2} to the function $\nu$. Let $\J = \bigcup_{(\lambda, \lambda') \in \Phi} \J^{\lambda, \lambda'}$. For a set $J \subset \FF_p^{d+t}$ we denote $\sigma(J) = \sum_{x \in J} x$.
\begin{prop}\label{adp1}
For every sufficiently small $c>0$, depending on $K,d,\delta$, there is a sequence of pairs $(J_i, J'_i) \in \J$ for $i = 1, \ldots, cp$ such that:
\begin{enumerate}
    \item For every $i \neq j$, the sets $J_i \cup J'_i$ and $J_j \cup J'_j$ are disjoint.
    \item The sum of cardinalities of all these sets is at most $2K_2 |S| cp$.
    \item Let $M_i = \{ \sigma(J_i), \sigma(J'_i) \} \subset \FF_p^{d+t}$. Then we have 
\begin{equation}\label{eq3}
    |M_1 + \ldots + M_{cp}| \ge \left(\frac{cp}{3t} \right)^{t}.
\end{equation}
\end{enumerate}
\end{prop}
\begin{proof}
First, the second conclusion is trivial: since $|J|+|J'| \le 2 K_2 |S|$ for every $(J, J') \in \J$ the sum of cardinalities of sets $J_i, J'_i$ is at most $2K_2 |S| cp$. 

Using the thickness of $\nu$ and simple union bounds one can find at least $j \gg_{K, d, \delta} p$ linear bases $B_1, \ldots, B_j \subset \FF_p^t$ with the property that the $i$-th basis $B_i$ has the form 
$$
\{ \sigma(J_{i, k}, J'_{i, k}) \}_{k = 1}^t,
$$ 
where $\{(J_{i, k}, J'_{i, k})\}_{i, k = 1, 1}^{j, t}$ is a collection of pairs from $\J$ such that all these pairs are pairwise disjoint. 
By iterative application of Lemma \ref{ad2} we can choose some pairs $(J_{i, k_i}, J'_{i, k_i})$ for $i = 1, \ldots, j$ which satisfy 
\begin{equation}\label{mink}
| \{0, \sigma(J_{1, k_1}, J'_{1, k_1})\} + \ldots + \{0, \sigma(J_{j, k_j}, J'_{j, k_j})\} | \ge \left (  \frac{j}{3t} \right)^t.     
\end{equation}
The Minkowski sum in the statement of the proposition is a shift of (\ref{mink}) and so the proposition holds for every $c \le j/p$.
\end{proof}

In the next proposition we continue the process of adding new pairs to the sequence $(J_i, J'_i)$ but we will invoke Lemma \ref{ad} instead of Lemma \ref{ad2}. Let $Y = M_1 + \ldots + M_{cp} \subset \FF_p^{d+t}$.

\begin{prop}\label{adp2}
For some sufficiently small $c>0$, depending only on $K,d,\delta$, there is a sequence of pairs $(J_i, J'_i) \in \J$ for $i = cp+1, \ldots, cp+l$ for some $l \le cp$ such that:
\begin{enumerate}
    \item For every $1 \le i \neq j \le cp+l$, the sets $J_i \cup J'_i$ and $J_j \cup J'_j$ are disjoint.
    \item The sum of cardinalities of all these sets is at most $0.1 \delta p$.
    \item For $i = cp+1, \ldots, cp+l$ let $M_i = \{ \sigma(J_i), \sigma(J'_i) \}$. Then we have 
\begin{equation}\label{eq4}
    |Y + M_{cp+1} + \ldots + M_{cp+l}| \ge p^t / 2.
\end{equation}
\end{enumerate}
\end{prop}

\begin{proof}
As in the previous proposition, the bound on the sum of cardinalities follows if we take $c < 0.01 \delta /(K_2 |S|)$.

We construct pairs $(J_i, J'_i)$ one by one. At step $1\le j \le cp$ we consider the set 
$$
Y_j = Y + M_{cp+1} + \ldots + M_{cp+j-1}
$$
and consider a function $\nu_j$ defined analogously to $\nu$ but with elements already appearing in previously chosen sets removed. Since the union of all these sets has size at most $ 2K_2 |S| c p \le 2K_2 |S| c |X|$, the function $\nu_j$ is $(T', \delta'/2)$-thick provided that $c$ is small enough in terms of $\delta'$, $K_2$ and $|S|$. 

If $|Y_j| \le p^t/2$, we can apply Lemma \ref{ad} to the set $Y_j$ and the function $\nu_j$ and obtain a pair $(J_{cp+j}, J'_{cp+j}) \in \J_j$ (where $\J_j$ is defined analogously to $\J$) such that 
\begin{equation}\label{Yj}
|Y_j \cup (Y_j + \sigma(J_{cp+j}, J'_{cp+j}))| \ge \left(1 + \frac{T' \delta'}{C p}\right) |Y_j|,    
\end{equation}
for some absolute constant $C$. We can then repeat this argument with $j+1$ instead of $j$.

The procedure above can stop only in two cases: if for some $j \le cp$ we get $|Y_j| \ge p^t/2$ which completes the proof, or if we reach $j = cp$. In the latter case we get by (\ref{Yj}):
$$
|Y_{cp+1}| \ge \left(1 + \frac{T' \delta'}{C p}\right)^{cp} |Y_1|  \gg e^{\frac{c T' \delta'}{C}} \left(\frac{c}{3t}\right)^t p^t.
$$
However, we have $T'\delta' \ge T \delta /K_3^2$ and if we take $T$ large enough compared to $c, \delta, K, K_3, d, t$ then the right hand side exceeds $p^t$, which is absurd. This completes the proof.
\end{proof}

By removing all constructed pairs from $X$ and applying the propositions above once again, we can construct another sequence of at most $\tilde j \le 2cp$ pairs $(\tilde J_{i}, \tilde J'_{i})$ which are pairwise disjoint and disjoint from the previously constructed sets, have the sum of sizes at most $0.1\delta p$ and satisfy $|\tilde M_1 + \ldots + \tilde M_{\tilde j}| \ge p^t /2$ (where $\tilde M_i = \{\sigma(\tilde J_i), \sigma(\tilde J'_i)\}$). 
Taking the union of these two sequences, applying the easy part of the Cauchy--Davenport theorem, and relabeling indices, we arrive at
\begin{cor}\label{fin}
There is a set of $j \le 4cp$ pairs $(J_i, J'_i) \in \J$, $i = 1, \ldots, j$, such that:
\begin{enumerate}
    \item The sets $J_i \cup J'_i$, $i=1, \ldots, j$ are pairwise disjoint.
    \item The sum of cardinalities of all these sets is at most $0.2 \delta p$.
    \item For $i = 1, \ldots, j$ let $M_i = \{ \sigma(J_i), \sigma(J'_i) \}$, then, for some $u_0 \in \FF_p^d$, we have
    \begin{equation}\label{M}
    M_1 + \ldots + M_j = \{u_0\} \times \FF_p^t.    
    \end{equation}
\end{enumerate}
\end{cor}

Let $(\lambda_i, \lambda'_i) \in \Phi$ be the pair corresponding to the sets $(J_i, J'_i)$. Note that by (\ref{M}) we then have
\begin{equation}\label{M2}
    u_0 = \sum_{q \in S} \sum_{i=1}^j \lambda_{i, q} q
\end{equation}
since exactly $\lambda_{i, q}$ elements in $J_i$ have first $d$ coordinates equal to $q$.
Let $B = \bigcup_{i=1}^j J_i$ and $B' = \bigcup_{i=1}^j J'_i$. Note that for every $q\in S$ we have
\begin{equation}\label{M3}
    |B \cap (\{q\} \times X_q)| = \sum_{i=1}^j \lambda_{i,q}.
\end{equation}
Part 2 of Corollary \ref{fin} implies $|B \cap (\{q\} \times X_q)| \le 0.2 \delta p$. 
Recall that $\delta p \le \alpha_q \le |X_q| - \delta p$ so for every $q \in S$ we have 
$$
0 \le \alpha_q - |B \cap (\{q\}\times X_q)| \le |X_q| - \delta p,
$$
and there exists a subset $D_q \subset X_q$ of size exactly $\alpha_q - |B \cap (\{q\} \times X_q)|$ such that $\{q\} \times D_q$ is disjoint from $B \cup B'$. Let $v_1 = \sum_{q\in S} \sum_{x\in D_q} x \in \FF_p^t$. By Corollary \ref{fin}, we can choose subsets $I_i \in \{J_i, J'_i\}$, $i=1, \ldots, j$, such that
\begin{equation}\label{latter}
\sum_{i=1}^j \sigma(I_i) = (u_0, -v_1).    
\end{equation}
We claim that the disjoint union 
$$
Y = I_1 \cup \ldots \cup I_j \cup \bigcup_{q \in S} (\{q\}\times D_q) \subset X
$$
consists of $p$ elements whose sum is zero. Indeed, we have $|I_i| = |J'_i| = |J_i|$ for every $i$ and so
\begin{align*}
|Y| = |I_1| + \ldots + |I_j| + \sum_{q\in S} |D_q| = |B| + \sum_{q \in S} \left(\alpha_q - |B \cap (\{q\} \times X_q)|\right) = \sum_{q\in S} \alpha_q = p.
\end{align*}
Let $(w_0, w_1) = \sum_{y \in Y} y$. We need to show that $w_0 = 0$ and $w_1 = 0$. By (\ref{M}), (\ref{M2}) and (\ref{latter}), we have
\begin{align*}
w_0 &= u_0 + \sum_{q \in S} |D_q| q \\
&= u_0 + \sum_{q\in S} (\alpha_q - |B \cap (\{q\} \times X_q)|) q \\
&= u_0 - \sum_{q\in S} |B \cap (\{q\} \times X_q)| q = u_0 - \sum_{q\in S} \sum_{i=1}^j \lambda_{i,q} q = 0.
\end{align*}
By (\ref{latter}) and the definition of $v_1$, we have $w_1 = -v_1 + v_1 = 0$. Thus, the sum of the elements of $Y$ is zero, and the set $X$ contains $p$ distinct elements with zero sum. This proves the theorem.
\end{proof}

\section{Balanced convex combinations}\label{balcon}

In this section we give the last ingredient needed in the proof of Theorem \ref{main}.

Let $w: \R^d \rightarrow \R_{\ge 0}$ be a function with finite support. For a subset $S \subset \R^d$ we denote by $w(S)$ the sum $\sum_{s \in S} w(s)$.
We say that a point $c \in \R^d$ is {\em $\theta$-central for $w$} if for every half-space $H^+$ which contains $c$ we have $w(H^+) \ge \theta w(\R^d)$.

\begin{lemma}\label{lm2}
Let $\theta > 0$, let $w: \R^d \rightarrow \R_{\ge 0}$ be a function with finite support $S$. 
Let $\Lambda$ be the minimal lattice containing $S$ and $c \in \Lambda \cap \,{\rm int}(\conv{S})$ be a $\theta$-central point for $w$.

Then for every $\varepsilon > 0$ and all $n > n_0(\varepsilon, w, c)$ there are non-negative integer coefficients $\alpha_q$ for $q \in S$ and $\mu = \mu(\varepsilon, w, c) > 0$ such that:
\begin{align}\label{eqlm}
    \sum_{q \in S} \alpha_q = n,~~    \sum_{q \in S} \alpha_q q = n c, 
\end{align}
and for every $q\in S$ we have 
\begin{align}\label{eqlm2}
    \mu n \le \alpha_q \le  (1 + \varepsilon) \frac{n w(q)}{\theta w(S)}.
\end{align}
\end{lemma}

\begin{proof}
Without loss of generality, we may assume that $c = 0$, the set $S$ spans $\R^d$, $\Lambda = \Z^d$ and $w(S) = \sum_{q \in S} w(q) = 1$.

\begin{claim}\label{real}
There are rational coefficients $\beta_q$ such that:
\begin{align*}
    \sum_{q\in S} \beta_q q = 0,~~ \sum_{q \in S} \beta_q = 1, 
\end{align*}
and $\beta_q \in (0, \theta^{-1}w(q))$ for every $q\in S$.
\end{claim}
\begin{proof}

Note that it is enough to find {\em real} coefficients $\beta_q$ with properties described in the claim. The existence of rational coefficients would then follow automatically.

We denote by $\R^S$ the space of all functions $\xi: S \rightarrow \R$. This space is equipped with the natural scalar product $\xi \cdot \eta = \sum_{q \in S} \xi(q) \eta(q)$. In what follows we identify $\R^S$ with the dual space $(\R^S)^*$ via this scalar product.

Let $H \subset \R^S$ be the set of vectors $(c_q)_{q \in S}$ such that $\sum_{q \in S} c_q q = 0$. 
Let $\Omega \subset \R^S$ be the set of all functions $v: S \rightarrow \R$ such that 
$$
0 \le v(q) \le \theta^{-1}w(q) \sum_{q' \in S} v(q'),
$$
for every $q \in S$. Note that if the intersection $H \cap \operatorname{int}(\Omega)$ is non-empty, then we are done: take a vector $v \in H \cap \operatorname{int}(\Omega)$ and define $\beta_q = \frac{v(q)}{\sum_{q' \in S} v(q')}$.

Let us assume that $H \cap \operatorname{int}(\Omega) = \emptyset$ and arrive at a contradiction. Since $H$ is a vector subspace and $\operatorname{int}(\Omega)$ is an open convex set, there exists a function $\xi \in \R^S$ such that  
$$
\xi(H) = 0 ~~ {\text{and}} ~~\xi(\Omega) \ge 0.
$$
The first condition can be reformulated as $\xi \in H^\bot$.
Note that the space $H^\bot$ is isomorphic to $\R^d$: given a function $\zeta \in H^\bot$ we define a linear function $\tilde \zeta$ on $\R^d$ by setting $\tilde \zeta(q) = \zeta(q)$ for $q \in S$ and extending $\tilde \zeta$ by linearity. The conditions that $S$ spans $\R^d$ and the linear equations defining $H^{\bot}$ imply that this definition is correct. 
Let $\tilde \xi$ be the linear function on $\R^d$ corresponding to $\xi$.

Let $\varepsilon_q$ be the element of the standard basis of $\R^S$ corresponding to $q \in S$ and denote $\sigma = \sum_{q \in S}\varepsilon_q$. The set $\Omega$ is defined as the set of vectors $v \in \R^S$ such that
\begin{equation} \label{defw}
\varepsilon_q \cdot v \ge 0~~ \text{and} ~~(w(q)\sigma - \theta \varepsilon_q) \cdot v \ge 0,
\end{equation}

for all $q \in S$. By duality, the condition $\xi(\Omega) \ge 0$ is a non-negative linear combination of inequalities (\ref{defw}).
Indeed, if not, then $\xi$ can be separated by a hyperplane from functions (\ref{defw}) in the space of all linear functions on $\R^S$. But this hyperplane will correspond to a point in $\Omega$ on which the value of $\xi$ is negative.
Thus, there are nonnegative real coefficients $a_q, b_q \ge 0$ such that
\begin{equation}\label{eqxi1}
    \xi = \sum_{q \in S} a_q \varepsilon_q + b_q (w(q)\sigma - \theta \varepsilon_q) = \sum_{q \in S} (a_q - \theta b_q)\varepsilon_q + \left( \sum_{q \in S} b_q w(q) \right) \sigma.
\end{equation}

Let $I \subset S$ be the set of $q \in S$ such that $\xi \cdot \varepsilon_q \le 0$. Since $c= 0$ is a $\theta$-central point for $w$ and $\xi\cdot \varepsilon_q = \tilde \xi(q)$ for all $q \in S$, we have 
\begin{equation}\label{omegaq}
\sum_{q \in I} w(q) \ge \theta. 
\end{equation}
On the other hand, for every $q \in I$ by (\ref{eqxi1}) we have
\begin{equation}\label{xieq}
\xi(q) = (a_q - \theta b_q) + \left(\sum_{q' \in S} b_{q'} w(q') \right) \le 0,
\end{equation}
hence, by discarding the non-negative term $a_q$ from (\ref{xieq}) we get
$$
\theta b_q \ge \sum_{q' \in S} b_{q'} w(q').
$$
Summing this over $q \in I$ with weights $w(q)>0$ we obtain:
$$
\theta \sum_{q \in I} b_q w(q) \ge \left( \sum_{q \in I} w(q) \right) \left( \sum_{q \in S} b_{q} w(q)  \right) \stackrel{(\ref{omegaq})}{\ge} \theta \left( \sum_{q \in S} b_{q} w(q)  \right).
$$
The sum on the left hand side is a subsum of the right hand side. Since $\theta > 0$ and $w(q) > 0$ for all $q$, it follows that an equality is attained in (\ref{xieq}) for all $q \in I$. This, however, means that $\xi(q) \ge 0$ for every $q\in S$ and the point $c = 0$ lies on the boundary of $\conv(S)$ which contradicts our assumption.
We conclude that there cannot be such a function $\xi$ and hence $H \cap {\rm int}(\Omega) \neq \emptyset$, as desired.
\end{proof}

Take some rational coefficients $\beta_q$ provided by Claim \ref{real}; note that we can define them as functions of $w$ and $c$. Let $m$ be the least common multiple of denominators of $\beta_q$.
Since $c=0$ belongs to the minimal lattice of $S$ there is an integer vector $\delta \in \Z^S$ such that $\sum_{q \in S} \delta_q q = 0$ and $\sum_{q \in S} \delta_q = 1$. Let $C = \max_{q \in S} |\delta_q|$. 

Let us define the function $n_0 = n_0(\varepsilon, w, c)$ by
$$
n_0 = 2Cm^2 + \varepsilon^{-1} Cm\theta \max_{q\in S} w(q)^{-1}, 
$$
(note that $w(q) > 0$ for every $q \in S$ by assumption). Now consider an arbitrary $n > n_0$. Write $n = am+r$ where $0 \le r < m$ and define the coefficients by $\alpha_q = a m \beta_q + r \delta_q$; note that $\alpha_q$ is an integer.
Let us check that all required conditions are satisfied:
\begin{align*}
    \sum_{q \in S} \alpha_q q &= \sum_{q \in S} am\beta_q q + r \delta_q q = 0, \\
    \sum_{q \in S} \alpha_q &= am+r = n, \\
    \alpha_q = am\beta_q + r \delta_q &\le am \theta^{-1} w(q) + rC \le n \theta^{-1} w(q) (1 + mC n^{-1} \theta w(q)^{-1}) < n\theta^{-1} w(q) (1 + \varepsilon),
\end{align*}
A similar computation gives $\alpha_q > \mu n$ for some $\mu > 0$ not depending on $n$. Lemma \ref{lm2} is proved.
\end{proof}

\section{Proof of Theorem \ref{main}}\label{s5}

In this section we put everything together and prove our main result, Theorem \ref{main}. 

Since $\s(\FF_p^d) \ge \w(\FF_p^d)(p-1)+1$ for every $d$ and $p$, it is enough to prove that, for every fixed $d \ge 1$, every $\zeta \in (0,1)$ and all sufficiently large primes $p > p_0(d, \zeta)$, the inequality 
$$
\s(\FF_p^d) \le (\w(\FF_p^d)+\zeta)p
$$
holds.

Let $X \subset \FF_p^d$ be a multiset of size at least $(\w(\FF_p^d)+\zeta) p$ and let $f: \FF_p^d \rightarrow \N$ be the characteristic function of $X$. Let $g: \N \rightarrow \N$ be a growing function which grows fast enough depending on $d, \zeta$, and set $\varepsilon = 100^{-d}\zeta$. Apply Theorem \ref{fdl} to the function $f$ with parameters $g$ and $\varepsilon$. Thus, for some $\delta \gg_{d, \varepsilon} 1$, there exists a flag decomposition $\Phi$ of $f$ on a convex flag $(\P, \Lambda)$ and functions $T, K: \P\rightarrow \N$ such that:
\begin{itemize}
    \item $\Phi$ is $K$-bounded and $(T, \varepsilon, \delta)$-complete. Recall that this means that $\Phi$ is minimal and reduced, $\varepsilon$-large elements $x\in \P$ are $(T(x),\delta)$-complete, and $\varepsilon$-large faces $\Gamma \subset P_x$ are realized.
    \item For every $x \in \P$ we have $T(x) \ge g(K(x))$, $K(x) \ll_{g, d, \varepsilon} 1$.
    
    \item We have $G(x) \ge \gamma p$ for some $\gamma \gg_{\delta, K(x)} 1$, that is, for every $q \in \operatorname{spt}\hat f_x$ we have $\hat f_x(q) \ge \gamma p$.
    \item We have 
\begin{equation}\label{sharp}
f^{\Phi}(V) \ge (1-\varepsilon) f(V) = (1-\varepsilon) |X| \ge (\w(\FF_p^d) + \zeta /2) p.    
\end{equation}
\end{itemize}

\begin{prop}\label{lbound}
The Helly constant $L(\P, \Lambda, \Omega)$ of the convex flag $(\P, \Lambda)$ is at most $\w(\FF_p^d)$.
\end{prop}

\begin{proof}
Recall that the set of proper points $\Omega$ of the flag decomposition $\Phi$ is defined as $\Omega = \conv(\Omega_0)$ where $\Omega_0$ is the set of points ${\bf q}$ of $\P$ such that $\hat f({\bf q}) > 0$. 

Take proper integer points ${\bf q}_1, \ldots, {\bf q}_n$ of the convex flag $\P$ for some $n > \w(\FF_p^d)$. We need to show that there exist coefficients $\alpha_i \in [0, 1)$  summing to $1$ such that the convex combination ${\bf q} = \sum \alpha_i {\bf q}_i$ is an integer point of $(\P, \Lambda)$.

Recall that for each $x \in \P$ we are given an affine subspace $V_x \subset \FF_p^d$ and an affine surjective map $\varphi_x: V_x \rightarrow \Lambda_x / p\Lambda_x$.
For $i=1, \ldots, n$, let $x_i = \inf \D^{{\bf q}_i}$ and let $q_i = {\bf q}_{i, x_i}$ be the point on the lattice $\Lambda_{x_i}$ corresponding to ${\bf q}_i$. Let $w_i \in V_{x_i} \subset \FF_p^d$ be an arbitrary vector such that $\varphi_{x_i}(w_i)$ is congruent to $q_i$ modulo $p \Lambda_{x_i}$. Such a vector exists since the map $\varphi_{x_i}$ is surjective. 

Since $n > \w(\FF_p^d)$, by the definition of $\w$ applied to the set $\{w_1, \ldots, w_n\}$, there are non-negative integer coefficients $\alpha_1, \ldots, \alpha_n$ such that 
\begin{align}
    \sum_{i = 1}^n \alpha_i & = p, \label{comb0} \\
    \sum_{i = 1}^n \alpha_i w_i &\equiv 0 \pmod{p}, \label{comb}
\end{align}
and $\alpha_i < p$ for every $i$.

Let ${\bf q}$ be the convex combination of points ${\bf q}_1, \ldots, {\bf q}_n$ with coefficients $\alpha_i/p$, i.e. ${\bf q} = \sum_{i=1}^n \frac{\alpha_i}{p} {\bf q}_i$. By definition, ${\bf q}$ is a point of the convex flag $\P$ such that 
$$
\D^{\bf q} = \bigcap_{i:\alpha_i \neq 0} \D^{{\bf q}_i}
$$
and for every $x \in \D^{\bf q}$ we have the following identity:
\begin{equation}\label{comb2}
{\bf q}_x = \sum_{i =1}^n \frac{\alpha_i}{p} {\bf q}_{i, x}.
\end{equation}
We claim that ${\bf q}_x \in \Lambda_x$ for every $x \in \D^{\bf q}$. Indeed, for each $i$ such that $\alpha_i\neq 0$, we have ${\bf q}_{i, x} = \psi_{x, x_i}(q_i) \in \Lambda_{x}$ and $\varphi_x(w_i) = \psi_{x, x_i} \varphi_{x_i}(w_i)$. Thus, the point ${\bf q}_{i, x}$ is congruent to $\varphi_x(w_i)$ modulo $p\Lambda_{x}$. After choosing a consistent pair of origins in the spaces $\Lambda_x/p\Lambda_x$ and $V_{x}$, the fact that the convex combination $\sum_{\alpha_i \neq 0} \frac{\alpha_i}{p} {\bf q}_{i, x}$ is an integer point of the lattice $\Lambda_x$ is equivalent to saying that $s = \sum_{\alpha_i \neq 0} \alpha_i {\bf q}_{i, x}$ is zero in the vector space $\Lambda_x/p\Lambda_x$. This sum depends on the choice of an origin, but the fact that it is zero does not. Finally, using the map $\varphi_x$, we get
\begin{equation}\label{zero}
    \sum_{\alpha_i \neq 0} \alpha_i {\bf q}_{i, x} \equiv \sum_{\alpha_i \neq 0} \alpha_i \varphi_x(w_i) = \varphi_x \left( \sum_{i = 1}^n \alpha_i w_i \right) \equiv 0.
\end{equation}
 
We conclude that ${\bf q}$ is an integer point of the flag $(\P, \Lambda)$. Since all $\alpha_i$ are less than $p$ this implies that $L(\P, \Lambda, \Omega) \le \w(\FF_p^d)$.
\end{proof}

\begin{remark}
If we assume that the original multiset $X \subset \FF_p^d$ is in fact a genuine set without multiplicities then the bound in Proposition \ref{lbound} can be refined to $L(\P, \Lambda) \le \w(\FF_p^{d-1})$ by observing that all maps $\varphi_x$ have at least one dimensional kernels and so one can pick vectors $w_i$ inside a generic hyperplane and apply the definition of $\w$ inside of it. By combining this with the rest of the proof one can show that $|X| \le (1+\zeta) \w(\FF^{d-1}_p) p$ whenever $X$ is a set with no $p$ elements with zero sum.
\end{remark}
\vskip 0.5cm

For a proper integer point ${\bf q}$ we assign a weight $w_{\bf q}$ defined as follows. Let $x = \inf \D^{\bf q}$ and put $q = {\bf q}_x \in \Lambda_x$ and denote by $[q]$ the class of $q$ in $\Lambda_x/p\Lambda_x$. Then we define
$$
w_{\bf q} = f_x( \varphi_x^{-1}[q] )
$$
(note that this notion is different from $\hat f({\bf q})$). Let $\mathcal Q$ be the set of all proper integer points ${\bf q}$ such that $w_{\bf q} > 0$. 
Since every fiber of $\varphi_x^{-1}$ which intersects the support of $f_x$ corresponds to a proper integer point ${\bf q}$ with $x = \inf \D^{\bf q}$, we have
\begin{equation}\label{weightl}
    \sum_{{\bf q} \in \mathcal Q} w_{\bf q} = \sum_{x \in \P} f_x(V_x) = f^{\Phi}(\FF_p^d) \ge (\w(\FF_p^d) + \zeta/2) p.
\end{equation}

By the Centerpoint Theorem (Corollary \ref{central}) applied to the convex flag $(\P, \Lambda, \Omega)$ and the point set $\mathcal Q$ with the weight function $w$, there exists a proper integer point ${\bf q}_0$ of $\P$ such that for every linear function $\xi$ with $\D_{\xi} \cap \D^{{\bf q}_0} \neq \emptyset$ we have
\begin{equation}\label{xio}
\sum_{{\bf q}\in \mathcal Q:~\xi({\bf q}) \ge \xi({\bf q}_0)} w_{\bf q} \ge \frac{f^{\Phi}(\FF_p^d)}{L(\P, \Lambda, \Omega)}.  
\end{equation}
Let $x = \inf \D^{{\bf q}_0}$, $q_0 = {\bf q}_{0, x}$ and take the linear function $\xi$ to be an extension of an arbitrary non-constant linear function $\xi_x$ on $\A_x$ to the flag $\P$ (in particular, $\sup \D_{\xi} = x$). Then, if we group the proper points ${\bf q}\in \mathcal Q$ supported on $x$ according to the image ${\bf q}_x \in \Lambda_x$, then the left hand side of (\ref{xio}) can be rewritten as 
\begin{equation}\label{xii}
\sum_{q \in \Lambda_x:~ \xi(q) \ge \xi(q_0)} \hat f_x(q) \ge \frac{f^{\Phi}(\FF_p^d)}{L(\P, \Lambda, \Omega)}.    
\end{equation}
Let $S \subset \Lambda_x$ be the set of points such that $\hat f_x(q) > 0$. Recall that, by definition, $P_x = \conv S$. Then, by (\ref{xii}), the point $q_0$ is a $\theta$-central point for the function $\hat f_x$ where 
\begin{equation}\label{theta}
\theta = \frac{f^{\Phi}(\FF_p^d)}{L(\P, \Lambda, \Omega) \hat f_x(S)}.    
\end{equation}
Note that $\hat f_x(S) \le f^{\Phi}(\FF_p^d)$ and so $\theta \ge \frac{1}{L(\P, \Lambda, \Omega)} \ge \frac{1}{\w(\FF_p^d)} \ge 4^{-d}$ by Proposition \ref{lbound} and Theorem \ref{thw}. 

Let $K = K(x)$; recall that since $\P$ is $K$-bounded, the set $S$ is contained in a box $[-K, K]^{\dim \Lambda_x}$ in the coordinate system $E_x$.
Recall that the gap property gives $\hat f_x(q) \ge \gamma p$ for every $q \in S$.  
Let $H =\varepsilon \gamma p$ and define a function $\omega: S \rightarrow \N$ as
$$
\omega(q) = \left\lfloor \frac{\hat f_x(q)}{H}  \right\rfloor
$$
for $q \in S$. It is clear that for every set of points $S' \subseteq S$ we have 
$$
|\omega(S') - \frac{1}{H} \hat f_x(S') | \le |S'| \le \frac{\hat f_x(S')}{\gamma p} \le \varepsilon \frac{1}{H} \hat f_x(S').
$$ 
Using this, it is easy to see that the point $q_0$ is $(\theta-\varepsilon)$-central for the function $\omega$. 

Note that for every $q \in S$ we have $\omega(q) \ge \hat f_x(q)/H-1 > 0$ and so the supports of $\hat f_x$ and $\omega$ coincide. One of the conditions of a $(T, \varepsilon, \delta)$-completeness is that $\Phi$ is a minimal flag decomposition, that is, that $\Lambda_x$ is the minimal lattice containing the support $S$ of $\hat f_x$. We conclude that $q_0$ belongs to the minimal lattice containing the support of $\omega$. 

Finally, we claim that $q_0$ belongs to the interior of $P_x$. For the sake of contradiction, suppose that $q_0 \in {\rm relint}\,\Gamma$ for some proper face $\Gamma \subset P_x$. Taking $\xi$ to be a linear function vanishing on $\Gamma$ and negative on $P_x \setminus \Gamma$, (\ref{xii}) gives $\hat f_x(\Gamma) \ge \frac{f^\Phi(\FF_p^d)}{L(\P, \Lambda, \Omega)} \ge 4^{-d} f^\Phi(\FF_p^d)$. Similarly, for every proper face $\Gamma' \subset \Gamma$, we can find a linear function $\xi$ vanishing on $q_0$ and negative on $\Gamma'$ and on $S \setminus \Gamma$. This gives $\hat f_x(\Gamma \setminus \Gamma') \ge 4^{-d} f^{\Phi}(\FF_p^d)$.
These observations imply that $\Gamma$ is a $4^{-d}$-large face of $P_x$. Since $\varepsilon < 4^{-d}$, one of the conclusions of Theorem \ref{fdl} tells us that $\Gamma \subset P_x$ is a realized face in the flag decomposition $\Phi$. By definition, this means that $\psi_{x, x_\Gamma}(P_{x_\Gamma}) \subset \Gamma$. In particular, $x_\Gamma \prec x$.
Since ${\bf q}_0$ is a proper point, we have $x_\Gamma \in \D^{{\bf q}_0}$. But we defined $x$ as the minimum element of $\D^{{\bf q}_0}$, a contradiction. We conclude that $q_0$ lies in the interior of $P_x$.

We are in a position to apply Lemma \ref{lm2}. Indeed, the function $\omega$ has a $(\theta-\varepsilon)$-central point $q=q_0$ which belongs to both the minimal lattice spanned by the support of $\omega$ and the interior of the convex hull of the support of $\omega$. Thus, Lemma \ref{lm2} may be applied to $\omega, q, \theta-\varepsilon$ and every integer $n > n_0(\varepsilon, \omega, q)$.
We may take $p$ large enough to ensure that $p > n_0(\varepsilon, \omega, q)$ holds for all possible choices of $\omega$ and $q$. Indeed, the set $S$ is contained in a box $[-K, K]^{\dim \Lambda_x}$ where $K \ll_{g, d, \varepsilon} 1$ and $\dim \Lambda_x \le d$ and the function $\omega$ takes values in the set of integers of size at most 
$$
\frac{f^\Phi(\FF_p^d)}{H} \le \frac{|X|}{\varepsilon \gamma p} \ll_{K, d, \varepsilon} 1.
$$
Thus, we can take $p_0(d, \zeta)$ to be larger than $n_0(\varepsilon, \omega,  q)$ for all possible choices of parameters. Therefore, Lemma \ref{lm2} indeed applies. Thus, there is a constant $\mu \gg_{\varepsilon, \omega, q} 1$ and there are integer coefficients $\alpha_q$, for $q \in S$, such that:
\begin{align}
    \sum_{q\in S} \alpha_q =& p, ~~ \sum_{q\in S} \alpha_q q = p q_0,\\
    \mu p \le \alpha_q& \le (1+\varepsilon) \frac{p \omega(q)}{(\theta-\varepsilon) \omega(S)}.\label{alphaq}
\end{align}
Let $d' = \dim \Lambda_x$ and denote $\dim V_x = d' + t$. We can perform a change of basis on $V_x$ and identify it with $\FF_p^{d'} \times \FF_p^t$ in such a way that the map $\varphi_x: V_x \rightarrow \Lambda_x/p\Lambda_x$ is the projection onto the first $d'$ coordinates and the reduction of the basis $E_x$ of $\Lambda_x$ modulo $p$ gives the first $d'$ elements of the standard basis of $V_x$. Furthermore, by making a shift and replacing $K$ by $2K$, we may also assume that $q_0 = 0$, so that $\sum_{q \in S} \alpha_q q = 0$.
With this notation in mind, for each point $q \in S$ define $X_q \subset \FF_p^t$ as the multiset corresponding to the function $f_{\preceq x}$ restricted to the fiber $\{q\}\times \FF_p^t$. Let $X' = \bigcup_{q \in S} \{q\} \times X_q$, or equivalently, $X'$ is the multiset of the function $f_{\preceq x}$.

By unraveling the definitions, we have $|X_q| = \hat f_x(q)$ for every $q \in S$ and $|X'| = \hat f_x(S)$. Thus, by (\ref{alphaq}) we have
\begin{align*}
\alpha_q \le \frac{(1+\varepsilon) p}{\theta-\varepsilon} \frac{\omega(q)}{\omega(S)} \le (1+ 10^d \varepsilon) p \theta^{-1} \frac{\hat f_x(q)}{\hat f_x(S)}= (1+ 10^d \varepsilon) p \theta^{-1} \frac{|X_q|}{|X'|}.
\end{align*}
Note that the second inequality follows from the bounds $\theta \ge 4^{-d} \ge 2 \cdot 10^{-d}$ and $\varepsilon \le 10^{-d}$. 
By (\ref{sharp}), (\ref{theta}) and Proposition \ref{lbound} we get
\begin{align*}
    \alpha_q \le (1+ 10^d \varepsilon) p \frac{\w(\FF_p^d) |X_q|}{f^\Phi(\FF_p^d)} \le (1+ 10^d \varepsilon) \frac{\w(\FF_p^d)}{\w(\FF_p^d) +\zeta/2}  \le \frac{1+10^d\varepsilon}{1+4^{-d}\zeta/2} |X_q| \le (1-10^{-d}\zeta) |X_q|,
\end{align*}
where we use $\varepsilon = 100^{-d}\zeta$ and $\zeta<1$. 

Finally, since the flag decomposition $\Phi$ is $(T, \varepsilon, \delta)$-complete, the element $x$ is $(T, \delta)$-complete: for every linear function $\xi: V_x \rightarrow \FF_p$ which is not constant on $\{0\}\times \FF_p^t$, the function $f_{\preceq x}$ is $(T, \delta)$-thick along $\xi$. Since $f_{\preceq x}$ is the characteristic function of $X'$, the same condition holds for the set $X'$ as well. Let $\delta' = \min\{\mu, \varepsilon, \delta\}$ and observe that the collection of sets $X_q$ and coefficients $\alpha_q$, $q\in S$, satisfy the conditions of Theorem \ref{thexp}.
For the theorem to apply, we need to ensure that $T > T_0(d', t, K, \delta')$ and $p > p_0(d', t, K, \delta')$. Recall that Theorem \ref{fdl} implies $T > g(K)$, where the function $g$ can grow arbitrarily fast depending on the parameters $d$ and $\varepsilon$. In our situation, $\delta' \gg_{K, d, \varepsilon} 1$, and so the function $T_0(d', t, K, \delta')$ is bounded in terms of $K, d, \varepsilon$. Therefore, there exists a function $g = g_{d, \varepsilon}$ such that $g(K) > T_0(d', t, K, \delta')$. At the start of the proof, we pick the function $g$ so that this condition holds. 

Thus, all the necessary conditions of Theorem \ref{thexp} are satisfied. Hence, the set $X' \subset X$ contains $p$ distinct elements with zero sum. Theorem \ref{main} is proved.

\section*{Acknowledgments}

I thank Lisa Sauermann and Andrey Kupavskii for many helpful comments on earlier versions of the paper. I thank Jan-Christoph Schlage-Puchta for letting me know about his and Gautami Bhowmik's unpublished work on the problem after my paper became available. I thank Fedya Petrov for valuable discussions. I thank the anonymous referee for detailed comments. 

ChatGPT 5.5-Pro was used for low-level editing and proof-reading.

\begin{dajauthors}
\begin{authorinfo}[dz]
  Dmitrii Zakharov\\
  Department of Mathematics\\
  Massachusetts Institute of Technology\\
  Cambridge, MA 02139, USA\\
  zakhdm\imageat{}mit\imagedot{}edu
\end{authorinfo}
\end{dajauthors}

\end{document}